\documentclass[11pt,a4paper,reqno]{amsart}

\title[Pieri formula]{Pieri-type multiplication formula for quantum Grothendieck polynomials}
\author{Satoshi Naito and Daisuke Sagaki}
\address[Satoshi Naito]{Department of Mathematics, Tokyo Institute of Technology, 
2-12-1 Oh-okayama, Meguro-ku, Tokyo 152-8551, Japan.}
\email{naito.s.ac@m.titech.ac.jp}
\address[Daisuke Sagaki]{Department of Mathematics, 
Faculty of Pure and Applied Sciences, University of Tsukuba, 
1-1-1 Tennodai, Tsukuba, Ibaraki 305-8571, Japan.}
\email{sagaki@math.tsukuba.ac.jp}
\keywords{quantum Grothendieck polynomials, quantum $K$-theory, Pieri formula, quantum Bruhat graph. \newline
Mathematics Subject Classification 2020: 
Primary 05E05, 05E14; 
Secondary 14M15, 14N35, 14N15}

\usepackage[truedimen,margin=25truemm]{geometry} 
\usepackage{amsmath,amssymb,amsthm,amscd}

\allowdisplaybreaks
\numberwithin{equation}{section}

\newcommand{\BZ}{\mathbb{Z}}

\newcommand{\BA}{\mathbb{A}}
\newcommand{\BB}{\mathbb{B}}
\newcommand{\BD}{\mathbb{D}}

\newcommand{\CS}{\mathcal{S}}

\newcommand{\FG}{\mathfrak{G}}
\newcommand{\FS}{\mathfrak{S}}

\newcommand{\SL}{\mathsf{L}}
\newcommand{\SP}{\mathsf{P}}
\newcommand{\SM}{\mathsf{M}}

\newcommand{\RA}{\mathrm{A}}
\newcommand{\RB}{\mathrm{B}}
\newcommand{\RC}{\mathrm{C}}
\newcommand{\RD}{\mathrm{D}}
\newcommand{\RE}{\mathrm{E}}
\newcommand{\RF}{\mathrm{F}}
\newcommand{\RG}{\mathrm{G}}
\newcommand{\RR}{\mathrm{R}}
\newcommand{\RS}{\mathrm{S}}
\newcommand{\RT}{\mathrm{T}}
\newcommand{\RU}{\mathrm{U}}
\newcommand{\RX}{\mathrm{X}}
\newcommand{\RY}{\mathrm{Y}}

\newcommand{\CRA}{n_{(k-1,k)}=0}
\newcommand{\CRB}{n_{(k-1,k)}=1}
\newcommand{\CRBa}{ \begin{subarray}{l} n_{(k-1,k)}=1 \\ (k-1,k) \not\in M \end{subarray} }
\newcommand{\CRBb}{ \begin{subarray}{l} (k-1,k) \in M \\ \kappa = (k-1,k) \end{subarray} }
\newcommand{\CRBc}{ \begin{subarray}{l} (k-1,k) \in M \\ \kappa \ne (k-1,k) \end{subarray} }
\newcommand{\CRBca}{
\begin{subarray}{c} (k-1,k) \in M \\ 
\kappa \ne (k-1,k) \\ 
\kappa \notin M \\ \end{subarray}}
\newcommand{\CRBcb}{
\begin{subarray}{c} (k-1,k) \in M \\ 
\kappa \ne (k-1,k) \\ 
\kappa \in M \\ \end{subarray}}
\newcommand{\CRC}{n_{(*,k-1)} = 0}
\newcommand{\CRD}{n_{(*,k-1)} \ge 1}

\newcommand{\CRAa}{\bp_{(*,k)} = \emptyset}
\newcommand{\CRAb}{\bp_{(*,k)} \ne \emptyset,\,\kappa \not\in M}
\newcommand{\CRAc}{\bp_{(*,k)} \ne \emptyset,\,\kappa \in M}

\newcommand{\CRFa}{\kappa' \in M}
\newcommand{\CRFb}{\kappa' \not\in M}

\newcommand{\CRR}{n_{(k,*)}=0}
\newcommand{\CRS}{n_{(k,*)} \ge 1}
\newcommand{\CRSa}{(k,b(\bullet)) \in M}
\newcommand{\CRSb}{(k,b(\bullet)) \not\in M}
\newcommand{\CRSaba}{\kappa'' \in M}
\newcommand{\CRSabb}{\kappa'' \not\in M}

\newcommand{\CRX}{\iota = (k-1,k)}
\newcommand{\CRY}{\iota \ne (k-1,k)}

\newcommand{\SB}{\mathsf{B}}
\newcommand{\SQ}{\mathsf{Q}}

\newcommand{\bm}{\mathbf{m}}
\newcommand{\bp}{\mathbf{p}}
\newcommand{\bq}{\mathbf{q}}
\newcommand{\bs}{\mathbf{s}}

\newcommand{\vp}{\varphi}

\newcommand{\K}{\mathsf{K}}
\newcommand{\Q}{Q}
\newcommand{\g}{g}
\newcommand{\h}{h}

\newcommand{\hspm}[2]{\mathbb{P}\mathsf{M}^{#1}_{#2}(w)}
\newcommand{\spm}[2]{\mathsf{PM}^{#1}_{#2}(w)}
\newcommand{\hsp}[2]{\mathbb{P}^{#1}_{#2}(w)}
\newcommand{\sfp}[2]{\mathsf{P}^{#1}_{#2}(w)}

\newcommand{\hspmw}[2]{\mathbb{P}\mathsf{M}^{#1}_{#2}}
\newcommand{\spmw}[2]{\mathsf{PM}^{#1}_{#2}}
\newcommand{\hspw}[2]{\mathbb{P}^{#1}_{#2}}
\newcommand{\sfpw}[2]{\mathsf{P}^{#1}_{#2}}
\newcommand{\ip}{\sigma}
\newcommand{\jp}{\tau}

\newcommand{\ta}{2\mathrm{a}}
\newcommand{\tb}{2\mathrm{b}}

\newcommand{\G}[2]{G^{#1}_{#2}}
\newcommand{\bF}[3]{\mathbf{F}^{#1}_{#2}(#3)}
\newcommand{\bE}[3]{\mathbf{E}^{#1}_{#2}(#3)}
\newcommand{\CSF}[2]{\mathcal{S}^{#1}_{#2}}
\newcommand{\CSE}[2]{\mathcal{R}^{#1}_{#2}}

\DeclareMathOperator{\ed}{end}

\newcommand{\QBG}{\mathrm{QBG}}
\newcommand{\Mark}{\mathrm{Mark}}

\newcommand{\lng}{w_{\circ}}

\newcommand{\edge}[1]{ \xrightarrow{\hspace{2pt}#1\hspace{2pt}} }

\newcommand{\ha}[1]{\widehat{#1}}
\newcommand{\ol}[1]{\overline{#1}}

\theoremstyle{plain}
\newtheorem{lem}{Lemma}[section]
\newtheorem{prop}[lem]{Proposition}
\newtheorem{thm}[lem]{Theorem}
\newtheorem{cor}[lem]{Corollary}

\newtheorem{ithm}{Theorem}
\newtheorem{clm}{Claim}[subsection]

\theoremstyle{definition}
\newtheorem{dfn}[lem]{Definition}

\theoremstyle{remark}
\newtheorem{ex}[lem]{Example}
\newtheorem{rem}[lem]{Remark}

\newcommand{\bqed}{\quad \hbox{\rule[-0.5pt]{3pt}{8pt}}}

\newenvironment{enu}{%
 \begin{enumerate}%
}{\end{enumerate}}

\begin{document}


%
\begin{abstract}
The purpose of this paper is to prove a Pieri-type multiplication formula 
for quantum Grothendieck polynomials, which was conjectured by Lenart-Maeno. 
This formula would enable us to compute explicitly 
the quantum product of two arbitrary (opposite) Schubert classes in the (small) quantum $K$-theory ring $QK(Fl_{n})$ of the (full) flag manifold $Fl_{n}$ of type $A_{n-1}$ on the basis of the fact that quantum Grothendieck polynomials represent (opposite) Schubert classes in $QK(Fl_{n})$.
\end{abstract}

\maketitle

%
\section{Introduction.} 
\label{sec:intro}
In the seminal paper \cite{LM}, 
the authors defined and studied quantum Grothendieck polynomials, 
which are a common generalization of Grothendieck polynomials and quantum Schubert polynomials; 
Grothendieck polynomials, introduced in \cite{LaS3}, are 
polynomial representatives for (opposite) Schubert classes in the $K$-theory $K(Fl_{n})$ of 
the (full) flag manifold $Fl_{n}$ of type $A_{n-1}$, and quantum Schubert polynomials, 
introduced in \cite{FGP}, represent the corresponding (opposite) 
Schubert classes 
in the (small) quantum cohomology 
$QH^{*}(Fl_{n}) := H^{*}(Fl_{n}) \otimes \BZ[\Q_{1}, \dots, \Q_{n-1}]$. 
They defined quantum Grothendieck polynomials as the images of 
Grothendieck polynomials under a certain $K$-theoretic ``quantization map'', 
which is based on the (conjectural) presentation of the (small) quantum $K$-theory ring $QK(Fl_{n})$ 
(defined in \cite{Giv} and \cite{Lee}) of $Fl_{n}$ given by Kirillov-Maeno 
(see \cite[Theorem~6.1]{MNS1} for the modified presentation of the 
torus-equivariant version of $QK(Fl_{n})$
, for which the formal power series ring $\BZ[\![\Q]\!] := \BZ[\![\Q_{1}, \ldots, \Q_{n-1}]\!]$ is used instead of the polynomial ring $\BZ[\Q_{1}, \ldots, \Q_{n-1}]$ as a base ring), 
and furthermore 
obtained a Monk-type multiplication formula (\cite[Theorem~6.4]{LM}) 
for quantum Grothendieck polynomials, which is expressed in terms of 
directed paths in the quantum Bruhat graph on the infinite symmetric group. 
Also, they conjectured (\cite[Conjecture~7.1]{LM}) that 
quantum Grothendieck polynomials represent (opposite) Schubert classes 
in the quantum $K$-theory ring $QK(Fl_{n})$ under the (conjectural) presentation of $QK(Fl_{n})$ by Kirillov-Maeno. 

In the joint paper \cite{LNS} with C.~Lenart, based on the works \cite{Kat1} and \cite{Kat2}, 
we proved a (torus-equivariant version of) Monk-type multiplication formula (\cite[Theorem~49]{LNS}) for (opposite) Schubert classes in $QK(Fl_{n})$, 
which is exactly of the same form as the one (\cite[Theorem~6.4]{LM}) for quantum Grothendieck polynomials. 
Since the quantum multiplicative structure of $QK(Fl_{n})$ is completely determined 
by a Monk-type multiplication formula 
(if we use the formal power series ring $\BZ[\![\Q]\!]$ as a base ring), 
which describes the quantum product 
with divisor classes, it follows that the conjecture (\cite[Conjecture~7.1]{LM}) by Lenart-Maeno holds true, i.e., 
that quantum Grothendieck polynomials indeed represent 
(opposite) Schubert classes in $QK(Fl_{n})$ 
(for the precise statement and its proof, see \cite[Theorem~51]{LNS}); 
see also \cite[Theorem~4.4]{MNS2}, which states that quantum double 
Grothendieck polynomials represent (opposite) Schubert 
classes in the torus-equivariant version of $QK(Fl_{n})$. 

The purpose of this paper is to prove another conjecture (\cite[Conjecture~6.7]{LM}) 
presented by Lenart-Maeno, that is, a Pieri-type multiplication formula 
for quantum Grothendieck polynomials. This formula is much more complicated than the Monk-type multiplication formula, and is a vast generalization of it; 
by specializing the quantum parameters $\Q_{1},\Q_{2},\dots$ at zero, 
we recover the classical Pieri-type multiplication formula 
for Grothendieck polynomials, which was obtained in \cite{LS}. 
Let us explain our result more precisely. 
We set $\BZ[\Q] := \BZ[\Q_{1},\Q_{2},\dots]$, 
$\BZ[x]:=\BZ[x_{1},x_{2},\dots]$, and 
$\BZ[\Q, x] := \BZ[\Q] \otimes \BZ[x]$. 
Let $S_{\infty}$ denote the infinite symmetric group on $\BZ_{+}:=\{ 1,2,\dots,n,\dots \}$. 
For each $w \in S_{\infty}$, let $\FG_{w}^{\Q} \in \BZ[\Q, x]$ denote 
the quantum Grothendieck polynomial associated to $w$ (see Section~\ref{subsec:QGP}). 
Now, for $k \geq 1$ and $0 \leq p \leq k$, we set $G_{p}^{k} := \FG_{c[k, p]}^{\Q}$, 
where $c[k, p] \in S_{\infty}$ denotes the cyclic permutation $(k-p+1, k-p+2, \ldots, k, k+1)$. 
Also, for $k \ge 1$ and $w \in S_{\infty}$, 
let $\sfp{k}{}$ denote the set of all $k$-Pieri chains starting at $w$, 
where a $k$-Pieri chain is a directed path in the quantum Bruhat graph $\QBG(S_{\infty})$ on $S_{\infty}$ 
satisfying the conditions in Definition~\ref{dfn:pieri}. For $k \geq 1$ and $0 \leq p \leq k$, 
let $\sfp{k}{p}$ denote the subset of $\sfp{k}{}$ consisting of 
the elements having a $p$-marking, and let $\Mark_{p}(\bp)$ denote the set of 
$p$-markings of $\bp \in \sfp{k}{p}$; a $p$-marking of a $k$-Pieri chain $\bp$ is 
a subset of the set of labels in the directed path $\bp$ of cardinality $p$ 
satisfying the conditions in Definition~\ref{dfn:mark}. 
For a directed path $\bp$ in $\QBG(S_{\infty})$, let $\ell(\bp) \in \BZ_{\geq 0}$ be its length, $\ed(\bp) \in S_{\infty}$ its endpoint, both of which are defined just below equation~\eqref{eq:bp01}, and let $\Q(\bp) \in \BZ[\Q_{1},\Q_{2},\dots]$ be the monomial defined by equation~\eqref{eq:quantity}. 

Our main result can be stated as follows. 

\begin{ithm}[$=$ Theorem~\ref{thm:pieri}]
Let $k \geq 1$ and $0 \leq p \leq k$. 
For an arbitrary $w \in S_{\infty}$, the following equality holds in $\BZ[\Q, x]$\,{\rm:}
\begin{equation} \label{eq:maineq}
\FG^{\Q}_{w} \G{k}{p} 
= \sum_{\bp \in \sfp{k}{p}} (-1)^{\ell(\bp)-p} (\# \Mark_{p}(\bp)) \Q(\bp) \FG^{\Q}_{\ed(\bp)}. 
\end{equation}
\end{ithm}

Based on the fact mentioned above that quantum Grothendieck polynomials represent (opposite) Schubert classes in $QK(Fl_{n})$, as a consequence of this theorem, we can verify the positivity conjecture (\cite[Conjecture~7.5]{LM}) in a special case that $v$ is the cyclic permutation $(k-p+1, k-p+2, \ldots, k, k+1)$, with $0 \leq p \leq k$, $k \geq 1$; see Corollary~\ref{cor:pos}. 
Also, this theorem specializes to the classical Pieri formula (\cite[Theorem~1.12]{LS}) for the multiplication of Grothendieck polynomials by setting $Q_{1} = \Q_{2} = \cdots = 0$, and to the Pieri formula (\cite[Corollary~4.3]{P}) for the multiplication of quantum Schubert polynomials by taking the the lowest homogeneous components on both sides of equation~\eqref{eq:maineq}; for details, see Remark~\ref{rem:gen}. 

Our proof of the Pieri-type multiplication formula is essentially combinatorial, 
and relies only on some basic properties of the combinatorially 
defined quantum Grothendieck polynomials, which are given in \cite{LM}. 
However, we should mention the connection between this formula and the quantum $K$-theory ring $QK(Fl_{n})$. 
We know from \cite[Theorem~51]{LNS} that if we use the formal power series ring 
$\BZ[\![\Q]\!] = \BZ[\![\Q_{1}, \ldots, \Q_{n-1}]\!]$ instead of 
the polynomial ring $\BZ[\Q_{1}, \ldots, \Q_{n-1}]$ as a base ring, then the quantum 
$K$-theory ring $QK(Fl_{n}) := K(Fl_{n}) \otimes \BZ[\![\Q]\!]$ 
is presented as the quotient ring $(\BZ[\![\Q]\!][x_1, \ldots, x_n])/\ha{I}^{\Q}_{n}$, 
where the ideal $\ha{I}^{\Q}_{n}$ in $\BZ[\![\Q]\!][x_1, \ldots x_n]$ 
is generated by the polynomials $\ol{E}^{n}_{i}(x_1, \ldots, x_n)$, $1 \leq i \leq n$; 
the polynomial $\ol{E}^{n}_{i}(x_1, \ldots, x_n)$ is 
(the specialization at $Q_{n} = 0$ of) the image of 
the elementary symmetric polynomial $e^{n}_{i}(x_1, \ldots, x_n)$ of degree $i$ 
in the variables $x_1, \ldots, x_n$ under the $K$-theoretic quantization map 
(see \cite[\S 3]{LM} for details). 
Namely, we have the following isomorphism of $\BZ[\![\Q]\!]$-algebras: 
\begin{equation*}
QK(Fl_{n}) \cong 
(\BZ[\![\Q]\!][x_{1}, \ldots, x_{n}])/\ha{I}^{\Q}_{n}; 
\end{equation*}
the torus-equivariant version of this result is obtained in \cite[Theorem~6.1]{MNS1}. 
Also, it is known (see \cite[Remark~3.27]{LM}) that the residue classes of the polynomials 
$G_{p_{1}, \dots, p_{n-1}}(x_1, \ldots, x_{n-1}) := 
G^{1}_{p_1}(x_1) G^{2}_{p_2}(x_1, x_2) \cdots G^{n-1}_{p_{n-1}}(x_1, \ldots, x_{n-1})$ 
for $0 \leq p_{i} \leq i$, with $1 \leq i \leq n-1$, form a $\BZ[\![\Q]\!]$-basis of 
the quotient ring $(\BZ[\![\Q]\!][x_1, \ldots, x_n])/\ha{I}^{\Q}_{n} \cong QK(Fl_{n})$; 
note that the formal power series ring $\BZ[\![\Q]\!]$ contains 
the localized polynomial ring $\BZ[(1-Q_1)^{\pm 1}, \ldots, (1-Q_{n-1})^{\pm 1}]$. 
Hence the Pieri-type multiplication formula would enable us to 
compute explicitly the quantum product of 
two arbitrary (opposite) Schubert classes in $QK(Fl_{n})$ 
on the basis of the fact (proved in \cite[\S 6.1]{LNS}) that 
the (opposite) Schubert classes in $QK(Fl_{n})$, 
indexed by the elements of $S_{n}$, are represented by 
the corresponding quantum Grothendieck polynomials 
under the isomorphism above; the torus-equivariant version 
of this fact is proved in \cite[Theorem~4.4]{MNS2}. 
More precisely, to compute the product of two quantum Grothendieck polynomials 
in the quotient ring $(\BZ[\![\Q]\!][x_1, \ldots x_n])/\ha{I}^{\Q}_{n}$, 
we expand the product in the polynomial ring $\BZ[\![\Q]\!][x_1, \ldots x_n]$ 
in terms of the quantum Grothendieck polynomials, and then drop 
all terms containing quantum Grothendieck polynomials 
associated to $w \in S_{\infty}$ with $w \notin S_{n}$, 
as in the case of quantum Schubert polynomials (\cite[\S 10]{FGP}); 
for details, see the explanation preceding \cite[Theorem~51]{LNS}, and also the statement \cite[Proposition~B.7]{MNS1} in the torus-equivariant case. 

This paper is organized as follows. In Section~\ref{sec:pieri}, 
after fixing the basic notation for the quantum Bruhat graph on $S_{\infty}$, 
we recall from \cite{LM} some known facts about quantum Grothendieck polynomials, and 
then state our main result, that is, a Pieri-type multiplication formula 
for quantum Grothendieck polynomials. 
In Section~\ref{sec:main}, deferring the proofs of three key propositions 
(Propositions~\ref{prop:mat1a}, \ref{prop:mat2a}, and \ref{prop:mat3a}) to 
subsequent sections, we give a proof of the Pieri-type multiplication formula; 
the proofs of these three propositions are given in 
Sections~\ref{sec:prfmat1}, \ref{sec:prfmat2}, and \ref{sec:prfmat3}, respectively. 
In Appendices~\ref{sec:lemdp} and \ref{sec:ins}, 
we state and prove some technical results needed in 
Sections~\ref{sec:prfmat1}, \ref{sec:prfmat2}, and \ref{sec:prfmat3}. 
In Appendix~\ref{sec:example}, we give a few examples of the Pieri-type multiplication formula. 

\medskip

\paragraph{\bf Acknowledgments.}
We would like to thank Cristian Lenart for helpful discussions on \cite{LM}, and in particular on \cite[Conjecture~6.7]{LM}.
S.N. was partly supported by JSPS Grant-in-Aid for Scientific Research (C) 21K03198.
D.S. was partly supported by JSPS Grant-in-Aid for Scientific Research (C) 19K03415.

%
\section{Pieri formula.}
\label{sec:pieri}
%
%
\subsection{Basic notation.}
\label{subsec:not}

For $n \in \BZ_{\ge 1}$, let $S_{n}$ denote the symmetric group on $\{1,2,\dots,n\}$, 
with $T_{n} = \bigl\{ (a,b) \mid 1 \le a < b \le n \bigr\}$ 
the set of transpositions in $S_{n}$ and $\ell_{n}:S_{n} \rightarrow \BZ_{\ge 0}$ 
the length function on $S_{n}$. For each $n,m \in \BZ_{\ge 1}$ with $n \le m$, 
let $\rho_{m,n}:S_{n} \hookrightarrow S_{m}$ be the canonical embedding of groups defined by 
\begin{equation*}
(\rho_{m,n}(w))(a) := 
\begin{cases}
w(a) & \text{for $ 1 \le a \le n$}, \\
a & \text{for $n+1 \le a \le m$}
\end{cases}
\end{equation*} 
for $w \in S_{n}$. The infinite symmetric group $S_{\infty}$ is defined to be 
the inductive limit of $\{S_{n}\}_{n \ge 1}$ 
with respect to these embeddings, which can be regarded as the subgroup of 
the group of bijections on $\BZ_{+}:=\{ 1,2,\dots,n,\dots \}$ consisting of those elements $w$ such that 
$w(a)=a$ for all but finitely many $a \in \BZ_{+}$. For each $n \in \BZ_{\ge 1}$, 
let $\rho_{n}:S_{n} \hookrightarrow S_{\infty}$ be the canonical embedding, by which 
we regard $S_{n}$ as a subgroup of $S_{\infty}$. We denote by $T_{\infty} = 
\bigl\{ (a,b) \mid \text{$a,b \in \BZ_{+}$ with $a < b$} \bigr\} \ (=\bigcup_{n=1}^{\infty} T_{n})$ 
the set of transpositions in $S_{\infty}$, and by $\ell_{\infty}:S_{\infty} \rightarrow \BZ_{\ge 0}$ 
the length function on $S_{\infty}$; note that $\ell_{\infty}(w) = \ell_{n}(w)$ 
for all $w \in S_{n} \hookrightarrow S_{\infty}$. 

\begin{dfn}[{cf. \cite[Definition~6.1]{BFP}}]
The quantum Bruhat graph $\QBG(S_{\infty})$ on $S_{\infty}$ 
is the $T_{\infty}$-labeled directed graph whose vertices are the elements of $S_{\infty}$ 
and whose (directed) edges are of the form: 
$x \xrightarrow{(a,b)} y$, with $x, y \in S_{\infty}$ and $(a,b) \in T_{\infty}$, 
such that $y = x \cdot (a,b)$ and either of the following holds: 
(B) $\ell_{\infty}(y) = \ell_{\infty}(x) + 1$, or (Q) $\ell_{\infty}(y) = \ell_{\infty}(x) - 2(b-a) + 1$. 
An edge satisfying (B) (resp., (Q)) is called a Bruhat edge (resp., a quantum edge).
\end{dfn}

For $m_1,\,m_2 \in \BZ$, we set $[m_1,m_2]:=\bigl\{ m \in \BZ \mid m_1 \le m \le m_2 \bigr\}$. 
We know the following lemma from \cite[Proposition~3.6]{Len}.
%
%
\begin{lem} \label{lem:edge}
Let $x \in S_{\infty}$, and $a,b \in \BZ_{+}$ with $a < b$. 
\begin{enu}
\item[\rm (B)] 
We have a Bruhat edge $x \edge{(a,b)} x \cdot (a,b)$ in $\QBG(S_{\infty})$ if and only if 
$x(a) < x(b)$ and $x(c) \notin [x(a),x(b)]$ for any $a < c < b$. 

\item[\rm (Q)]
We have a quantum edge $x \edge{(a,b)} x \cdot (a,b)$ in $\QBG(S_{\infty})$ if and only if 
$x(a) > x(b)$ and $x(c) \in [x(b),x(a)]$ for all $a < c < b$. 
\end{enu}
\end{lem}

For simplicity of notation, 
we write a directed path 
\begin{equation} \label{eq:bp00}
\bp:w=x_{0} \edge{(a_{1},b_{1})} x_{1} \edge{(a_{2},b_{2})} \cdots 
\edge{(a_{r},b_{r})} x_{r}
\end{equation}
in the quantum Bruhat graph $\QBG(S_{\infty})$ as:
\begin{equation} \label{eq:bp01}
\bp=(w\,;\,(a_{1},b_{1}),\dots,(a_{r},b_{r})); 
\end{equation}
when $r=0$, we define $\bp$ as $\bp=(w\,;\,\emptyset)=\emptyset$. 
We define $\ell(\bp):=r$ and $\ed(\bp):=x_{r}$. 
We call a (consecutive) subsequence $\bs$ of labels in $\bp$ of the form: 
\begin{equation} \label{eq:bs01}
(a_{s+1},b_{s+1}),(a_{s+2},b_{s+2}), \dots,(a_{t-1},b_{t-1}), (a_{t},b_{t}), 
\end{equation}
with $0 \le s \le t \le r$, a segment in $\bp$; 
if $s=t$, then the segment $\bs$ is understood to be empty, 
and we write it as $\emptyset$. 
We can also think of the directed path $\bp$ as a segment of $\bp$ 
corresponding to the special case that $s=0$ and $t=r$; thus it is natural to 
define the length $\ell(\bs)$ of the segment $\bs$ above by $\ell(\bs):=t-s$. 
Using the segment $\bs$ of the form \eqref{eq:bs01}, we can write $\bp$ in \eqref{eq:bp01} as:
\begin{equation*}
\bp=(w\,;\,(a_{1},b_{1}),\dots,(a_{s},b_{s}),\bs,(a_{t+1},b_{t+1}),\dots,(a_{r},b_{r})). 
\end{equation*}
When $\bp$ and $\bs$ are of the forms \eqref{eq:bp01} and \eqref{eq:bs01}, 
respectively, we set
\begin{align*}
n_{(a,*)}(\bs) & := \# \bigl\{ s+1 \le u \le t \mid a_{u}=a \bigr\}, \\
n_{(*,b)}(\bs) & := \# \bigl\{ s+1 \le u \le t \mid b_{u}=b \bigr\}, \\
n_{(a,b)}(\bs) & := \# \bigl\{ s+1 \le u \le t \mid (a_{u},b_{u})=(a,b) \bigr\}. 
\end{align*}
If $s < t$, then we set $\iota(\bs):=(a_{s+1},b_{s+1})$ and $\kappa(\bs):=(a_{t},b_{t})$, 
and call them the initial label and the final label of $\bs$, respectively; 
if $s=t$, i.e., $\bs=\emptyset$, then $\iota(\bs)$ and $\kappa(\bs)$ are undefined. 
If all the labels in a segment $\bs$ are distinct 
(almost all directed paths in this paper satisfy this condition; 
see Definitions~\ref{dfn:pieri} and \ref{dfn:monk} below), 
we identify $\bs$ with the set of labels in $\bs$. 

We can show the following lemma, 
which plays a crucial role in our proof of Theorem~\ref{thm:pieri} below,
by exactly the same argument as for \cite[Lemma~2.7]{LS}
(see also \cite{BFP} and \cite[Theorem~7.3]{LNSSS1}). 
%
%
\begin{lem} \label{lem:int}
Let $v \in S_{\infty}$, and $a,b,c,d \in \BZ_{+}$. 
\begin{enu}
\item Assume that $a < b$, $c < d$, and $\{a,b\} \cap \{c,d\} = \emptyset$. 
If $(v\,;\,(a,b),(c,d))$ is a directed path, then so is $(v\,;\,(c,d),(a,b))$. 

\item Assume that $a < b < c$. 
If $(v\,;\,(a,c),(b,c))$ is a directed path, then so is 
$(v\,;\,(b,c), (a,b))$. Also, if $(v\,;\,(b,c),(a,c))$ is a directed path, 
then so is $(v\,;\,(a,b), (b,c))$. 

\item Assume that $a < b < c$. 
If $(v\,;\,(a,b),(a,c))$ is a directed path, then so is $(v;(b,c),(a,b))$. 
Also, if $(v\,;\,(a,c),(a,b))$ is a directed path, then so is $(v\,;\,(a,b),(b,c))$.

\item Assume that $a < b < c$. 
If $(v\,;\,(a,b),(b,c))$ is a directed path, then 
either $(v\,;\,(b,c), (a,c))$ or $(v\,;\,(a,c), (a,b))$ is a directed path. 
Also, if $(v\,;\,(b,c), (a,b))$ is a directed path in the quantum Bruhat graph, then 
either $(v\,;\,(a,c), (b,c))$ or $(v\,;\,(a,b),(a,c))$ is a directed path. 
\end{enu}
\end{lem}

Now, let $w \in S_{\infty}$. Let $k \ge 2$, and 
let $\bp$ be a directed path in $\QBG(S_{\infty})$ of the form: 
\begin{equation*}
\bp=(w\,;\,\dots\dots,
\underbrace{(j_{1},k),(j_{2},k),\dots,(j_{t},k)}_{=:\,\bs}),
\end{equation*}
with $t \ge 0$. Let $d \ge k+1$ be such that 
\begin{equation} \label{eq:Alg01a}
(w\,;\,\dots\dots,\underbrace{(j_{1},k),(j_{2},k),\dots,(j_{t},k)}_{=\,\bs},(k,d))
\end{equation}
is also a directed path in $\QBG(S_{\infty})$. We present an algorithm, which we call 
{\bf Algorithm $(\bs : (k,d))$}, as follows; 
we will use this algorithm in order to construct certain bijections 
in Propositions~\ref{prop:mat1}, ~\ref{prop:mat2}, and \ref{prop:mat3}. 

\begin{itemize}
\item[(i)] Start with the directed path \eqref{eq:Alg01a}. 
\item[(ii)] Assume that we have a directed path of the form:
\begin{equation*}
(w\,;\,\dots\dots,
  \underbrace{(j_{1},k),\dots,(j_{u},k)}_{%
  \text{omitted if $u=0$} },(k,d),
  \underbrace{(j_{u+1},d),\dots,(j_{t},d)}_{%
  \text{omitted if $u=t$} })
\end{equation*}
for some $0 \le u \le t$. If $u=0$, then end the algorithm. 
If $u > 0$, then we see from Lemma~\ref{lem:int}\,(4), 
applied to the segment $(j_{u},k),(k,d)$, 
that either of the following occurs: 
(iia)~we have a directed path of the form: 
\begin{equation*}
(w\,;\,\dots\dots,
  (j_{1},k),\dots,(j_{u-1},k),(k,d),(j_{u},d),
  (j_{u+1},d),\dots,(j_{t},d)), 
\end{equation*}
or (iib)~we have a directed path of the form: 
\begin{equation*}
(w\,;\,\dots\dots,
  (j_{1},k),\dots,(j_{u-1},k),(j_{u},d),(j_{u},k),
  (j_{u+1},d),\dots,(j_{t},d)). 
\end{equation*}
If (iib) occurs, then end the algorithm. 
If (iia) occurs, then go back to the beginning of (ii), 
with $u$ replaced by $u-1$. 
\end{itemize}

\begin{ex} \label{ex:algo}
In examples in this paper and in Appendix~\ref{sec:example}, 
we use one-line notation for elements in $S_{\infty}$. 
Namely, the symbol $a_1 a_2 \cdots a_n$ denotes the element $w \in S_{\infty}$
such that $w(i) = a_{i}$ for $1 \le i \le n$ and $w(j)=j$ for $j \ge n+1$. 
Also, for a label $(a,b)$ of a directed path in $\QBG(S_{\infty})$, 
we write $(a,b)_{\SB}$ (resp., $(a,b)_{\SQ}$) 
if the edge labeled by $(a,b)$ is a Bruhat (resp., quantum) edge. 
\begin{enu}
\item Assume that $k = 5$. Let us run {\bf Algorithm $(\bs : (5,7))$}, with
\begin{equation*}
(7465321; \underbrace{(1,5)_{\SQ},(3,5)_{\SB}, (4,5)_{\SB}}_{= \, \bs},(5,7)_{\SQ}). 
\end{equation*}
First, we apply Lemma~\ref{lem:int}\,(4) to the segment 
$(4,5)_{\SB},(5,7)_{\SQ}$; we see that (iia) occurs since 
\begin{equation*}
(7465321; (1,5)_{\SQ},(3,5)_{\SB}, (5,7)_{\SQ},(5,7)_{\SB})
\end{equation*}
is a directed path. As the next step, 
we apply Lemma~\ref{lem:int}\,(4) 
to the segment $(3,5)_{\SB}, (5,7)_{\SQ}$; 
we see that (iia) occurs since 
\begin{equation*}
(7465321; (1,5)_{\SQ},(5,7)_{\SQ},(3,7)_{\SB}, (5,7)_{\SB})
\end{equation*}
is a directed path. Also, 
we apply Lemma~\ref{lem:int}\,(4) to 
the segment $(1,5)_{\SQ},(5,7)_{\SQ}$; we see that 
(iib) occurs since 
\begin{equation*}
(7465321; (1,7)_{\SQ},(1,5)_{\SB},(3,7)_{\SB}, (5,7)_{\SB})
\end{equation*}
is a directed path; we end the algorithm here. 

\item Assume that $k = 5$. Let us run {\bf Algorithm $(\bs : (5,7))$}, with
\begin{equation*}
(32615874; \underbrace{(1,5)_{\SB},(4,5)_{\SB}}_{= \, \bs},(5,7)_{\SB}). 
\end{equation*}
We see that 
\begin{align*}
& (32615874; (1,5)_{\SB},(5,7)_{\SB},(4,7)_{\SB}), \\
& (32615874; (5,7)_{\SB},(1,7)_{\SB},(4,7)_{\SB})
\end{align*}
are directed paths, which implies that (iia) occurs in each step 
of the algorithm. Since $(5,7)$ has been moved to the left of 
the segment $(1,7),(4,7)$, which is obtained from $\bs$ 
by replacing $(j,5)$ by $(j,7)$, and so ``$u=0$'' in this case, 
we end the algorithm. 
\end{enu}
\end{ex}

%
\subsection{Quantum Grothendieck polynomials.}
\label{subsec:QGP}

For $n \in \BZ_{\ge 1}$, we set 
\begin{equation*}
\K_{n} := \BZ[\Q_{1},\Q_{2},\dots,\Q_{n-1}] \otimes_{\BZ} \BZ[x_{1},x_{2},\dots,x_{n}].
\end{equation*}
Also, we set
\begin{align*}
\K_{\infty} & := \BZ[\Q_{1},\Q_{2},\dots] \otimes_{\BZ} \BZ[x_{1},x_{2},\dots], \\
\K_{\infty}' & := \BZ[(1-\Q_{1})^{\pm 1},(1-\Q_{2})^{\pm 1},\dots] \otimes_{\BZ} \K_{n} \ (\supset \K_{n}). 
\end{align*}
%
%
Let $\FG_{w}^{\Q} \in \K_{n}$, $w \in S_{n}$, be the quantum Grothendieck polynomials
defined in \cite[Definition~3.18]{LM}. We know the following stability property from \cite[Proposition~3.20]{LM}. 

\begin{prop} \label{prop:stability}
Let $n,m \in \BZ_{\ge 1}$ with $n \le m$. Then, 
$\FG_{\rho_{m,n}(w)}^{\Q} \in \K_{m}$ is identical to $\FG_{w}^{\Q} \in \K_{n} \subset \K_{m}$ 
for all $w \in S_{n}$. 
\end{prop}

By Proposition~\ref{prop:stability}, 
we obtain a family $\{ \FG_{w}^{\Q} \}_{w \in S_{\infty}}$ 
of polynomials in $\K_{\infty}$. 
For $1 \leq p \leq k$, we set 
%
%
\begin{equation} \label{eq:ckp}
\G{k}{p}:=\FG_{(k-p+1,k-p+2,\dots,k,k+1)}^{\Q}, 
\end{equation}
where $(k-p+1,k-p+2,\dots,k,k+1) \in S_{\infty}$ is the cyclic permutation. 
By convention, we set $\G{k}{0}:=1$ for all $k \ge 1$, and 
$\G{k}{p}:=0$ unless $k \ge 1$ and $0 \le p \le k$. 
%
%
\begin{prop} \label{prop:rec}
Let $k \ge 2$ and $1 \le p \le k$. 
The following equality holds in $\K_{\infty}'$\,{\rm:}
\begin{equation} \label{eq:rec}
\begin{split}
\G{k}{p} - \G{k-1}{p-1}  = \, & 
  (1-\Q_{k})(1-x_{k})(1-\Q_{k-1})^{-1} \times \\
& \big\{ ( \G{k-1}{p} - \Q_{k-1} \G{k-2}{p-1} ) -  
         ( \G{k-1}{p-1} - \Q_{k-1} \G{k-2}{p-2} ) \bigr\}. 
\end{split}
\end{equation}
\end{prop}

\begin{proof}
By \cite[(3.30) and (3.32)]{LM}, we see that 
$\ol{G}^{k}_{p} =\G{k}{p}+\Q_{k}(1-\Q_{k})^{-1}(\G{k}{p}-\G{k-1}{p-1})$ in $\K_{\infty}'$, 
where $\ol{G}^{k}_{p} := \G{k}{p}|_{\Q_{k}=0}$. 
Hence we have 
$\ol{G}^{k-1}_{p} =\G{k-1}{p}+\Q_{k-1}(1-\Q_{k-1})^{-1}(\G{k-1}{p}-\G{k-2}{p-1})$ and 
$\ol{G}^{k-1}_{p-1} =\G{k-1}{p-1}+\Q_{k-1}(1-\Q_{k-1})^{-1}(\G{k-1}{p-1}-\G{k-2}{p-2})$, 
where $\ol{G}^{k-1}_{p} := \G{k-1}{p}|_{\Q_{k-1}=0}$, $\ol{G}^{k-1}_{p-1} := \G{k-1}{p-1}|_{\Q_{k-1}=0}$. 
Substituting these equalities into \cite[(3.32)]{LM}, we obtain \eqref{eq:rec}, as desired. This proves the proposition. 
\end{proof}

For a directed path $\bp$ in $\QBG(S_{\infty})$ of the form \eqref{eq:bp00}, 
we define a monomial $\Q(\bp)$ by
\begin{equation} \label{eq:quantity}
\Q(\bp):=\prod_{
 \begin{subarray}{c}
 1 \le s \le r \\[1mm]
 \text{$x_{s-1} \edge{(a_{s},b_{s})} x_{s}$ is} \\[1mm]
 \text{a quantum edge}
 \end{subarray}} (Q_{a_{s}}Q_{a_{s}+1} \cdots Q_{b_{s}-1}) \in \BZ[\Q_{1},\Q_{2},\dots]. 
\end{equation}

%
\subsection{Monk-type multiplication formula.}
\label{subsec:monk}
%
%
\begin{dfn} \label{dfn:monk}
Let $v \in S_{\infty}$, and $k \ge 1$. A directed path 
\begin{equation*}
\bm = (v\,;\,
  \underbrace{ (a_{1},k),(a_{2},k),\dots,(a_{s},k),}_{%
    \begin{subarray}{c}
    \text{This segment is called } \\[1mm]
    \text{the $(*,k)$-segment of $\bm$,} \\[1mm]
    \text{and denoted by $\bm_{(*,k)}$.}
    \end{subarray}%
    }
  \underbrace{ (k,b_{t}),(k,b_{t-1}),\dots,(k,b_{1}) }_{%
    \begin{subarray}{c}
    \text{This segment is called } \\[1mm]
    \text{the $(k,*)$-segment of $\bm$,} \\[1mm]
    \text{and denoted by $\bm_{(k,*)}$.}
    \end{subarray}})
\end{equation*}
in $\QBG(S_{\infty})$
satisfying the conditions that 
$s \ge 0$ and $k > a_{1} > a_{2} > \cdots > a_{s} \ge 1$, and that 
$t \ge 0$ and $k < b_{1} < b_{2} < \cdots < b_{t}$, is called 
a $k$-Monk chain starting at $x$. 
\end{dfn}
Let $\SM_{k}(v)$ denote 
the set of all $k$-Monk chains starting at $v$. 
We know the following formula from 
\cite[Theorem~6.1]{LM}. 
%
%
\begin{prop} \label{prop:LM61}
For $x \in S_{\infty}$ and $k \ge 1$, 
the following holds in $\K_{\infty}$\,{\rm:} 
\begin{equation} \label{eq:LM61}
(1-\Q_{k})(1-x_{k})\FG^{\Q}_{x} = 
\sum_{ \bm \in \SM_{k}(x) }
(-1)^{\ell(\bm_{(k,*)})} \Q(\bm) \FG^{\Q}_{ \ed(\bm) }. 
\end{equation}
\end{prop}

%
\subsection{Main result: Pieri-type multiplication formula.}
\label{subsec:pieri}

We define a total order $\preceq$ on the set $T_{\infty}=\bigl\{ (a,b) \mid 
\text{$a,b \in \BZ_{+}$ with $a < b$} \bigr\}$ of transpositions in $S_{\infty}$ by
\begin{equation}
(a,b) \prec (c,d) \stackrel{\mathrm{def}}{\iff} 
\text{($b > d$) or ($b=d$ and $a < c$)}. 
\end{equation}

For each $k \ge 1$, we set
$\SL_{k}:=\bigl\{ (a,b) \in T_{\infty} \mid a \le k < b \bigr\}$.
%
%
\begin{dfn} \label{dfn:pieri}
Let $w \in S_{\infty}$ and $k \ge 1$. 
A directed path 
\begin{equation*}
\bp= (w\,;\,(a_{1},b_{1}),\dots,(a_{r},b_{r}))
\end{equation*}
in $\QBG(S_{\infty})$ is called a $k$-Pieri chain 
if it satisfies the following conditions: 
\begin{enu}
\item[(P0)] $(a_{s},b_{s}) \in \SL_{k}$ for all $1 \le s \le r$, and 
$n_{(a,b)}(\bp) \in \{0,1\}$ for each $(a,b) \in \SL_{k}$; 
\item[(P1)] $b_{1} \ge b_{2} \ge \cdots \ge b_{r}$; 
\item[(P2)] If $r \ge 3$, and if $a_{t}=a_{s}$ 
for some $1 \le t < s \le r-1$, 
then $(a_{s},b_{s}) \prec (a_{s+1},b_{s+1})$. 
\end{enu}
\end{dfn}

Let $\sfpw{k}{}=\sfp{k}{}$ denote the set of 
all $k$-Pieri chains starting at $w \in S_{\infty}$.

\begin{ex}[see also Appendix~\ref{sec:example}] \label{ex:pieri}
Recall the notation from Example~\ref{ex:algo}. 
We can find some examples of $k$-Pieri chains in \cite{LS}, whose edges are all Bruhat edges. 
For example, 
\begin{align*}
& \bp_{1}:=(w = 4261735; (1,7)_{\SB}, (2,6)_{\SB}, (4,6)_{\SB}, (3,5)_{\SB}, (2,5)_{\SB}) \in \sfp{4}{}, \\
& \bp_{2}:=(w = 215436; (3,6)_{\SB}, (1,5)_{\SB}, (2,5)_{\SB}, (1,4)_{\SB}) \in \sfp{3}{}, \\
& \bp_{3}:=(w = 52173846; (1,8)_{\SB}, (5,7)_{\SB}, (2,7)_{\SB}, (3,7)_{\SB}, (4,6)_{\SB}, (1,6)_{\SB}, (5,6)_{\SB}) \in \sfp{5}{}. 
\end{align*}
Also, we have the following $4$-Pieri chain, with $w = 7465321$: 
\begin{equation*}
\bp_{4}:= (w = 7465321; (4,7)_{\SQ}, (2,7)_{\SB}, (5,7)_{\SB}, (4,6)_{\SB}) \in \sfp{4}{}. 
\end{equation*}
However, the directed path
\begin{equation*}
(7465321; (4,7)_{\SQ}, (2,7)_{\SB}, (5,7)_{\SB}, (4,6)_{\SB}, (3,6)_{\SQ})
\end{equation*}
is not an element of $\sfp{4}{}$ 
since it does not satisfy condition (P2); notice that $(4,6) \succ (3,6)$. 
\end{ex}

Let $\bp \in \sfp{k}{}$. We see by (P1) in Definition~\ref{dfn:pieri} that
for each $m \ge k+1$, there exists a unique longest (possibly, empty) segment 
in $\bp$ in which all labels are contained in $\{ (a,m) \mid 1 \le a \le k\}$. 
We call this segment the $(*,m)$-segment of $\bp$, and denote it by $\bp_{(*,m)}$; 
we can write $\bp$ as: 
\begin{equation*}
\bp=(w\,;\,\dots,\bp_{(*,m+1)},\bp_{(*,m)},\bp_{(*,m-1)},\dots,\bp_{(*,k+1)}).
\end{equation*}
Also, if a label $(a,m)$ appears in $\bp_{(*,m)}$, 
then we denote by $\bp_{(*,m)}^{(a,m)}$ the segment in $\bp_{(*,m)}$ 
consisting of all labels appearing after the label $(a,m)$. 
%
%
\begin{dfn} \label{dfn:mark}
Let $w \in S_{\infty}$, and $k \ge 1$, $0 \le p \le k$. 
Let $\bp= (w\,;\,(a_{1},b_{1}),\dots,(a_{r},b_{r})) \in \sfp{k}{}$; 
recall that all the labels in $\bp$ are distinct 
(see (P0) in Definition~\ref{dfn:pieri}). 
A subset $M$ of the set $\{(a_{s},b_{s}) \mid 1 \le s \le r\}$ of labels in $\bp$, 
with $\#M=p$, is called a $p$-marking of $\bp$ if it satisfies the following conditions:
\begin{enu}
\item if $(a_{s},b_{s}) \in M$, then $a_{u} \ne a_{s}$ for all $1 \le u < s$; 
\item if $(a_{s},b_{s}) \notin M$ and $s < r$, 
then $(a_{s},b_{s}) \prec (a_{s+1},b_{s+1})$; 
\item if $b_{1} = b_{2} = \cdots = b_{t}$ and $a_{1} > a_{2} > \cdots > a_{t}$
for some $t \ge 1$, then $(a_{t},b_{t}) \in M$. 
\end{enu}
\end{dfn}

Let $\Mark_{p}(\bp)$ denote the set of $p$-markings of $\bp$, 
and denote by $\sfpw{k}{p}=\sfp{k}{p}$ the subset of $\sfpw{k}{}=\sfp{k}{}$ 
consisting of all elements having $p$-markings. We set
\begin{equation}
\hspw{k}{p} = 
\hspw{k}{p}(w):= 
 \bigl\{ (\bp,M) \mid \bp \in \sfp{k}{p},\,M \in \Mark_{p}(\bp) \bigr\}.
\end{equation}

\begin{ex}[see also Appendix~\ref{sec:example}]
Recall the notation from Example~\ref{ex:algo}. 
Let $\bp_{1}$ be as in Example~\ref{ex:pieri}. 
Let us compute $\Mark_{p}(\bp_{1})$ for $0 \le p \le k = 4$. 
By condition (3), the label $(1,7)$ must be contained in all $p$-markings. 
Since $(3,5) \succ (2,5)$, it follows from condition (2) that 
the label $(3,5)$ must be contained in all $p$-markings. 
Hence we have $\Mark_{p}(\bp_{1}) = \emptyset$ for $p=0,1$. 
By condition (1), the label $(2,5)$ cannot be contained in any $p$-marking. 
We deduce (cf. the example after \cite[Definition 1.10]{LS}) that
\begin{align*}
& \Mark_{2}(\bp_{1}) = \bigl\{ \{(1,7), (3,5)\} \bigr\}, \quad
  \Mark_{3}(\bp_{1}) = \bigl\{ \{(1,7), (2,6), (3,5)\},\,\{(1,7), (4,6), (3,5)\} \bigr\},\\
& \Mark_{4}(\bp_{1}) = \bigl\{ \{(1,7), (2,6), (4,6), (3,5)\} \bigr\}. 
\end{align*}

Next, let $\bp_{4}$ be as in Example~\ref{ex:pieri}. 
Let us compute $\Mark_{p}(\bp_{4})$ for $0 \le p \le k = 4$. 
By condition (3), the labels $(4,7)$ and $(2,7)$ must be contained in all $p$-markings;  
in particular, we have $\Mark_{p}(\bp_{4}) = \emptyset$ for $p=0,1$. 
By condition (1), the label $(4,6)$ cannot be contained in any $p$-marking. 
We deduce that
\begin{align*}
& \Mark_{2}(\bp_{4}) = \bigl\{ \{(4,7), (2,7)\} \bigr\}, \quad
  \Mark_{3}(\bp_{4}) = \bigl\{ \{(4,7), (2,7), (5,7)\}  \bigr\}, \quad
  \Mark_{4}(\bp_{4}) = \emptyset. 
\end{align*}
\end{ex}

The following is the main result of this paper, 
which implies \cite[Conjecture~6.7]{LM}. 
%
%
\begin{thm} \label{thm:pieri}
Let $k \ge 1$ and $0 \le p \le k$. 
For an arbitrary $w \in S_{n}$, the following equalities hold in $\K_{\infty}$\,{\rm:}
\begin{equation} \label{eq:pieri}
\begin{split}
\FG^{\Q}_{w} \G{k}{p} 
& = \sum_{(\bp,M) \in \hsp{k}{p}} (-1)^{\ell(\bp)-p} \Q(\bp) \FG^{\Q}_{\ed(\bp)} \\[3mm]
& = \sum_{\bp \in \sfp{k}{p}} (-1)^{\ell(\bp)-p} (\# \Mark_{p}(\bp)) \Q(\bp) \FG^{\Q}_{\ed(\bp)}. 
\end{split}
\end{equation}
\end{thm}

For a few examples, see Appendix~\ref{sec:example}.

%
\begin{rem} \label{rem:mark}
Keep the setting of Theorem~\ref{thm:pieri}. 
For $\bp= (w\,;\,(a_{1},b_{1}),\dots,(a_{r},b_{r})) \in \sfp{k}{p}$, 
we set $m_{0}(\bp):=\# \bigl\{ 1 \le a \le k \mid n_{(a,*)}(\bp) \ge 1 \bigr\}$. 
It follows from condition (1) in Definition~\ref{dfn:mark} that 
$p \le m_{0}(\bp)$. Also, if we set
\begin{equation*}
\begin{split}
M(\bp):= \, & \bigl\{t \ge 1 \mid 
\text{$b_{1} = b_{2} = \cdots = b_{t}$ and $a_{1} > a_{2} > \cdots > a_{t}$} \bigr\} \\
& \cup \bigl\{1 \le s \le r-1 \mid (a_{s},b_{s}) \succ (a_{s+1},b_{s+1}) \bigr\}, 
\end{split}
\end{equation*}
and $m(\bp):=\# M(\bp)$, then by conditions (2) and (3) in  Definition~\ref{dfn:mark}, 
we see that $M(\bp) \subset M$ for all $M \in \Mark_{p}(\bp)$. 
In addition, we have
\begin{equation*}
\# \Mark_{p}(\bp) = 
\binom{ m_{0}(\bp) - m(\bp) }{ p - m(\bp) }. 
\end{equation*}
\end{rem}

\begin{rem} \label{rem:gen} \mbox{}
\begin{enu}
\item In the special case that $p=1$, formula \eqref{eq:pieri} in Theorem~\ref{thm:pieri} is just \cite[Theorem~6.4]{LM}; 
indeed, the starting point of our inductive argument 
for the proof of Theorem~\ref{thm:pieri} 
is established by \cite[Theorem~6.4]{LM}. 
%
%
\item If we set $\Q_{1} = \Q_{2} = \cdots = 0$ in Theorem~\ref{thm:pieri}, then 
we obtain the classical Pieri formula (\cite[Theorem~1.12]{LS}) for the multiplication of Grothendieck polynomials. 
This assertion follows immediately 
from the fact that for $\bp \in \sfp{k}{p}$, $\Q(\bp) \ne 1$ 
if and only if 
$\bp$ has a quantum edge, and the fact that the specialization of 
the quantum Grothendieck polynomial $\FG^{\Q}_{w}$ at $\Q_{1} = Q_{2} = \cdots = 0$ 
is identical to the Grothendieck polynomial $\FG_{w}$ (see \cite[Proposition~3.22\,(2)]{LM}). 

\item We know from \cite[Proposition~3.22\,(2)]{LM} that 
the lowest homogeneous component of 
the quantum Grothendieck polynomial $\FG^{\Q}_{w}$ is identical to 
the quantum Schubert polynomial $\FS^{\Q}_{w}$. 
By taking the lowest homogeneous components 
on both sides of equation \eqref{eq:pieri}, we obtain the Pieri formula 
(\cite[Corollary~4.3]{P}) for the multiplication of quantum Schubert polynomials. 
\end{enu}
\end{rem}

Now, let 
\begin{equation*}
\BD:= \bigl\{ d=(d_{1},d_{2},\dots) \in \BZ^{\infty} \mid 
\text{$d_{i} = 0$ for all but finitely many $i \ge 1$} \bigr\}, 
\end{equation*}
and set $\Q^{d}:=\prod_{i \ge 1} Q_{i}^{d_{i}}$ for 
$d=(d_{1},d_{2},\dots) \in \BD$. Let $w \in S_{\infty}$, and 
$v := (k-p+1, k-p+2, \dots, k,k+1) \in S_{\infty}$; 
recall from \eqref{eq:ckp} that $\G{k}{p} = \FG^{\Q}_{v}$. 
We define $N_{w,v}^{u}(d) \in \BZ$ for $u \in S_{\infty}$ and $d \in \BD$ by: 
\begin{equation}
\FG^{\Q}_{w} \G{k}{p} = \sum_{d \in \BD} Q^{d} \sum_{u \in S_{\infty}}
N_{w,v}^{u}(d) \FG^{\Q}_{u}. 
\end{equation}
We obtain the following corollary of Theorem~\ref{thm:pieri}, 
which implies that the positivity conjecture (\cite[Conjecture~7.5]{LM}) holds true in the special case that 
$v$ is the cyclic permutation $(k-p+1, k-p+2, \dots, k,k+1)$; 
note that we should define the degree of $Q_{i}$ 
to be $2$ for all $i \ge 1$, and replace $(-1)^{\sum_{i} d_{i}}$ in 
\cite[Conjecture~7.5]{LM} by $(-1)^{\sum_{i}(2d_{i})}=1$.
%
%
\begin{cor} \label{cor:pos}
For $v = (k-p+1, k-p+2, \dots, k,k+1) \in S_{\infty}$, 
the following positivity holds for all $u,w \in S_{\infty}$ and $d \in \BD$\,{\rm:} 
\begin{equation*}
(-1)^{\ell(u) - \ell(w) - \ell(v)} N_{w,v}^{u}(d) \ge 0.
\end{equation*}
\end{cor}

\begin{proof}
Fix $u \in S_{\infty}$ and $d \in \BD$ arbitrarily. 
From Theorem~\ref{thm:pieri}, we deduce that 
\begin{equation*}
N_{w,v}^{u}(d) = 
\sum_{ \begin{subarray}{c} 
(\bp,M) \in \hsp{k}{p} \\ 
\ed(\bp)=u,\,\Q(\bp)=\Q^{d} 
\end{subarray} } (-1)^{\ell(\bp)-p}. 
\end{equation*}
Let $(\bp,M) \in \hsp{k}{p}$ be such that 
$\ed(\bp)=u$ and $\Q(\bp)=\Q^{d}$. 
Since the directed path $\bp$ starts at $w$ and ends at $u$, 
it is easily verified that 
$(-1)^{\ell(\bp)} = (-1)^{\ell(w)-\ell(u)}$. 
Also, we have $\ell(v) = \ell((k-p+1, k-p+2, \dots, k,k+1)) = p$. 
Therefore, we conclude that 
\begin{align*}
(-1)^{\ell(u) - \ell(w) - \ell(v)}N_{w,v}^{u}(d) 
 & = 
\sum_{ \begin{subarray}{c} 
(\bp,M) \in \hsp{k}{p} \\ 
\ed(\bp)=u,\,\Q(\bp)=\Q^{d} 
\end{subarray} } (-1)^{\ell(u) - \ell(w) - \ell(v) - \ell(\bp) - p} =  
\sum_{ \begin{subarray}{c} 
(\bp,M) \in \hsp{k}{p} \\ 
\ed(\bp)=u,\,\Q(\bp)=\Q^{d} 
\end{subarray} } 1 \ge 0.
\end{align*}
This proves the corollary. 
\end{proof}

%
\section{Proof of Theorem~\ref{thm:pieri}.}
\label{sec:main}

%
\subsection{Outline of the proof.}
\label{subsec:outline}

Let us fix an arbitrary $w \in S_{\infty}$. We will prove Theorem~\ref{thm:pieri} by induction on $k$. 
It is obvious that Theorem~\ref{thm:pieri} holds for 
$k \ge 1$ and $p=0$. Also, we know from \cite[Theorem~6.4]{LM} 
that Theorem~\ref{thm:pieri} holds for $k \ge 1$ and $p=1$. 
Hence Theorem~\ref{thm:pieri} holds for $k = 1$.
Let us assume that $k \ge 2$. 

\vspace{3mm}

In the following, we use the notation: 
\begin{equation*}
\spmw{\h}{\g} = \spm{\h}{\g}:=
 \bigl\{ (\bp \mid \bm) \mid \bp \in \sfp{\h}{\g},\,\bm \in \SM_{k}(\ed(\bp)) \bigr\}, 
\end{equation*}
\begin{equation*}
\hspmw{\h}{\g}=\hspm{\h}{\g}:=
 \bigl\{ ((\bp,M) \mid \bm) \mid (\bp,M) \in \hsp{\h}{\g},\,\bm \in \SM_{k}(\ed(\bp)) \bigr\}
\end{equation*}
for $(\h,\g) \in \{(k-1,p-1),\,(k-1,p),\,(k-2,p-1),\,(k-2,p-2)\}$. 

By \eqref{eq:rec}, we have
\begin{equation} \label{eq:rec1}
\begin{split}
\FG^{\Q}_{w} \G{k}{p}  = \, & \FG^{\Q}_{w} \G{k-1}{p-1} + 
  (1-\Q_{k})(1-x_{k})(1-\Q_{k-1})^{-1} \times \\
& \left( ( \FG^{\Q}_{w} \G{k-1}{p} - \Q_{k-1} \FG^{\Q}_{w} \G{k-2}{p-1} ) -  
         ( \FG^{\Q}_{w} \G{k-1}{p-1} - \Q_{k-1} \FG^{\Q}_{w} \G{k-2}{p-2} ) \right)
\end{split}
\end{equation}
in $\K_{\infty}'$. By the induction hypothesis and Proposition~\ref{prop:LM61}, 
we deduce that for each $(\h,\g) \in \{ (k-1,p-1),\,(k-1,p),\,(k-2,p-1),\,(k-2,p-2) \}$, 
\begin{align}
& (1-\Q_{k})(1-x_{k}) \FG^{\Q}_{w} \G{\h}{\g} = 
  \sum_{(\bp,M) \in \hspw{\h}{\g} = \hsp{\h}{\g}} (-1)^{\ell(\bp)- g} \Q(\bp) 
  (1-\Q_{k})(1-x_{k}) \FG^{\Q}_{\ed(\bp)} \nonumber \\[3mm]
& = \sum_{\bq=((\bp,M) \mid \bm) \in \hspmw{\h}{\g}} 
    \bF{\h}{\g}{\bq}
= \CSF{\h}{\g} \label{eq:rec2}
\end{align}
in $\K_{\infty} \subset \K_{\infty}'$, 
where for $\bq = ((\bp,M) \mid \bm) \in \hspmw{\h}{\g}$, we set 
\begin{equation} \label{eq:bF}
\bF{\h}{\g}{\bq}:=
  (-1)^{\ell(\bp)-\g+\ell(\bm_{(k,*)}) } \Q(\bp)\Q(\bm) \FG^{\Q}_{\ed(\bm)}, 
\end{equation}
and then 
\begin{equation} \label{eq:CSF}
(\CSF{\h}{\g})_{X}:=\sum_{\bq \in X} \bF{\h}{\g}{\bq} 
  \quad \text{for $X \subset \hspmw{\h}{\g}$}, \qquad 
\CSF{\h}{\g}:=(\CS_{\g}^{\h})_{\hspmw{\h}{\g}}. 
\end{equation}
%
%
\begin{rem} \label{rem:empty}
We identify $\hspw{\h}{\g}$ with 
%
%
\begin{equation} \label{eq:empty}
(\hspmw{\h}{\g})_{\emptyset} = (\hspmw{\h}{\g})_{\bm = \emptyset}:=
\bigl\{ ((\bp,M) \mid \bm) \in \hspmw{\h}{\g} \mid 
\bm = \emptyset \bigr\} \subset \hspmw{\h}{\g}; 
\end{equation}
under this identification, a subset $X$ of $\hspw{\h}{\g}$ is 
identified with 
$\bigl\{ ((\bp,M) \mid \bm) \in \hspmw{\h}{\g} \mid (\bp,M) \in X,\,
\bm = \emptyset \bigr\} \subset (\hspmw{\h}{\g})_{\emptyset} \subset \hspmw{\h}{\g}$. 
Let $(\bp,M) \in \hspw{\h}{\g}$, and set 
$\bq := ((\bp,M) \mid \bm)$ with $\bm = \emptyset$. 
Since $\ell(\bm_{(k,*)}) = 0$, $\Q(\bm) =1$, $\ed(\bm)=\ed(\bp)$, 
we see that $\bE{\h}{\g}{\bp,M}:=(-1)^{\ell(\bp)-g} \Q(\bp) \FG^{\Q}_{\ed(\bp)}$ 
is identical to $\bF{\h}{\g}{\bq}$; see \eqref{eq:bEkp} below. 
\end{rem}

By the induction hypothesis, we have
\begin{equation} \label{eq:rec3}
\FG^{\Q}_{w} \G{k-1}{p-1} = 
  \sum_{(\bp,M) \in \hspw{k-1}{p-1}} (-1)^{\ell(\bp)-p} \Q(\bp) \FG^{\Q}_{\ed(\bp)}
= \underbrace{ (\CSF{k-1}{p-1})_{(\hspmw{k-1}{p-1})_{\bm = \emptyset}} }_{ =: \, (\CSF{k-1}{p-1})_{\emptyset} }. 
\end{equation}
Substituting \eqref{eq:rec2} and \eqref{eq:rec3} into \eqref{eq:rec1}, 
we obtain 
\begin{align}
\FG^{\Q}_{w} \G{k}{p} = \, & (\CSF{k-1}{p-1})_{\emptyset} + 
 (1-\Q_{k-1})^{-1} \times \nonumber \\
& \left( ( \CSF{k-1}{p} - \Q_{k-1} \CSF{k-2}{p-1} ) -  
         ( \CSF{k-1}{p-1} - \Q_{k-1} \CSF{k-2}{p-2} ) \right) \label{eq:Ind1}
\end{align}
in $\K_{\infty}'$. 

\vspace{3mm}

\paragraph{\bf Step 1 (Sections~\ref{subsec:decomp1} and \ref{subsec:mat1}).}
In Section~\ref{subsec:decomp1}, we decompose 
$\hspmw{k-1}{\g}=\hspm{\h}{\g}$ and $\hspmw{k-2}{\g-1}=\hspm{\h}{\g-1}$, 
with $g \in \{p-1,p\}$, into some subsets; 
see \eqref{eq:AB} and \eqref{eq:CD}, respectively. 
We list the conditions appearing in these decompositions 
in \eqref{eq:cond1a} and \eqref{eq:cond1b}. 
Under these decompositions, we can rewrite the term
$(1-\Q_{k-1})^{-1}( \CSF{k-1}{\g} - \Q_{k-1} \CSF{k-2}{\g-1} )$, $\g \in  \{p-1,p\}$, 
which appear on the right-hand side of \eqref{eq:Ind1}, as follows (see \eqref{eq:mat1a}): 
\begin{align}
& (1-\Q_{k-1})^{-1} 
  ( \CSF{k-1}{\g} - \Q_{k-1} \CSF{k-2}{\g-1} ) 
  = (1-\Q_{k-1})^{-1} \times \nonumber \\[1.5mm]
& \left(
\sum_{
\begin{subarray}{c}
\spadesuit \in \{ \RA,\,\RB_1,\,\RB_2,\,\RB_3 \} \\[1mm]
\clubsuit \in \{\RX,\,\RY \}
\end{subarray} } 
(\CSF{k-1}{\g})_{\spadesuit\clubsuit} - Q_{k-1}
\sum_{
\begin{subarray}{c}
\spadesuit \in \{ \RC,\,\RD_{11},\,\RD_{12},\,\RD_2 \} \\[1mm]
\clubsuit \in \{\RX,\,\RY \}
\end{subarray} } 
(\CSF{k-2}{\g-1})_{\spadesuit\clubsuit}
\right), \label{eq:step1}
\end{align}
where $(\CSF{\h}{\g})_{\spadesuit\clubsuit}$ 
denotes the sum of $\bF{\h}{\g}{\bq}$ over the elements $\bq \in \hspmw{\h}{\g}$ 
satisfying conditions $\spadesuit$ and $\clubsuit$. 

In Section~\ref{subsec:mat1}, we state Proposition~\ref{prop:mat1a}, 
which follows from Proposition~\ref{prop:mat1}. 
%
By the equalities in Proposition~\ref{prop:mat1a}, 
we can rewrite \eqref{eq:Ind1} as: 
\begin{align}
& \FG^{\Q}_{w} \G{k}{p} = (\CSF{k-1}{p-1})_{\emptyset} \nonumber \\
& + \bigl(
   (\CSF{k-1}{p})_{\RA\RY} + (\CSF{k-1}{p})_{\RB_2\RY} + 
   (\CSF{k-1}{p})_{\RB_3\RY} - Q_{k-1} (\CSF{k-2}{p-1})_{\RD_2\RY} \bigr) \nonumber \\
& - \bigl(
   (\CSF{k-1}{p-1})_{\RA\RY} + (\CSF{k-1}{p-1})_{\RB_2\RY} + 
   (\CSF{k-1}{p-1})_{\RB_3\RY} - Q_{k-1} (\CSF{k-2}{p-2})_{\RD_2\RY} \bigr). \label{eq:Ind2x}
\end{align}

\vspace{3mm}

\paragraph{\bf Step 2 (Sections~\ref{subsec:decomp2} and \ref{subsec:mat2}).}
Let $g \in \{p-1,p\}$. 
In Section~\ref{subsec:decomp2}, we decompose the subsets 
$(\hspmw{k-1}{\g})_{\RA\RY}$, $(\hspmw{k-1}{\g})_{\RB_2\RY}$, and 
$(\hspmw{k-1}{\g})_{\RB_3\RY}$ of $\hspmw{k-1}{\g}$, 
which corresponds to the terms 
$(\CSF{k-1}{\g})_{\RA\RY}$, $(\CSF{k-1}{\g})_{\RB_2\RY}$, 
and $(\CSF{k-1}{\g})_{\RB_3\RY}$ in \eqref{eq:Ind2x}, respectively, 
into some subsets; 
see \eqref{eq:RARY}, \eqref{eq:decB23Y}, and \eqref{eq:decB23}. 
We list the conditions appearing in these decompositions 
in \eqref{eq:cond2}. 

In Section~\ref{subsec:mat2}, we state Proposition~\ref{prop:mat2a}, 
which follows from Proposition~\ref{prop:mat2}. 
Combining the equalities in Proposition~\ref{prop:mat2a} and 
those obtained by replacing $\hspmw{}{}$ in 
\eqref{eq:RARY} and \eqref{eq:decB23} by $\CSF{}{}$, we see that
\begin{align*}
&  (\CSF{k-1}{\g})_{\RA\RY} + (\CSF{k-1}{\g})_{\RB_2\RY} + 
   (\CSF{k-1}{\g})_{\RB_3\RY} - Q_{k-1} (\CSF{k-2}{\g-1})_{\RD_2\RY} \\
& = (\CSF{k-1}{\g})_{\RA_1 \RY_2} + (\CSF{k-1}{\g})_{\RA_3 \RY_2}+ 
    (\CSF{k-1}{\g})_{\RB_{2}\RY_{2}} + (\CSF{k-1}{\g})_{\RB_{3}^{2}\RY_{2}} \\
& \phantom{=}
  + \sum_{ \spadesuit \in \{ \RA_1,\,\RA_3,\,\RB_{2},\,\RB_{3}^{2} \} }
    (\CSF{k-1}{\g})_{\spadesuit \emptyset}
\end{align*}
for $\g \in \{p-1,p\}$; see \eqref{eq:mat2a}. Also, we have
\begin{equation*}
(\CSF{k-1}{p-1})_{\emptyset}
= 
\sum_{\spadesuit \in \{ \RA_1,\,\RA_2,\,\RA_3,\,\RB_{1},\,\RB_{2},\,\RB_{3}^{1},\,\RB_{3}^{2} \} } 
(\CSF{k-1}{p-1})_{\spadesuit \emptyset}.
\end{equation*}
From these equalities together with \eqref{eq:Ind2x}, 
we deduce that 
\begin{align}
\FG_{w}^{\Q} \G{k}{p} & = 
(\CSF{k-1}{p})_{\RA_1 \RY_2} + (\CSF{k-1}{p})_{\RA_3 \RY_2}+ 
(\CSF{k-1}{p})_{\RB_{2}\RY_{2}} + (\CSF{k-1}{p})_{\RB_{3}^{2}\RY_{2}} \nonumber \\[2mm]
& \phantom{=}
  + \sum_{\spadesuit \in \{ \RA_1,\,\RA_3,\,\RB_{2},\,\RB_{3}^{2} \} } 
    (\CSF{k-1}{p})_{\spadesuit \emptyset} \nonumber \\[2mm]
& \phantom{=}
 - (\CSF{k-1}{p-1})_{\RA_1 \RY_2} - (\CSF{k-1}{p-1})_{\RA_3 \RY_2} 
 - (\CSF{k-1}{p-1})_{\RB_{2}\RY_{2}} - (\CSF{k-1}{p-1})_{\RB_{3}^{2}\RY_{2}} \nonumber \\[2mm]
& \phantom{=}
  + \sum_{\spadesuit \in \{ \RA_2,\,\RB_1,\,\RB_{3}^{1} \} }
    (\CSF{k-1}{p-1})_{\spadesuit \emptyset} \nonumber \\
& = (\CSF{k-1}{p})_{\RA_1\RY_2} + (\CSF{k-1}{p})_{\RE} + 
    (\CSF{k-1}{p})_{\RA_1 \emptyset} + (\CSF{k-1}{p})_{\RG} \nonumber \\
& \phantom{=}
  - (\CSF{k-1}{p-1})_{\RA_1\RY_2} - (\CSF{k-1}{p-1})_{\RE} + (\CSF{k-1}{p-1})_{\RF}, \label{eq:Ind4x}
\end{align}
where $(\CSF{k-1}{\g})_{\RE} = (\CSF{k-1}{\g})_{\RA_3 \RY_2}+
(\CSF{k-1}{\g})_{\RB_{2} \RY_2}+(\CSF{k-1}{\g})_{\RB_{3}^{2} \RY_2}$ for $g \in \{p-1,p\}$, 
and 
\begin{equation*}
  (\CSF{k-1}{p-1})_{\RF} = \sum_{\spadesuit \in \{ \RA_2,\,\RB_1,\,\RB_{3}^{1} \} }
  (\CSF{k-1}{p-1})_{\spadesuit \emptyset}, \qquad 
  (\CSF{k-1}{p})_{\RG} = \sum_{\spadesuit \in \{ \RA_3,\,\RB_{2},\,\RB_{3}^{2} \} }
  (\CSF{k-1}{p})_{\spadesuit \emptyset}; 
\end{equation*}
see \eqref{eq:Ind3} and \eqref{eq:Ind4}. 

\vspace{3mm}

\paragraph{\bf Step 3 (Sections~\ref{subsec:decomp3} and \ref{subsec:mat3}).}
In Section~\ref{subsec:decomp3}, we decompose the subset
$(\hspmw{k-1}{p-1})_{\RF}$ of $\hspmw{k-1}{p-1}$, 
which corresponds to the sum $(\CSF{k-1}{p-1})_{\RF}$ of 
certain terms in \eqref{eq:Ind4x}, 
and $\hspw{k}{p}$ into some subsets; 
see \eqref{eq:RF} and \eqref{eq:RRS}, respectively. 
We list the conditions appearing in these decompositions 
in \eqref{eq:cond3a}, \eqref{eq:cond3b}, and \eqref{eq:cond3c}. 
In particular, we have
\begin{equation} \label{eq:RFx}
(\CSF{k-1}{p-1})_{\RF} = 
(\CSF{k-1}{p-1})_{\RF_{1}} + 
(\CSF{k-1}{p-1})_{\RF_{2}^{1}} + 
(\CSF{k-1}{p-1})_{\RF_{2}^{2}}. 
\end{equation}

Here, recalling from Remark~\ref{rem:empty} 
the definition of $\bE{k}{p}{\bq}$ for $\bq=(\bp,M) \in \hspw{k}{p}$, 
we set 
\begin{equation*}
\CSE{k}{p}:=\sum_{\bq \in \hspw{k}{p}} \bE{k}{p}{\bq}, \qquad 
(\CSE{k}{p})_{\spadesuit}:=\sum_{\bq \in (\hspw{k}{p})_{\spadesuit}} \bE{k}{p}{\bq} 
\text{ for } \spadesuit \in \{\RR, \RS, \RS_1, \RS_2, \RS_1^1, \RS_1^2, \RS_1^{\ta}, \RS_1^{\tb} \}. 
\end{equation*}
Then we have 
%
%
\begin{equation} \label{eq:Ind5x}
\CSE{k}{p} = 
(\CSE{k}{p})_{\RR} + 
(\CSE{k}{p})_{\RS_1^1} + 
(\CSE{k}{p})_{\RS_1^{\ta}} + 
(\CSE{k}{p})_{\RS_1^{\tb}} + 
(\CSE{k}{p})_{\RS_2}. 
\end{equation}

In Section~\ref{subsec:mat3}, we state Proposition~\ref{prop:mat3a}, 
which follows from Proposition~\ref{prop:mat3} in Section~\ref{sec:prfmat3}. 
Now, we conclude that 
\begin{align*}
& \FG_{w}^{\Q}\G{k}{p} - \CSE{k}{p} \\[3mm]
& = (\CSF{k-1}{p})_{\RA_1\RY_2} + (\CSF{k-1}{p})_{\RE} + 
(\CSF{k-1}{p})_{\RA_1 \emptyset} + (\CSF{k-1}{p})_{\RG} \\
& - (\CSF{k-1}{p-1})_{\RA_1\RY_2} - (\CSF{k-1}{p-1})_{\RE} + (\CSF{k-1}{p-1})_{\RF} \\
& - (\CSE{k}{p})_{\RR} - (\CSE{k}{p})_{\RS_1^1} - (\CSE{k}{p})_{\RS_1^{\ta}} - 
  (\CSE{k}{p})_{\RS_1^{\tb}} - (\CSE{k}{p})_{\RS_2} \quad \text{by \eqref{eq:Ind4x} and \eqref{eq:Ind5x}} \\[3mm]
& = (\CSE{k}{p})_{\RS_2} + (\CSE{k}{p})_{\RS_1^{\tb}} - (\CSF{k-1}{p-1})_{\RF_{2}^{2}} + 
(\CSE{k}{p})_{\RR} - (\CSF{k-1}{p-1})_{\RF_{1}} \\
& + (\CSE{k}{p})_{\RS_1^1} + (\CSE{k}{p})_{\RS_1^{\ta}} - (\CSF{k-1}{p-1})_{\RF_{2}^{1}} + (\CSF{k-1}{p-1})_{\RF} \\
& - (\CSE{k}{p})_{\RR} - (\CSE{k}{p})_{\RS_1^1} - (\CSE{k}{p})_{\RS_1^{\ta}} - 
  (\CSE{k}{p})_{\RS_1^{\tb}} - (\CSE{k}{p})_{\RS_2} \quad \text{by Proposition~\ref{prop:mat3a}} \\[3mm]
& = 0 \quad \text{by \eqref{eq:RFx}}. 
\end{align*}
This completes the proof of Theorem~\ref{thm:pieri}.

%
\subsection{Decomposition into subsets (1).}
\label{subsec:decomp1}

Let $\g \in \{p-1,\,p\}$. We set
\begin{align*}
& (\hspw{k-1}{\g})_{\CRA}:=
\bigl\{ (\bp,M) \in \hspw{k-1}{\g} \mid n_{(k-1,k)}(\bp)=0 \bigr\}, \\
& (\hspw{k-1}{\g})_{\CRB}:=
\bigl\{ (\bp,M) \in \hspw{k-1}{\g} \mid n_{(k-1,k)}(\bp)=1 \bigr\}, \\
& (\hspw{k-1}{\g})_{\CRBa}:=
\bigl\{ (\bp,M) \in \hspw{k-1}{\g} \mid n_{(k-1,k)}(\bp)=1,\,(k-1,k) \not\in M \bigr\}, \\
& (\hspw{k-1}{\g})_{\CRBb}:=
\bigl\{ (\bp,M) \in \hspw{k-1}{\g} \mid (k-1,k) \in M,\,\kappa(\bp)=(k-1,k) \bigr\}, \\
& (\hspw{k-1}{\g})_{\CRBc}:=
\bigl\{ (\bp,M) \in \hspw{k-1}{\g} \mid (k-1,k) \in M,\,\kappa(\bp) \ne (k-1,k) \bigr\}; 
\end{align*}
recall from Section~\ref{subsec:not} that $n_{(k-1,k)}(\bp)$ denotes the number of 
edges in $\bp$ whose label is $(k-1,k)$, and $\kappa(\bp)$ denotes the final label of $\bp$.
Also, notice that if $(k-1,k) \in M$, then $n_{(k-1,k)}(\bp)=1$. 
We have 
\begin{align} 
\hspw{k-1}{\g} & = (\hspw{k-1}{\g})_{ \CRA } \sqcup (\hspw{k-1}{\g})_{\CRB} \nonumber \\[3mm]
& = (\hspw{k-1}{\g})_{ n_{(k-1,k)}=0 } \sqcup (\hspw{k-1}{\g})_{\CRBa} \nonumber \\ 
& \phantom{=} \sqcup (\hspw{k-1}{\g})_{\CRBb} \sqcup (\hspw{k-1}{\g})_{ \CRBc }. \label{eq:PAB}
\end{align}

\begin{ex} \label{ex:AB}
We use the notation and setting of Example~\ref{ex2}. It is easily verified that
\begin{align*}
& \# (\hspw{3}{2})_{ n_{(3,4)}=0 } = 5, \quad 
\text{e.g., }
( (w \,;\, (3,6)_{\SB}, (1,5)_{\SB}), \{(3,6),(1,5)\} ) \in (\hspw{3}{2})_{n_{(3,4)}=0}; \\
& \# (\hspw{3}{2})_{ n_{(3,4)}=1 } =  18; \\
& \# (\hspw{3}{2})_{ \begin{subarray}{c} n_{(3,4)}=1 \\ (3,4) \notin M \end{subarray} } = 5, \quad 
\text{e.g., } 
( (w \,;\, (1,5)_{\SB}, (2,5)_{\SB}, (3,4)_{\SQ}), \{(1,5),(2,5)\} ) \in 
(\hspw{3}{2})_{ \begin{subarray}{c} n_{(3,4)}=1 \\ (3,4) \notin M \end{subarray} }; \\
& \# (\hspw{3}{2})_{ \begin{subarray}{c} (3,4) \in M \\ \kappa=(3,4) \end{subarray} } = 3, \quad 
\text{e.g., } 
( (w \,;\, (1,5)_{\SB}, (2,5)_{\SB}, (3,4)_{\SQ}), \{(1,5),(3,4)\}) \in 
(\hspw{3}{2})_{ \begin{subarray}{c} (3,4) \in M \\ \kappa=(3,4) \end{subarray} }; \\
& \# (\hspw{3}{2})_{ \begin{subarray}{c} (3,4) \in M \\ \kappa \ne (3,4) \end{subarray} } = 10, \quad
\text{e.g., } 
( (w \,;\, (3,4)_{\SQ}, (1,4)_{\SB}), \{(3,4),(1,4)\} ) \in 
(\hspw{3}{2})_{ \begin{subarray}{c} (3,4) \in M \\ \kappa \ne (3,4) \end{subarray} }. 
\end{align*}
\end{ex}

\begin{rem} \label{rem:B123} \mbox{}
\begin{enu}
\item[(1)] Note that $\max \SL_{k-1}=(k-1,k)$ in the total order $\preceq$. 
Also, we deduce by Definition~\ref{dfn:mark}\,(2) that 
if $(\bp,M) \in (\hspw{k-1}{\g})_{\CRBa}$, then 
$\kappa(\bp)=(k-1,k)$. 

\item[(2)] By Definition~\ref{dfn:mark}\,(1), we see that 
if $(\bp,M) \in (\hspw{k-1}{\g})_{\CRBb}$, then $n_{(k-1,*)}(\bp)=1$; 
recall from Section~\ref{subsec:not} that $n_{(k-1,*)}(\bp)$ denotes the number of 
edges in $\bp$ whose label is of the form $(k-1,b)$ for some $b$.

\item[(3)] If $(\bp,M) \in (\hspw{k-1}{\g})_{ \CRBc }$, then $n_{(k-1,*)}(\bp) = 1$. 
Indeed, suppose, for a contradiction, that $n_{(k-1,*)}(\bp) \ge 2$. 
Since $\kappa(\bp) \ne (k-1,k)$ and $\max \SL_{k-1}=(k-1,k)$, 
we see by (P2) that a label of the form $(k-1,b)$ appears 
after $(k-1,k)$ in $\bp$; notice that $b > k$ by (P0).  
Therefore, it follows from (P1) that $k \ge b$, which is a contradiction. 
\end{enu}
\end{rem}

Let $\spadesuit$ be one of the following conditions on $(\bp,M) \in \hspw{k-1}{\g}$: 
\begin{equation} \label{eq:cAB}
\begin{cases}
(\RA) & n_{(k-1,k)}(\bp)=0, \\
(\RB) & n_{(k-1,k)}(\bp)=1, \\
(\RB_1) & \text{$n_{(k-1,k)}(\bp)=1$ and $(k-1,k) \not\in M$}, \\
(\RB_2) & \text{$(k-1,k) \in M$ and $\kappa(\bp)=(k-1,k)$}, \\
(\RB_3) & \text{$(k-1,k) \in M$ and $\kappa(\bp) \ne (k-1,k)$}.
\end{cases}
\end{equation}
We set 
\begin{align*}
& (\hspmw{k-1}{\g})_{\begin{subarray}{l}\spadesuit \\ \CRX \end{subarray}} := 
\bigl\{ ((\bp,M) \mid \bm) \in \hspmw{k-1}{\g} 
\mid (\bp,M) \in (\hspw{k-1}{\g})_{\spadesuit},\,
\iota(\bm)=(k-1,k) \bigr\}, \\
& (\hspmw{k-1}{\g})_{\begin{subarray}{l} \spadesuit \\ \CRY \end{subarray}} := 
\bigl\{ ((\bp,M) \mid \bm) \in \hspmw{k-1}{\g} 
\mid (\bp,M) \in (\hspw{k-1}{\g})_{\spadesuit},\,
\iota(\bm) \ne (k-1,k) \bigr\};
\end{align*}
recall from Section~\ref{subsec:not} that 
$\iota(\bm)$ denotes the initial label of $\bm$. 
Here we also consider the following conditions on $((\bp,M) \mid \bm) \in \hspmw{k-1}{\g}$:
\begin{equation} \label{eq:cXY}
\begin{cases}
(\RX) & \iota(\bm) = (k-1,k), \\
(\RY) & \iota(\bm) \ne (k-1,k).
\end{cases}
\end{equation}
We have 
\begin{equation} \label{eq:AB}
\hspmw{k-1}{\g} = 
\bigsqcup_{
\begin{subarray}{c}
\spadesuit \in \{\RA,\,\RB \} \\[1mm]
\clubsuit \in \{\RX,\,\RY \}
\end{subarray} } 
(\hspmw{k-1}{\g})_{\spadesuit\clubsuit} =
\bigsqcup_{
\begin{subarray}{c}
\spadesuit \in \{\RA,\,\RB_1,\,\RB_2,\,\RB_3 \} \\[1mm]
\clubsuit \in \{\RX,\,\RY \}
\end{subarray} } 
(\hspmw{k-1}{\g})_{\spadesuit\clubsuit}. 
\end{equation}

Next, for $\g \in \{p-1,\,p\}$, we set
\begin{align*}
& (\sfpw{k-2}{\g-1})_{ \CRC }:=
\bigl\{ \bp \in \sfpw{k-2}{\g-1} \mid n_{(*,k-1)}(\bp) = 0 \bigr\}, \\
& (\sfpw{k-2}{\g-1})_{ \CRD }:=
\bigl\{ \bp \in \sfpw{k-2}{\g-1} \mid n_{(*,k-1)}(\bp) \ge 1 \bigr\}; 
\end{align*}
recall from Section~\ref{subsec:not} that $n_{(*,k-1)}(\bp)$ denotes the number of 
edges in $\bp$ whose label is of the form $(a,k-1)$ for some $a$.
We consider the following conditions $\RC$ and $\RD$: 
\begin{equation} \label{eq:cCD}
\begin{cases}
(\RC) & n_{(*,k-1)}(\bp) = 0,  \\
(\RD) & n_{(*,k-1)}(\bp) \ge 1. 
\end{cases}
\end{equation}

\begin{ex} \label{ex:CD}
We use the notation and setting of Example~\ref{ex2}. It is easily verified that
\begin{align*}
& \# (\sfpw{3}{2})_{ n(*,4)=0 } = 4, \quad 
\text{e.g., } (w \,;\, (3,6)_{\SB}, (1,5)_{\SB}, (2,5)_{\SB}) \in (\sfpw{3}{2})_{ n(*,4)=0 }; \\
& \# (\sfpw{3}{2})_{ n(*,4) \ge 1 } = 16, \quad 
\text{e.g., } (w \,;\, (1,5)_{\SB}, (3,4)_{\SQ}, (1,4)_{\SB}) \in (\sfpw{3}{2})_{ n(*,4) \ge 1 }. 
\end{align*}
\end{ex}

%
\begin{dfn} \label{dfn:D**}
Let $\bp \in (\sfpw{k-2}{\g-1})_{\RD}=(\sfpw{k-2}{\g-1})_{\CRD}$, and write it as: 
\begin{equation} \label{eq:dec12-1}
\bp = (w\,;\,\dots\dots,
\overbrace{(i_{1},k),\dots,(i_{s},k)}^{ =\,\bp_{(*,k)} },
\overbrace{(j_{1},k-1),\dots,(j_{t},k-1)}^{ =\,\bp_{(*,k-1)} }), 
\end{equation}
where $s \ge 0$, $t \ge 1$, and $1 \le i_{1},\dots,i_{s},j_{1},\dots,j_{t} \le k-2$. 
Consider the following directed path obtained by adding an edge labeled 
by $(k-1,k)$ to the end of $\bp$: 
\begin{equation} \label{eq:Alg0}
(w\,;\,\dots\dots,
  (i_{1},k),\dots,(i_{s},k), 
  \overbrace{(j_{1},k-1),\dots,(j_{t},k-1)}^{ =\,\bp_{(*,k-1)} },(k-1,k)). 
\end{equation}
Run {\bf Algorithm $(\bp_{(*,k-1)}:(k-1,k))$} (see Section~\ref{subsec:not}) 
for this directed path, and then consider the following conditions: 
\begin{enumerate}
%
\item[($\RD_1$)]
%
{\bf Algorithm $(\bp_{(*,k-1)}:(k-1,k))$} ends with a directed path of the form: 
\begin{equation} \label{eq:dec12-2}
( w\,;\,\dots\dots,
  (i_{1},k),\dots,(i_{s},k),(k-1,k),
  (j_{1},k),(j_{2},k),\dots,(j_{t},k) ), 
\end{equation}
%
\item[($\RD_{11}$)]
%
{\bf Algorithm $(\bp_{(*,k-1)}:(k-1,k))$} ends with a directed path of 
the form \eqref{eq:dec12-2} above, with 
$\{i_{1},\dots,i_{s}\} \cap \{j_{1},\dots,j_{t}\} = \emptyset$, 
%
\item[($\RD_{12}$)]
%
{\bf Algorithm $(\bp_{(*,k-1)}:(k-1,k))$} ends with a directed path of 
the form \eqref{eq:dec12-2} above, with 
$\{i_{1},\dots,i_{s}\} \cap \{j_{1},\dots,j_{t}\} \ne \emptyset$, 
%
\item[($\RD_2$)]
%
{\bf Algorithm $(\bp_{(*,k-1)}:(k-1,k))$} ends with a directed path of the form:
\begin{equation} \label{eq:dec12-3}
\begin{split}
(w\,;\, & \dots\dots,(i_{1},k),\dots,(i_{s},k),(j_{1},k-1),\dots,(j_{t(\bp)-1},k-1), \\
& (j_{t(\bp)},k),(j_{t(\bp)},k-1),(j_{t(\bp)+1},k),\dots,(j_{t},k))
\end{split}
\end{equation}
for some $1 \le t(\bp) \le t$. 
\end{enumerate}

\begin{ex} \label{ex:D**}
We use the notation and setting of Example~\ref{ex2};
recall from Example~\ref{ex:CD} that $\# (\sfpw{3}{2})_{ n(*,4) \ge 1 } = 16$. 
In the following list, we omit $w = 32514$ from directed paths, 
and write only label sequences. 
\begin{equation*}
\begin{array}{c|c|c}
\bp \in (\sfpw{3}{2})_{ n(*,4) \ge 1 } & \text{Result of $(\bp_{(*,4)}:(4,5))$ for $\bp$} & \\ \hline\hline
(3,6)_{\SB}, (1,5)_{\SB}, (2,5)_{\SB}, (3,4)_{\SQ} & 
(3,6)_{\SB}, (1,5)_{\SB}, (2,5)_{\SB}, (4,5)_{\SB}, (3,5)_{\SQ} & \RD_{11} \\ \hline
(3,6)_{\SB}, (1,5)_{\SB}, (3,4)_{\SQ} & 
(3,6)_{\SB}, (1,5)_{\SB}, (4,5)_{\SB}, (3,5)_{\SQ} & \RD_{11} \\ \hline
(3,6)_{\SB}, (2,5)_{\SB}, (3,4)_{\SQ} & 
(3,6)_{\SB}, (2,5)_{\SB}, (4,5)_{\SB}, (3,5)_{\SQ} & \RD_{11} \\ \hline
(1,5)_{\SB}, (2,5)_{\SB}, (3,4)_{\SQ} & 
(1,5)_{\SB}, (2,5)_{\SB}, (4,5)_{\SB}, (3,5)_{\SQ} & \RD_{11} \\ \hline
(1,5)_{\SB}, (2,5)_{\SB}, (3,4)_{\SQ}, (1,4)_{\SB} & 
(1,5)_{\SB}, (2,5)_{\SB}, (4,5)_{\SB}, (3,5)_{\SQ}, (1,5)_{\SB} & \RD_{12} \\ \hline
(1,5)_{\SB}, (2,5)_{\SB}, (3,4)_{\SQ}, (1,4)_{\SB},(2,4)_{\SB} & 
(1,5)_{\SB}, (2,5)_{\SB}, (4,5)_{\SB}, (3,5)_{\SQ}, (1,5)_{\SB}, (2,5)_{\SB} & \RD_{12} \\ \hline
(1,5)_{\SB}, (2,5)_{\SB}, (3,4)_{\SQ}, (2,4)_{\SB} & 
(1,5)_{\SB}, (2,5)_{\SB}, (4,5)_{\SB}, (3,5)_{\SQ}, (2,5)_{\SB} & \RD_{12} \\ \hline
(1,5)_{\SB}, (3,4)_{\SQ} & 
(1,5)_{\SB}, (4,5)_{\SB}, (3,5)_{\SQ} & \RD_{11} \\ \hline
(1,5)_{\SB}, (3,4)_{\SQ}, (1,4)_{\SB} & 
(1,5)_{\SB}, (4,5)_{\SB}, (3,5)_{\SQ}, (1,5)_{\SB} & \RD_{12} \\ \hline
(1,5)_{\SB}, (3,4)_{\SQ}, (1,4)_{\SB}, (2,4)_{\SB} & 
(1,5)_{\SB}, (3,4)_{\SQ}, (1,4)_{\SB}, (2,5)_{\SB}, (2,4)_{\SB} & \RD_{2} \\ \hline
(1,5)_{\SB}, (3,4)_{\SQ}, (2,4)_{\SB} & 
(1,5)_{\SB}, (3,4)_{\SQ}, (2,5)_{\SB}, (2,4)_{\SB} & \RD_{2} \\ \hline
(2,5)_{\SB}, (3,4)_{\SQ} & (2,5)_{\SB}, (4,5)_{\SB}, (3,5)_{\SQ} & \RD_{11} \\ \hline
(2,5)_{\SB}, (3,4)_{\SQ}, (2,4)_{\SB} & 
(2,5)_{\SB}, (4,5)_{\SB}, (3,5)_{\SQ}, (2,5)_{\SB}  & \RD_{12} \\ \hline 
(3,4)_{\SQ}, (1,4)_{\SB} & (3,4)_{\SQ}, (1,5)_{\SB}, (1,4)_{\SB} & \RD_{2} \\ \hline
(3,4)_{\SQ}, (1,4)_{\SB}, (2,4)_{\SB} & 
(3,4)_{\SQ}, (1,5)_{\SB}, (1,4)_{\SB}, (2,5)_{\SB} & \RD_{2} \\ \hline
(3,4)_{\SQ}, (2,4)_{\SB} & 
(3,4)_{\SQ}, (2,5)_{\SB}, (2,4)_{\SB} & \RD_{2} 
\end{array}
\end{equation*}
\end{ex}

For each $\spadesuit \in \{ \RC,\,\RD_{1},\,\RD_{11},\,\RD_{12},\,\RD_{2} \}$, 
we set
\begin{equation*}
(\sfpw{k-2}{\g-1})_{\spadesuit}:=
\bigl\{ \bp \in \sfpw{k-2}{\g-1} \mid \text{$\bp$ satisfies condition ($\spadesuit$)} \bigr\}.
\end{equation*}
\end{dfn}

It is easily seen that 
\begin{equation*}
(\sfpw{k-2}{\g-1})_{\RD} = (\sfpw{k-2}{\g-1})_{\RD_{1}} \sqcup (\sfpw{k-2}{\g-1})_{\RD_{2}} = 
(\sfpw{k-2}{\g-1})_{\RD_{11}} \sqcup (\sfpw{k-2}{\g-1})_{\RD_{12}} \sqcup (\sfpw{k-2}{\g-1})_{\RD_{2}}.
\end{equation*}
For each $\spadesuit \in \{\RC,\,\RD,\,\RD_{1},\RD_{2},\RD_{11},\RD_{12} \}$, we set
\begin{align*}
& (\hspw{k-2}{\g-1})_{\spadesuit}:=
\bigl\{ (\bp,M) \in \hspw{k-2}{\g-1} \mid \bp \in (\sfpw{k-2}{\g-1})_{\spadesuit} \bigr\}, \\
& (\hspmw{k-2}{\g-1})_{\spadesuit \RX} := 
\bigl\{ ((\bp,M) \mid \bm) \in \hspmw{k-2}{\g-1} 
\mid (\bp,M) \in (\hspw{k-2}{\g-1})_{\spadesuit},\,
\iota(\bm) = (k-1,k) \bigr\}, \\[1.5mm]
& (\hspmw{k-2}{\g-1})_{\spadesuit \RY} := 
\bigl\{ ((\bp,M) \mid \bm) \in \hspmw{k-2}{\g-1} 
\mid (\bp,M) \in (\hspw{k-2}{\g-1})_{\spadesuit},\,
\iota(\bm) \ne (k-1,k) \bigr\}, 
\end{align*}
where conditions ($\RX$) and ($\RY$) 
are as given in \eqref{eq:cXY}. We have 
\begin{align} 
\hspm{k-2}{\g-1} & = \bigsqcup_{
\begin{subarray}{c}
\spadesuit \in \{\RC,\,\RD \} \\[1mm]
\clubsuit \in \{\RX,\,\RY \}
\end{subarray} } 
(\hspmw{k-2}{\g-1})_{\spadesuit\clubsuit}
= \bigsqcup_{
\begin{subarray}{c}
\spadesuit \in \{\RC,\,\RD_{1},\,\RD_{2} \} \\[1mm]
\clubsuit \in \{\RX,\,\RY \}
\end{subarray} } 
(\hspmw{k-2}{\g-1})_{\spadesuit\clubsuit} \nonumber \\[3mm]
& = 
\bigsqcup_{
\begin{subarray}{c}
\spadesuit \in \{\RC,\,\RD_{11},\,\RD_{12},\,\RD_{2} \} \\[1mm]
\clubsuit \in \{\RX,\,\RY \}
\end{subarray} } 
(\hspmw{k-2}{\g-1})_{\spadesuit\clubsuit}. \label{eq:CD}
\end{align}
%
%
\subsection{Matching (1).}
\label{subsec:mat1}
Here we list the conditions introduced in Section~\ref{subsec:decomp1}, 
which are needed in this subsection and in 
Section~\ref{sec:prfmat1}: 
\begin{itemize}
\item For $((\bp,M) \mid \bm) \in \hspmw{k-1}{\g}$ with $g \in \{p-1,p\}$, 
\begin{equation} \label{eq:cond1a}
\begin{cases}
(\RA) & n_{(k-1,k)}(\bp)=0, \\
(\RB) & n_{(k-1,k)}(\bp)=1, \\
(\RB_1) & \text{$n_{(k-1,k)}(\bp)=1$ and $(k-1,k) \not\in M$}, \\
(\RB_2) & \text{$(k-1,k) \in M$ and $\kappa(\bp)=(k-1,k)$}, \\
(\RB_3) & \text{$(k-1,k) \in M$ and $\kappa(\bp) \ne (k-1,k)$}, \\
(\RX) & \iota(\bm) = (k-1,k), \\
(\RY) & \iota(\bm) \ne (k-1,k).
\end{cases}
\end{equation}
\item For $((\bp,M) \mid \bm) \in \hspmw{k-2}{\g-1}$ with $g \in \{p-1,p\}$, 
\begin{equation} \label{eq:cond1b}
\begin{cases}
(\RC) & n_{(*,k-1)}(\bp) = 0, \\
(\RD) & n_{(*,k-1)}(\bp) \ge 1, \\
(\RD_1) & \text{see Definition~\ref{dfn:D**}; {\bf Algorithm $(\bp_{(*,k-1)}:(k-1,k))$} ends with} \\
& \text{a directed path of the form \eqref{eq:dec12-2}}, \\
(\RD_{11}) & \text{see Definition~\ref{dfn:D**}; ($\RD_1$) holds, and  
$\{i_{1},\dots,i_{s}\} \cap \{j_{1},\dots,j_{t}\} = \emptyset$}, \\
(\RD_{12}) & \text{see Definition~\ref{dfn:D**}; ($\RD_1$) holds, and  
$\{i_{1},\dots,i_{s}\} \cap \{j_{1},\dots,j_{t}\} \ne \emptyset$}, \\
(\RD_2) & \text{see Definition~\ref{dfn:D**}; {\bf Algorithm $(\bp_{(*,k-1)}:(k-1,k))$} ends with} \\
& \text{a directed path of the form \eqref{eq:dec12-3}}, \\
(\RX) & \iota(\bm) = (k-1,k), \\
(\RY) & \iota(\bm) \ne (k-1,k).
\end{cases}
\end{equation}
\end{itemize}
In the proposition below, 
$(\CSF{\h}{\g})_{\spadesuit \clubsuit}$
denotes the sum of $\bF{\h}{\g}{\bq}$ 
over $\bq \in (\hspmw{\h}{\g})_{\spadesuit \clubsuit}$; 
see \eqref{eq:bF} and \eqref{eq:CSF}. 
%
%
\begin{prop}[to be proved in Section~\ref{sec:prfmat1}] \label{prop:mat1a}
The following equalities hold in $\K_{\infty}'$\,{\rm:} 
\begin{align*}
& (\CSF{k-1}{\g})_{\RB_1\RY} = - (\CSF{k-1}{\g})_{\RA\RX}, & 
& (\CSF{k-1}{\g})_{\RB_1\RX} = - \Q_{k-1} (\CSF{k-1}{\g})_{\RA\RY}, \\
& (\CSF{k-2}{\g-1})_{\RC\RY} = \Q_{k-1}^{-1} (\CSF{k-1}{\g})_{\RB_2\RX}, & 
& (\CSF{k-2}{\g-1})_{\RC\RX} = (\CSF{k-1}{\g})_{\RB_2\RY}, \\ 
& (\CSF{k-2}{\g-1})_{\RD_{11}\RY} = \Q_{k-1}^{-1} (\CSF{k-1}{\g})_{\RB_3 \RX}, & 
& (\CSF{k-2}{\g-1})_{\RD_{11}\RX} = (\CSF{k-1}{\g})_{\RB_3 \RY}, \\
& (\CSF{k-2}{\g-1})_{\RD_2\RY} = - \Q_{k-1}^{-1} (\CSF{k-2}{\g-1})_{\RD_{12}\RX}, & 
& (\CSF{k-2}{\g-1})_{\RD_2\RX} = - (\CSF{k-2}{\g-1})_{\RD_{12}\RY}. 
\end{align*}
\end{prop}

From \eqref{eq:AB} and \eqref{eq:CD}, we deduce that in $\K_{\infty}'$, 
\begin{align}
& (1-\Q_{k-1})^{-1} 
  ( \CSF{k-1}{\g} - \Q_{k-1} \CSF{k-2}{\g-1} ) 
  = (1-\Q_{k-1})^{-1} \times \nonumber \\[1.5mm]
& \left(
\sum_{
\begin{subarray}{c}
\spadesuit \in \{ \RA,\,\RB_1,\,\RB_2,\,\RB_3 \} \\[1mm]
\clubsuit \in \{\RX,\,\RY \}
\end{subarray} } 
(\CSF{k-1}{\g})_{\spadesuit\clubsuit} - Q_{k-1}
\sum_{
\begin{subarray}{c}
\spadesuit \in \{ \RC,\,\RD_{11},\,\RD_{12},\,\RD_2 \} \\[1mm]
\clubsuit \in \{\RX,\,\RY \}
\end{subarray} } 
(\CSF{k-2}{\g-1})_{\spadesuit\clubsuit}
\right),  \label{eq:mat1a}
\end{align}
where 
\begin{align*}
& (\CSF{k-1}{\g})_{\spadesuit\clubsuit} = 
  (\CSF{k-1}{\g})_{(\hspmw{k-1}{\g})_{\spadesuit\clubsuit}}, \qquad
& (\CSF{k-2}{\g-1})_{\spadesuit\clubsuit} = 
  (\CSF{k-2}{\g-1})_{(\hspmw{k-2}{\g-1})_{\spadesuit\clubsuit}}. 
\end{align*}
Substituting the equalities in Proposition~\ref{prop:mat1a} 
into the right-hand side of \eqref{eq:mat1a}, we obtain 
\begin{align*}
& (1-\Q_{k-1})^{-1} 
  ( \CSF{k-1}{\g} - \Q_{k-1} \CSF{k-2}{\g-1} ) \\
& = 
(\CSF{k-1}{\g})_{\RA\RY} + (\CSF{k-1}{\g})_{\RB_2\RY} + 
(\CSF{k-1}{\g})_{\RB_3\RY} - Q_{k-1} (\CSF{k-2}{\g-1})_{\RD_2\RY}. 
\end{align*}
Combining this equality with \eqref{eq:Ind1}, we conclude that in $\K_{\infty}$,
\begin{equation} \label{eq:Ind2}
\begin{split}
& \FG^{\Q}_{w} \G{k}{p} = (\CSF{k-1}{p-1})_{\emptyset} \\
& + \bigl(
   (\CSF{k-1}{p})_{\RA\RY} + (\CSF{k-1}{p})_{\RB_2\RY} + 
   (\CSF{k-1}{p})_{\RB_3\RY} - Q_{k-1} (\CSF{k-2}{p-1})_{\RD_2\RY} \bigr) \\
& - \bigl(
   (\CSF{k-1}{p-1})_{\RA\RY} + (\CSF{k-1}{p-1})_{\RB_2\RY} + 
   (\CSF{k-1}{p-1})_{\RB_3\RY} - Q_{k-1} (\CSF{k-2}{p-2})_{\RD_2\RY} \bigr).
\end{split}
\end{equation}

%
\subsection{Decomposition into subsets (2).}
\label{subsec:decomp2}

Let $g \in \{p-1,p\}$. Recall condition $\RA$, that is, 
$n_{(k-1,k)}(\bp) = 0$ for $\bp \in \hspw{k-1}{\g}$, 
where $n_{(k-1,k)}(\bp)$ denotes the number of edges whose label is $(k-1,k)$. 
We decompose 
$(\hspw{k-1}{\g})_{\RA} = (\hspw{k-1}{\g})_{\CRA}$ into 
the following three subsets: 
\begin{align*}
(\hspw{k-1}{\g})_{ \begin{subarray}{c} \CRA \\ \CRAa \end{subarray} } & :=
 \bigl\{ (\bp,M) \in (\hspw{k-1}{\g})_{n_{(k-1,k)}(\bp)=0} \mid 
 \bp_{(*,k)} = \emptyset \bigr\}, \\
(\hspw{k-1}{\g})_{ 
 \begin{subarray}{c} \CRA \\ \CRAb \end{subarray} } & :=
 \bigl\{ (\bp,M) \in (\hspw{k-1}{\g})_{n_{(k-1,k)}(\bp)=0} \mid 
 \bp_{(*,k)} \ne \emptyset \text{ and } \kappa(\bp) \not\in M \bigr\}, \\
(\hspw{k-1}{\g})_{ 
 \begin{subarray}{c} \CRA \\ \CRAc \end{subarray} } & :=
 \bigl\{ (\bp,M) \in (\hspw{k-1}{\g})_{n_{(k-1,k)}(\bp)=0} \mid 
 \bp_{(*,k)} \ne \emptyset \text{ and } \kappa(\bp) \in M \bigr\}; 
\end{align*}
recall from Section~\ref{subsec:not} that 
$\kappa(\bp)$ denotes the final label of $\bp$, 
and from Section~\ref{subsec:pieri} that 
$\bp_{(*,k)}$ denotes the $(*,k)$-segment of $\bp$. 
Note that 
\begin{equation} \label{eq:A123}
(\hspw{k-1}{\g})_{ \CRA } = 
(\hspw{k-1}{\g})_{\begin{subarray}{c} \CRA \\ \CRAa \end{subarray}} 
\sqcup 
(\hspw{k-1}{\g})_{\begin{subarray}{c} \CRA \\ \CRAb \end{subarray}}
\sqcup (\hspw{k-1}{\g})_{\begin{subarray}{c} \CRA \\ \CRAc \end{subarray}}; 
\end{equation}
here we consider the following conditions: 
\begin{equation} \label{eq:cAn}
\begin{cases}
(\RA_1) & n_{(k-1,k)}(\bp)=0 \text{ and } \bp_{(*,k)} = \emptyset, \\
(\RA_2) & n_{(k-1,k)}(\bp)=0, \, \bp_{(*,k)} \ne \emptyset, \text{ and } \kappa(\bp) \not\in M, \\
(\RA_3) & n_{(k-1,k)}(\bp)=0, \, \bp_{(*,k)} \ne \emptyset, \text{ and } \kappa(\bp) \in M.
\end{cases}
\end{equation}
Then we can rewrite \eqref{eq:A123} as: 
\begin{equation} \label{eq:decA}
(\hspw{k-1}{\g})_{\RA} = (\hspw{k-1}{\g})_{\RA_1} \sqcup 
(\hspw{k-1}{\g})_{\RA_2} \sqcup (\hspw{k-1}{\g})_{\RA_3}. 
\end{equation}

\begin{ex} \label{ex:A**}
(1) We use the notation and setting of Example~\ref{ex2};
recall from Example~\ref{ex:AB} that $\# (\hspw{3}{2})_{ n_{(3,4)}=0 } = 5$. 
It is easily verified that
\begin{equation*}
(\hspw{3}{2})_{\begin{subarray}{c} n_{(3,4)}=0 \\ \bp_{(*,4)} = \emptyset \end{subarray}} 
= (\hspw{3}{2})_{ n(3,4)=0 }, 
\qquad 
(\hspw{3}{2})_{\begin{subarray}{c} n_{(3,4)}=0 \\ \bp_{(*,4)} \ne \emptyset,\,
\kappa \notin M \end{subarray}} = 
(\hspw{3}{2})_{\begin{subarray}{c} n_{(3,4)}=0 \\ \bp_{(*,4)} = \emptyset,\,
\kappa \in M \end{subarray}} = \emptyset. 
\end{equation*}

(2) We use the notation and setting of Example~\ref{ex1}. 
It is easily verified that
\begin{align*}
& (\hspw{2}{2})_{\begin{subarray}{c} n_{(2,3)}=0 \\ \bp_{(*,3)} = \emptyset \end{subarray}}
=\bigl\{ ( (w \,;\, (1,4)_{\SB}, (2,4)_{\SB}), \{(1,4),(2,4)\} ) \bigr\}, \\
& (\hspw{2}{2})_{\begin{subarray}{c} n_{(2,3)}=0 \\ \bp_{(*,3)} \ne \emptyset,\,
\kappa \notin M \end{subarray}} 
= \bigl\{ ( (w \,;\, (1,4)_{\SB}, (2,4)_{\SB}, (1,3)_{\SQ}), \{(1,4),(2,4)\}) \bigr\}, \\
& (\hspw{2}{2})_{\begin{subarray}{c} n_{(2,3)}=0 \\ \bp_{(*,3)} \ne \emptyset,\,
\kappa \in M \end{subarray}} = \emptyset. 
\end{align*}

\end{ex}

Now, recall condition $\RY$, that is, $\iota(\bm) \ne (k-1,k)$ 
for $((\bp,M) \mid \bm) \in \hspmw{k-1}{\g}$, 
where $\iota(\bm)$ denotes the initial label of $\bm$. 
In addition, consider the following conditions: 
\begin{equation} \label{eq:cYn}
\begin{cases}
(\RY_1) & \iota(\bm) \ne (k-1,k) \text{ and } \bm_{(*,k)} = \emptyset, \\
(\RY_2) & \iota(\bm) \ne (k-1,k),\, \bm_{(*,k)} = \emptyset, \text{ and } \bm_{(k,*)} \ne \emptyset, \\
(\RY_3) & \iota(\bm) \ne (k-1,k) \text{ and } \bm_{(*,k)} \ne \emptyset; 
\end{cases}
\end{equation}
recall from Definition~\ref{dfn:monk} that $\bm_{(*,k)}$ and 
$\bm_{(k,*)}$ denote the $(*,k)$-segment and the $(k,*)$-segment, respectively. 
For $\spadesuit \in \{ \RA_1, \RA_2, \RA_3 \}$, we set
\begin{align*}
& (\hspmw{k-1}{\g})_{\spadesuit \RY}:=
\bigl\{ ((\bp,M) \mid \bm) \in (\hspmw{k-1}{\g})_{\RA \RY} \mid 
(\bp,M) \in (\hspw{k-1}{\g})_{\spadesuit} \bigr\}, \\
& (\hspmw{k-1}{\g})_{\spadesuit \emptyset}:=
\bigl\{ ((\bp,M) \mid \bm)  \in (\hspmw{k-1}{\g})_{\spadesuit \RY} \mid 
\bm = \emptyset \bigr\}, \\
& (\hspmw{k-1}{\g})_{\spadesuit \RY_1}:=
\bigl\{ ((\bp,M) \mid \bm) \in (\hspmw{k-1}{\g})_{\spadesuit \RY} \mid 
\bm_{(*,k)} = \emptyset \bigr\}, \\
& (\hspmw{k-1}{\g})_{\spadesuit \RY_2}:= 
\bigl\{ ((\bp,M) \mid \bm) \in (\hspmw{k-1}{\g})_{\spadesuit \RY} \mid 
\bm_{(*,k)} = \emptyset,\,
\bm_{(k,*)} \ne \emptyset \bigr\}, \\
& (\hspmw{k-1}{\g})_{\spadesuit \RY_3}:=
\bigl\{ ((\bp,M) \mid \bm) \in (\hspmw{k-1}{\g})_{\spadesuit \RY} \mid 
\bm_{(*,k)} \ne \emptyset \bigr\}.
\end{align*}
For each $\spadesuit \in \{ \RA_1, \RA_2, \RA_3 \}$, we have 
\begin{equation*}
\begin{split}
(\hspmw{k-1}{\g})_{\spadesuit \RY_1} & = 
(\hspmw{k-1}{\g})_{\spadesuit \emptyset} \sqcup 
(\hspmw{k-1}{\g})_{\spadesuit \RY_2}, \\
(\hspmw{k-1}{\g})_{\spadesuit \RY} & = 
(\hspmw{k-1}{\g})_{\spadesuit \RY_1} \sqcup 
(\hspmw{k-1}{\g})_{\spadesuit \RY_3} \\
& = (\hspmw{k-1}{\g})_{\spadesuit \emptyset} \sqcup 
(\hspmw{k-1}{\g})_{\spadesuit \RY_2} \sqcup 
(\hspmw{k-1}{\g})_{\spadesuit \RY_3}, 
\end{split}
\end{equation*}
and 
\begin{align} 
(\hspmw{k-1}{\g})_{\RA\RY} & = 
(\hspmw{k-1}{\g})_{\RA_1 \RY_3} \sqcup (\hspmw{k-1}{\g})_{\RA_2 \RY} \sqcup (\hspmw{k-1}{\g})_{\RA_3 \RY_3} \nonumber \\ 
& \quad 
  \sqcup (\hspmw{k-1}{\g})_{\RA_1 \emptyset} \sqcup (\hspmw{k-1}{\g})_{\RA_3 \emptyset} \nonumber \\
& \quad 
  \sqcup (\hspmw{k-1}{\g})_{\RA_1 \RY_2}
  \sqcup (\hspmw{k-1}{\g})_{\RA_3 \RY_2}. \label{eq:RARY}
\end{align}

Next, recall conditions $\RB$, $\RB_2$, and $\RB_3$, that is,
\begin{equation*}
\begin{cases}
(\RB) & \text{$n_{(k-1,k)}(\bp) = 1$}, \\
(\RB_2) & \text{$(k-1,k) \in M$ and $\kappa(\bp)=(k-1,k)$}, \\
(\RB_3) & \text{$(k-1,k) \in M$ and $\kappa(\bp) \ne (k-1,k)$}
\end{cases}
\end{equation*}
for $(\bp,M) \in \hspw{k-1}{\g}$ with $g \in \{p-1,p\}$; 
note that $(k-1,k) \in M$ implies $n_{(k-1,k)}(\bp) = 1$. 
Observe that if $(\bp,M) \in (\hspw{k-1}{\g})_{\RB_2} \sqcup (\hspw{k-1}{\g})_{\RB_3}$, 
then $\bp$ is of the form: 
%
%
\begin{equation} \label{eq:B23Y-1}
\bp = (w\,;\,\underbrace{\dots\dots\dots\dots\dots\dots}_{
\begin{subarray}{c}
\text{This segment contains } \\[1mm]
\text{no label of } \\[1mm]
\text{the form $(k-1,*)$.}
\end{subarray}},
\overbrace{ (i_{1},k),\dots,(i_{s},k),(k-1,k), 
\underbrace{(j_{1},k),\dots,(j_{t},k)}_{=\,\bp_{(*,k)}^{(k-1,k)} }
}^{ =\,\bp_{(*,k)} }), 
\end{equation}
with $s,t \ge 0$; 
recall from Section~\ref{subsec:pieri} that 
$\bp_{(*,k)}^{(k-1,k)}$ denotes the segment 
in $\bp_{(*,k)}$ consisting of all labels appearing after $(k-1,k)$. 
We see that for $(\bp,M) \in (\hspw{k-1}{\g})_{\RB_2} \sqcup (\hspw{k-1}{\g})_{\RB_3}$, 
\begin{equation} \label{eq:B2iff1}
\bp_{(*,k)}^{(k-1,k)} = \emptyset \quad
\text{if and only if} \quad 
(\bp,M) \in(\hspw{k-1}{\g})_{\RB_2}; 
\end{equation}
\begin{equation} \label{eq:B2iff2}
\text{$\kappa(\bp)=(a,k)$ for some $1 \le a \le k-2$} \quad 
\text{if and only if} \quad (\bp,M) \in(\hspw{k-1}{\g})_{\RB_3}.
\end{equation}
We decompose $(\hspw{k-1}{\g})_{\RB_3} = 
(\hspw{k-1}{\g})_{ \CRBc }$ 
into the following two subsets: 
\begin{align*}
& (\hspw{k-1}{\g})_{ \CRBca } := 
\bigl\{ (\bp,M) \in \hspw{k-1}{\g} \mid 
(k-1,k) \in M, \, \kappa(\bp) \ne (k-1,k), \, \kappa(\bp) \notin M \bigr\}, \\
& (\hspw{k-1}{\g})_{ \CRBcb } := 
\bigl\{ (\bp,M) \in \hspw{k-1}{\g} \mid 
(k-1,k) \in M, \, \kappa(\bp) \ne (k-1,k), \, \kappa(\bp) \in M \bigr\};
\end{align*}
here, we consider the following conditions: 
\begin{equation} \label{eq:cB3n}
\begin{cases}
(\RB_{3}^{1}) & (k-1,k) \in M, \, \kappa(\bp) \ne (k-1,k), \text{ and } \kappa(\bp) \notin M, \\
(\RB_{3}^{2}) & (k-1,k) \in M, \, \kappa(\bp) \ne (k-1,k), \text{ and } \kappa(\bp) \in M.
\end{cases}
\end{equation}
Then, we have 
\begin{equation} \label{eq:decB3}
(\hspw{k-1}{\g})_{\RB_3} = (\hspw{k-1}{\g})_{\RB_{3}^{1}} \sqcup (\hspw{k-1}{\g})_{\RB_{3}^{2}}.
\end{equation}

\begin{ex} \label{ex:B3*}
We use the notation and setting of Example~\ref{ex2}; 
recall from Example~\ref{ex:AB} that 
$\# (\hspw{3}{2})_{ \begin{subarray}{c} (3,4) \in M \\ \kappa \ne (3,4) \end{subarray} } = 10$. 
It is easily verified that 
\begin{align*}
& \# (\hspw{3}{2})_{ \begin{subarray}{c} (3,4) \in M \\ 
\kappa \ne (3,4) \\ \kappa \not\in M \end{subarray} } = 8, \quad 
\text{e.g., } 
( (w \,;\, (1,5)_{\SB}, (2,5)_{\SB}, (3,4)_{\SQ}, (1,4)_{\SB}), \{(1,5),(3,4)\}) \in
(\hspw{3}{2})_{ \begin{subarray}{c} (3,4) \in M \\ 
\kappa \ne (3,4) \\ \kappa \not\in M \end{subarray} }; \\[3mm]
& \# (\hspw{3}{2})_{ \begin{subarray}{c} (3,4) \in M \\ 
\kappa \ne (3,4) \\ \kappa \in M \end{subarray} } = 2, \quad
\text{e.g., } 
( (w \,;\, (3,4)_{\SQ}, (1,4)_{\SB}), \{(3,4),(1,4)\}) \in 
(\hspw{3}{2})_{ \begin{subarray}{c} (3,4) \in M \\ 
\kappa \ne (3,4) \\ \kappa \in M \end{subarray} }.
\end{align*}
\end{ex}

Recall condition $\RY$, that is, $\iota(\bm) \ne (k-1,k)$ 
for $((\bp,M) \mid \bm) \in \hspmw{k-1}{\g}$, and also 
conditions $\RY_1$, $\RY_2$, and $\RY_3$ from \eqref{eq:cYn}. 
Let $\spadesuit \in \{ \RB_{2}, \RB_{3}, \RB_{3}^{1}, \RB_{3}^{2} \}$. 
We set
\begin{align*}
& (\hspmw{k-1}{\g})_{\spadesuit \RY}:=
\bigl\{ ((\bp,M) \mid \bm) \in (\hspmw{k-1}{\g})_{\RB \RY} \mid 
(\bp,M) \in (\hspw{k-1}{\g})_{\spadesuit} \bigr\}, \\
& (\hspmw{k-1}{\g})_{\spadesuit \emptyset}:=
\bigl\{ ((\bp,M) \mid \bm)  \in (\hspmw{k-1}{\g})_{\spadesuit \RY} \mid 
\bm = \emptyset \bigr\}, \\
& (\hspmw{k-1}{\g})_{\spadesuit \RY_1}:=
\bigl\{ ((\bp,M) \mid \bm) \in (\hspmw{k-1}{\g})_{\spadesuit \RY} \mid 
\bm_{(*,k)} = \emptyset \bigr\}, \\
& (\hspmw{k-1}{\g})_{\spadesuit \RY_2}:= 
\bigl\{ ((\bp,M) \mid \bm) \in (\hspmw{k-1}{\g})_{\spadesuit \RY} \mid 
\bm_{(*,k)} = \emptyset,\,
\bm_{(k,*)} \ne \emptyset \bigr\}, \\
& (\hspmw{k-1}{\g})_{\spadesuit \RY_3}:=
\bigl\{ ((\bp,M) \mid \bm) \in (\hspmw{k-1}{\g})_{\spadesuit \RY} \mid 
\bm_{(*,k)} \ne \emptyset \bigr\}. 
\end{align*}
We have
\begin{equation} \label{eq:decB23Y}
\begin{split}
& (\hspmw{k-1}{\g})_{\spadesuit \RY}
  = (\hspmw{k-1}{\g})_{\spadesuit \RY_1} \sqcup 
    (\hspmw{k-1}{\g})_{\spadesuit \RY_3}, \\
& \qquad \text{with} \quad
(\hspmw{k-1}{\g})_{\spadesuit \RY_1}
 = (\hspmw{k-1}{\g})_{\spadesuit \emptyset} \sqcup 
    (\hspmw{k-1}{\g})_{\spadesuit \RY_2}.
\end{split}
\end{equation}
In addition, 
for $\spadesuit \in \{ \RB_{2}, \RB_{3}, \RB_{3}^{1}, \RB_{3}^{2} \}$, 
we consider the following conditions on 
$((\bp,M) \mid \bm) \in (\hspmw{k-1}{\g})_{\spadesuit \RY_3}$: 
\begin{equation} \label{eq:cUT}
\begin{cases}
(\RT) & \bp_{(*,k)} \cap \bm_{(*,k)} \ne \emptyset, \\
(\RU) & \bp_{(*,k)} \cap \bm_{(*,k)} = \emptyset. 
\end{cases}
\end{equation}
Then we set
\begin{align*}
& (\hspmw{k-1}{\g})_{\spadesuit \RY_3}^{\RT}:=
\bigl\{ ((\bp,M) \mid \bm) \in (\hspmw{k-1}{\g})_{\spadesuit \RY_3} \mid 
\bp_{(*,k)} \cap \bm_{(*,k)} \ne \emptyset \bigr\}, \\
& (\hspmw{k-1}{\g})_{\spadesuit \RY_3}^{\RU}:=
\bigl\{ ((\bp,M) \mid \bm) \in (\hspmw{k-1}{\g})_{\spadesuit \RY_3} \mid 
\bp_{(*,k)} \cap \bm_{(*,k)} = \emptyset \bigr\}; 
\end{align*}
we have 
\begin{equation*}
(\hspmw{k-1}{\g})_{\spadesuit \RY_3}
= (\hspmw{k-1}{\g})_{\spadesuit \RY_3}^{\RT} \sqcup 
    (\hspmw{k-1}{\g})_{\spadesuit \RY_3}^{\RU}. 
\end{equation*}

%
\begin{rem} \label{rem:B23Y1}
Let $((\bp,M) \mid \bm) \in (\hspmw{k-1}{\g})_{\RB_2\RY_3}^{\RT} \sqcup 
(\hspmw{k-1}{\g})_{\RB_3\RY_3}^{\RT}$. 
We write $\bp$ as in \eqref{eq:B23Y-1}, and $\bm$ as:
\begin{equation*}
\bm = (\ed(\bp)\,;\,
\underbrace{(c_{1},k),\dots,(c_{u},k)}_{%
 =\,\bm_{(\ast,k)} }, 
\underbrace{(k,d_{r}),\dots,(k,d_{1})}_{%
 =\,\bm_{(k,\ast)} }), 
\end{equation*}
where $u \ge 1$ (recall condition $\RT$) and 
$c_{1} \ne k-1$ (recall condition $\RY$, that is, $\iota(\bm) \ne (k-1,k)$). 
By the definition, 
$\{ i_{1},\dots,i_{s},k-1,j_{1},\dots,j_{t} \} \cap \{c_{1},\dots,c_{u}\} \ne \emptyset$.
Recall that $1 \le c_{u'} \le k-2$ for all $1 \le u' \le u$. Since 
\begin{equation*}
(w\,;\,\dots\dots,
  (i_{1},k),\dots,(i_{s},k),(k-1,k),(j_{1},k),\dots,(j_{t},k),
  (c_{1},k),\dots,(c_{u},k))
\end{equation*}
is a directed path, and since $1 \le j_{1},\dots,j_{t} \le k-2$, 
it follows from Lemma~\ref{lem:noak} that 
\begin{equation*}
\{ i_{1},\dots,i_{s},k-1,j_{1},\dots,j_{t} \} \cap \{c_{1},\dots,c_{u}\} = 
\{ i_{1},\dots,i_{s} \} \cap \{c_{1},\dots,c_{u}\}.
\end{equation*}
In particular, we have $\kappa(\bp) \ne \iota(\bm)$. 
\end{rem}

Furthermore, 
for $\spadesuit \in \{\RB_{2},\RB_{3},\RB_{3}^{1},\RB_{3}^{2}\}$, 
we consider the following conditions on 
$\bq=((\bp,M) \mid \bm) \in (\hspmw{k-1}{\g})_{\spadesuit \RY_3}^{\RT}$; 
note that $\kappa(\bp) \ne \iota(\bm)$, as seen in Remark~\ref{rem:B23Y1}, 
and also recall condition $\RT$, that is, 
$\bp_{(*,k)} \cap \bm_{(*,k)} \ne \emptyset$: 
\begin{equation} \label{eq:cUn}
\begin{cases}
(\RT_1) & \bp_{(*,k)} \cap \bm_{(*,k)} \ne \emptyset \text{ and }
          \iota(\bm) \in \bp_{(*,k)}, \\
(\RT_2) &  \bp_{(*,k)} \cap \bm_{(*,k)} \ne \emptyset \text{ and } 
          \iota(\bm) \notin \bp_{(*,k)}, \\
(\RT_3) & \bp_{(*,k)} \cap \bm_{(*,k)} \ne \emptyset, \, 
          \iota(\bm) \in \bp_{(*,k)}, \text{ and } \kappa(\bp) \prec \iota(\bm), \\
(\RT_4) & \bp_{(*,k)} \cap \bm_{(*,k)} \ne \emptyset \text{ and }
          (\iota(\bm) \notin \bp_{(*,k)} \text{ or } \kappa(\bp) \succ \iota(\bm)). 
\end{cases}
\end{equation}
We set 
\begin{align*}
(\hspmw{k-1}{\g})_{\RB_3^1\RY_3}^{\RT_3} & :=
\bigl\{ ((\bp,M) \mid \bm) \in (\hspmw{k-1}{\g})_{\RB_3^1\RY_3}^{\RT} \mid 
\iota(\bm) \in \bp_{(*,k)} \text{ and } \kappa(\bp) \prec \iota(\bm) \bigr\}, \\[2mm]
(\hspmw{k-1}{\g})_{\RB_3^1\RY_3}^{\RT_4} & :=
(\hspmw{k-1}{\g})_{\RB_3^1\RY_3}^{\RT} \setminus 
(\hspmw{k-1}{\g})_{\RB_3^1\RY_3}^{\RT_3} \\
& = \bigl\{ ((\bp,M) \mid \bm) \in (\hspmw{k-1}{\g})_{\RB_3^1\RY_3}^{\RT} \mid 
\iota(\bm) \notin \bp_{(*,k)} \text{ or } \kappa(\bp) \succ \iota(\bm) \bigr\}, 
\end{align*}
and
\begin{align*}
(\hspmw{k-1}{\g})_{\spadesuit \RY_3}^{\RT_1} & :=
\bigl\{ ((\bp,M) \mid \bm) \in (\hspmw{k-1}{\g})_{\spadesuit \RY_3}^{\RT} \mid 
\iota(\bm) \in \bp_{(*,k)} \bigr\}, \\[3mm]
(\hspmw{k-1}{\g})_{\spadesuit \RY_3}^{\RT_2} & :=
(\hspmw{k-1}{\g})_{\spadesuit \RY_3}^{\RT} \setminus 
(\hspmw{k-1}{\g})_{\spadesuit \RY_3}^{\RT_1} \\
& = \bigl\{ ((\bp,M) \mid \bm) \in (\hspmw{k-1}{\g})_{\spadesuit \RY_3}^{\RT} \mid 
\iota(\bm) \notin \bp_{(*,k)} \bigr\}
\end{align*}
for $\spadesuit \in \{ \RB_{2},\,\RB_{3}^{2} \}$. 
Also, we define
\begin{align}
\BB_1 & :=
  (\hspmw{k-1}{\g})_{\RB_{2}\RY_{3}}^{\RT_{1}} \sqcup 
  (\hspmw{k-1}{\g})_{\RB_{3}^{1}\RY_{3}}^{\RT_{3}} \sqcup 
  (\hspmw{k-1}{\g})_{\RB_{3}^{2}\RY_{3}}^{\RT_{1}}, \label{eq:BB1} \\[1mm]
\BB_2 & :=
  (\hspmw{k-1}{\g})_{\RB_{2}\RY_{3}}^{\RT_{2}} \sqcup 
  (\hspmw{k-1}{\g})_{\RB_{3}^{1}\RY_{3}}^{\RT_{4}} \sqcup 
  (\hspmw{k-1}{\g})_{\RB_{3}^{2}\RY_{3}}^{\RT_{2}}, \label{eq:BB2} \\[1mm] 
\BB_3 & := 
(\hspmw{k-1}{\g})_{\RB_{2} \RY_3}^{\RU} \sqcup (\hspmw{k-1}{\g})_{\RB_{3} \RY_3}^{\RU}. \label{eq:BB3}
\end{align}
We have 
\begin{align}
(\hspmw{k-1}{\g})_{\RB_2\RY} \sqcup (\hspmw{k-1}{\g})_{\RB_3\RY} = \, &
\BB_1 \sqcup \BB_2 
\sqcup \overbrace{(\hspmw{k-1}{\g})_{\RB_{3}^{1} \RY_1} \sqcup \BB_3}^{=: \, \BB_4} \nonumber \\ 
& \sqcup (\hspmw{k-1}{\g})_{\RB_{2} \emptyset} \sqcup (\hspmw{k-1}{\g})_{\RB_{3}^2 \emptyset} \nonumber \\
& \sqcup (\hspmw{k-1}{\g})_{\RB_{2} \RY_2} \sqcup (\hspmw{k-1}{\g})_{\RB_{3}^{2} \RY_2}. \label{eq:decB23}
\end{align}

Finally, we see from \eqref{eq:PAB}, \eqref{eq:A123}, \eqref{eq:decB3} that
\begin{align} 
\hspw{k-1}{p-1} = \, & 
\overbrace{
(\hspw{k-1}{p-1})_{\RA_1} \sqcup 
(\hspw{k-1}{p-1})_{\RA_2} \sqcup 
(\hspw{k-1}{p-1})_{\RA_3}}^{=\,(\hspw{k-1}{p-1})_{\RA}} \nonumber \\
& \sqcup 
(\hspw{k-1}{p-1})_{\RB_1} \sqcup 
(\hspw{k-1}{p-1})_{\RB_{2}} \sqcup
\underbrace{%
(\hspw{k-1}{p-1})_{\RB_{3}^{1}} \sqcup
(\hspw{k-1}{p-1})_{\RB_{3}^{2}} }_{ =\,(\hspw{k-1}{p-1})_{\RB_{3}} };  \label{eq:A123B23}
\end{align}
recall from Remark~\ref{rem:empty} that 
$(\hspw{k-1}{p-1})_{\spadesuit}$ is identified with 
$(\hspmw{k-1}{p-1})_{\spadesuit \emptyset} \subset (\hspmw{k-1}{p-1})_{\emptyset} 
\subset \hspmw{k-1}{p-1}$ 
for each $\spadesuit \in \{ \RA_1,\,\RA_2,\,\RA_3,\,\RB_1,\,
\RB_{2},\,\RB_{3}^{1},\,\RB_{3}^{2} \}$. 

%
\subsection{Matching (2).}
\label{subsec:mat2}
Here we list the conditions 
needed in this subsection and in Section~\ref{sec:prfmat2}: 
\begin{itemize}
\item For $((\bp,M) \mid \bm) \in \hspmw{k-1}{\g}$ with $g \in \{p-1,p\}$, 
%
%
\begin{equation} \label{eq:cond2}
\begin{cases}
(\RA) & n_{(k-1,k)}(\bp)=0, \\
(\RA_1) & n_{(k-1,k)}(\bp)=0 \text{ and } \bp_{(*,k)} = \emptyset, \\
(\RA_2) & n_{(k-1,k)}(\bp)=0, \, \bp_{(*,k)} \ne \emptyset, \text{ and } \kappa(\bp) \not\in M, \\
(\RA_3) & n_{(k-1,k)}(\bp)=0, \, \bp_{(*,k)} \ne \emptyset, \text{ and } \kappa(\bp) \in M, \\
(\RB_2) & \text{$(k-1,k) \in M$ and $\kappa(\bp)=(k-1,k)$}, \\
(\RB_3) & \text{$(k-1,k) \in M$ and $\kappa(\bp) \ne (k-1,k)$}, \\
(\RB_3^1) & (k-1,k) \in M, \, \kappa(\bp) \ne (k-1,k), \, \kappa(\bp) \notin M, \\
(\RB_3^2) & (k-1,k) \in M, \, \kappa(\bp) \ne (k-1,k), \, \kappa(\bp) \in M, \\
(\RD_2) & \text{see Definition~\ref{dfn:D**}; {\bf Algorithm $(\bp_{(*,k-1)}:(k-1,k))$} ends with} \\
& \text{a directed path of the form \eqref{eq:dec12-3}}, \\
(\RT) & \bp_{(*,k)} \cap \bm_{(*,k)} \ne \emptyset, \\
(\RU) & \bp_{(*,k)} \cap \bm_{(*,k)} = \emptyset, \\
(\RT_1) & \bp_{(*,k)} \cap \bm_{(*,k)} \ne \emptyset \text{ and } \iota(\bm) \in \bp_{(*,k)}, \\
(\RT_2) & \bp_{(*,k)} \cap \bm_{(*,k)} \ne \emptyset \text{ and } \iota(\bm) \notin \bp_{(*,k)}, \\
(\RT_3) & \bp_{(*,k)} \cap \bm_{(*,k)} \ne \emptyset, \, 
          \iota(\bm) \in \bp_{(*,k)}, \text{ and } \kappa(\bp) \prec \iota(\bm), \\
(\RT_4) & \bp_{(*,k)} \cap \bm_{(*,k)} \ne \emptyset \text{ and } 
          (\iota(\bm) \notin \bp_{(*,k)} \text{ or } \kappa(\bp) \succ \iota(\bm)), \\
(\RY) & \iota(\bm) \ne (k-1,k), \\
(\RY_1) & \iota(\bm) \ne (k-1,k) \text{ and } \bm_{(*,k)} = \emptyset, \\
(\RY_2) & \iota(\bm) \ne (k-1,k),\, \bm_{(*,k)} = \emptyset, \text{ and } \bm_{(k,*)} \ne \emptyset, \\
(\RY_3) & \iota(\bm) \ne (k-1,k) \text{ and } \bm_{(*,k)} \ne \emptyset. 
\end{cases}
\end{equation}
\end{itemize}
We set
\begin{equation} \label{eq:BA}
\BA:=
(\hspmw{k-1}{\g})_{\RA_1 \RY_3} \sqcup (\hspmw{k-1}{\g})_{\RA_2 \RY} \sqcup (\hspmw{k-1}{\g})_{\RA_3 \RY_3}.
\end{equation}
Also, recall that
\begin{align*}
\BB_1 & =
  (\hspmw{k-1}{\g})_{\RB_{2}\RY_{3}}^{\RT_{1}} \sqcup 
  (\hspmw{k-1}{\g})_{\RB_{3}^{1}\RY_{3}}^{\RT_{3}} \sqcup 
  (\hspmw{k-1}{\g})_{\RB_{3}^{2}\RY_{3}}^{\RT_{1}}, \\[1mm]
\BB_2 & =
  (\hspmw{k-1}{\g})_{\RB_{2}\RY_{3}}^{\RT_{2}} \sqcup 
  (\hspmw{k-1}{\g})_{\RB_{3}^{1}\RY_{3}}^{\RT_{4}} \sqcup 
  (\hspmw{k-1}{\g})_{\RB_{3}^{2}\RY_{3}}^{\RT_{2}}, \\[1mm] 
\BB_3 & = 
(\hspmw{k-1}{\g})_{\RB_{2} \RY_3}^{\RU} \sqcup (\hspmw{k-1}{\g})_{\RB_{3} \RY_3}^{\RU}, \\[1mm]
\BB_4 & =(\hspmw{k-1}{\g})_{\RB_{3}^{1} \RY_1} \sqcup \BB_3. 
\end{align*}
%
%
\begin{prop}[to be proved in Section~\ref{sec:prfmat2}] \label{prop:mat2a}
The following equalities hold in $\K_{\infty}'$\,{\rm:} 
\begin{equation*}
(\CSF{k-1}{\g})_{\BA} = (\CSF{k-1}{\g})_{\BB_2} = (\CSF{k-1}{\g})_{\BB_4} = 0, \qquad 
(\CSF{k-1}{\g})_{\BB_1} = Q_{k-1} (\CSF{k-2}{\g-1})_{\RD_2\RY}. 
\end{equation*}
\end{prop}

From Proposition~\ref{prop:mat2a}, together with \eqref{eq:RARY} and \eqref{eq:decB23}, 
we deduce that
\begin{align}
&  (\CSF{k-1}{\g})_{\RA\RY} + (\CSF{k-1}{\g})_{\RB_2\RY} + 
   (\CSF{k-1}{\g})_{\RB_3\RY} - Q_{k-1} (\CSF{k-2}{\g-1})_{\RD_2\RY} \nonumber \\
& = (\CSF{k-1}{\g})_{\RA_1 \RY_2} + (\CSF{k-1}{\g})_{\RA_3 \RY_2}+ 
    (\CSF{k-1}{\g})_{\RB_{2}\RY_{2}} + (\CSF{k-1}{\g})_{\RB_{3}^{2}\RY_{2}} \nonumber \\
& \phantom{=}
  + \sum_{ \spadesuit \in \{ \RA_1,\,\RA_3,\,\RB_{2},\,\RB_{3}^{2} \} }
    (\CSF{k-1}{\g})_{\spadesuit \emptyset}. \label{eq:mat2a}
\end{align}
Also, it follows from \eqref{eq:A123B23} and the comment following it that 
\begin{equation} \label{eq:mat2b}
(\CSF{k-1}{p-1})_{\emptyset}
= 
\sum_{\spadesuit \in \{ \RA_1,\,\RA_2,\,\RA_3,\,\RB_{1},\,\RB_{2},\,\RB_{3}^{1},\,\RB_{3}^{2} \} } 
(\CSF{k-1}{p-1})_{\spadesuit \emptyset}.
\end{equation}
Putting together \eqref{eq:mat2a}, \eqref{eq:mat2b}, and \eqref{eq:Ind2}, we obtain
\begin{align}
\FG_{w}^{\Q} \G{k}{p} & = 
(\CSF{k-1}{p})_{\RA_1 \RY_2} + (\CSF{k-1}{p})_{\RA_3 \RY_2}+ 
(\CSF{k-1}{p})_{\RB_{2}\RY_{2}} + (\CSF{k-1}{p})_{\RB_{3}^{2}\RY_{2}} \nonumber \\[2mm]
& \phantom{=}
  + \sum_{\spadesuit \in \{ \RA_1,\,\RA_3,\,\RB_{2},\,\RB_{3}^{2} \} } 
    (\CSF{k-1}{p})_{\spadesuit \emptyset} \nonumber \\[2mm]
& \phantom{=}
 - (\CSF{k-1}{p-1})_{\RA_1 \RY_2} - (\CSF{k-1}{p-1})_{\RA_3 \RY_2} 
 - (\CSF{k-1}{p-1})_{\RB_{2}\RY_{2}} - (\CSF{k-1}{p-1})_{\RB_{3}^{2}\RY_{2}} \nonumber \\[2mm]
& \phantom{=}
  + \sum_{\spadesuit \in \{ \RA_2,\,\RB_1,\,\RB_{3}^{1} \} }
    (\CSF{k-1}{p-1})_{\spadesuit \emptyset}. \label{eq:Ind3}
\end{align}
We set
\begin{align*}
& (\hspmw{k-1}{\g})_{\RE}: = (\hspmw{k-1}{\g})_{\RA_3 \RY_2} 
  \sqcup (\hspmw{k-1}{\g})_{\RB_{2} \RY_2} \sqcup (\hspmw{k-1}{\g})_{\RB_{3}^{2} \RY_2}
  \qquad \text{for $g \in \{p-1,p\}$}, \\[3mm]
& (\hspmw{k-1}{p-1})_{\RF}: = \bigsqcup_{\spadesuit \in \{ \RA_2,\,\RB_1,\,\RB_{3}^{1} \} }
  (\hspmw{k-1}{p-1})_{\spadesuit \emptyset}, \qquad 
  (\hspmw{k-1}{p})_{\RG}: = \bigsqcup_{\spadesuit \in \{ \RA_3,\,\RB_{2},\,\RB_{3}^{2} \} }
  (\hspmw{k-1}{p})_{\spadesuit \emptyset}. 
\end{align*}
Then, by \eqref{eq:Ind3}, we have
%
%
\begin{equation} \label{eq:Ind4}
\begin{split}
\FG_{w}^{\Q} \G{k}{p} = \, & (\CSF{k-1}{p})_{\RA_1\RY_2} + (\CSF{k-1}{p})_{\RE} + 
(\CSF{k-1}{p})_{\RA_1 \emptyset} + (\CSF{k-1}{p})_{\RG} \\
& - (\CSF{k-1}{p-1})_{\RA_1\RY_2} - (\CSF{k-1}{p-1})_{\RE} + (\CSF{k-1}{p-1})_{\RF}.
\end{split}
\end{equation}

\begin{rem} \label{rem:can2} \mbox{}
\begin{enu}
\item Let $g \in \{p-1,p\}$. 
The set $(\hspmw{k-1}{\g})_{\RE}$ is identical to the subset of $\hspmw{k-1}{\g}$ 
consisting of the elements $\bq = ((\bp,M) \mid \bm)$ satisfying the condition that 
\begin{equation} \label{eq:cE}
(\RE) \quad \bp_{(*,k)} \ne \emptyset, \quad 
\kappa(\bp) \in M, \quad 
\bm_{(*,k)}=\emptyset, \quad \text{and} \quad 
\bm_{(k,*)} \ne \emptyset.
\end{equation} 

\item The set $(\hspmw{k-1}{p-1})_{\RF}$ is identical to the subset of $\hspmw{k-1}{p-1}$ 
consisting of the elements $\bq = ((\bp,M) \mid \bm)$ satisfying the condition that
\begin{equation} \label{eq:cF}
(\RF) \quad \bm=\emptyset, \quad 
\bp_{(*,k)} \ne \emptyset, \quad \text{and} \quad 
\kappa(\bp) \notin M. 
\end{equation} 

\item The set $(\hspmw{k-1}{p})_{\RG}$ is identical to the subset of $\hspmw{k-1}{p}$ 
consisting of the elements $\bq = ((\bp,M) \mid \bm)$ satisfying the condition that
\begin{equation} \label{eq:cG}
(\RG) \quad \bm=\emptyset, \quad 
\bp_{(*,k)} \ne \emptyset, \quad \text{and} \quad 
\kappa(\bp) \in M. 
\end{equation} 
\end{enu}
\end{rem}

%
\subsection{Decomposition into subsets (3).}
\label{subsec:decomp3}

Let $(\hspmw{k-1}{p-1})_{\RF_1}$ (resp., $(\hspmw{k-1}{p-1})_{\RF_2}$) 
be the subset of $\hspmw{k-1}{p-1}$ consisting of the elements 
$\bq = ((\bp,M) \mid \bm)$ satisfying 
the following condition $\RF_1$ (resp., $\RF_2$): 
\begin{equation} \label{eq:cFn}
\begin{cases}
(\RF_1) & \underbrace{\bm=\emptyset, \, \bp_{(*,k)} \ne \emptyset,\,
\kappa(\bp) \notin M}_{\text{condition $\RF$; see \eqref{eq:cF}}}, \text{ and } n_{(a,*)}(\bp) = 1, \\[10mm]
(\RF_2) & \bm=\emptyset, \, \bp_{(*,k)} \ne \emptyset,\,
\kappa(\bp) \notin M, \text{ and } n_{(a,*)}(\bp) \ge 2, 
\end{cases}
\end{equation}
where $\kappa(\bp)=(a,k)$; 
note that since $\bp_{(*,k)} \ne \emptyset$, 
it follows that $\kappa(\bp)=(a,k)$ for some $1 \le a \le k-1$. 

%
\begin{dfn} \label{dfn:kap1}
Let $\bq = ((\bp,M) \mid \emptyset) \in (\hspmw{k-1}{p-1})_{\RF_2}$. 
We define $\ip(\bp) \ge 0$ and $f_{\ip}=f_{\ip}(\bp) \ge k$ 
for $0 \le \ip \le \ip(\bp)$ 
by the following algorithm: 
\begin{enu}
\item[(1)] Set $f_{0}:=k$; 
note that $\bp_{(*,f_{0})}=\bp_{(*,k)} \ne \emptyset$.

\item[(2)] Assume that we have defined $f_{\ip}$ in such a way that 
$\bp_{(*,f_{\ip})} \ne \emptyset$ for $\ip \ge 0$. 
We write the final label of $\bp_{(*,f_{\ip})}$ as 
$(a,f_{\ip})$, with $1 \le a \le k-1$. 

\begin{enu}
\item[(2a)] If the set $\{ f > f_{\ip} \mid (a,f) \in \bp \}$ 
is empty, then we set $\ip(\bp):=\ip$, and end the algorithm. 

\item[(2b)] If the set $\{ f > f_{\ip} \mid (a,f) \in \bp \}$ 
is nonempty, then we define $f_{\ip+1}$ to be the minimum element of this set, 
and go back to the beginning of (2). 
\end{enu}
\end{enu}
\noindent
Then we define $\kappa'(\bp)$ to be the final label of 
$\bp_{(*,f_{\ip(\bp)})}$.
\end{dfn}

\begin{ex} \label{ex:kap1}
Recall from Example~\ref{ex:pieri} that 
\begin{equation*}
\bp_{3}=(w = 52173846; (1,8)_{\SB}, (5,7)_{\SB}, (2,7)_{\SB}, (3,7)_{\SB}, (4,6)_{\SB}, (1,6)_{\SB}, (5,6)_{\SB}) \in \sfp{5}{}.
\end{equation*}
It is easily verified that 
$M=\bigl\{ (1,8), (5,7), (3,7), (4,6) \bigr\}$ is a $4$-marking of $\bp_{3}$, 
and $\bq = ((\bp_{3},M) \mid \emptyset) \in (\hspmw{5}{4})_{\RF_2}$, with $k = 6$ and $p = 5$. 
We have $f_{0} = k = 6$, and $a = 5$. Since 
$\{ f > f_{0} = 6 \mid (5,f) \in \bp \} = \{ 7 \}$, we set $f_{1}:=7$. 
The final label of $(\bp_{3})_{(*,7)}$ is $(3,7)$. 
Since $\{ f > f_{1} = 7 \mid (3,f) \in \bp \} = \emptyset$, 
we end the algorithm here, and obtain $\sigma(\bp_{3})=1$, $\kappa'(\bp_{3}) := (3,7)$; 
note that $\bq$ satisfies condition $\RF_{2}^{1}$ below 
since $\kappa'(\bp_{3}) := (3,7) \in M$. 
\end{ex}

We set
\begin{align*}
(\hspmw{k-1}{p-1})_{ \begin{subarray}{l} \RF_{2} \\ \CRFa \end{subarray} } & :=
\bigl\{ ((\bp,M) \mid \emptyset) \in (\hspmw{k-1}{p-1})_{\RF_2} \mid 
\kappa'(\bp) \in M \bigr\}, \\
(\hspmw{k-1}{p-1})_{ \begin{subarray}{l} \RF_{2} \\ \CRFb \end{subarray} } & :=
\bigl\{ ((\bp,M) \mid \emptyset) \in (\hspmw{k-1}{p-1})_{\RF_2} \mid 
\kappa'(\bp) \notin M \bigr\}; 
\end{align*}
here we consider the following conditions on 
$\bq = ((\bp,M) \mid \emptyset) \in \hspmw{k-1}{p-1}$: 
\begin{equation} \label{eq:cF2n}
\begin{cases}
(\RF_2^1) & 
\underbrace{ \underbrace{ \bm=\emptyset, \, \bp_{(*,k)} \ne \emptyset,\,
\kappa(\bp) \notin M}_{%
\text{condition $\RF$; see \eqref{eq:cF}}}, \, 
n_{(a,*)}(\bp) \ge 2}_{\text{condition $\RF_2$; see \eqref{eq:cFn}}}, \text{ and } 
\kappa'(\bp) \in M,  \\[15mm]
(\RF_2^2) & 
\bm=\emptyset, \, \bp_{(*,k)} \ne \emptyset,\,
\kappa(\bp) \notin M, \, n_{(a,*)}(\bp) \ge 2, \text{ and } 
\kappa'(\bp) \not\in M, 
\end{cases}
\end{equation}
where $\kappa(\bp)=(a,k)$ for some $1 \le a \le k-1$. We have 
\begin{equation} \label{eq:RF}
(\hspmw{k-1}{p-1})_{\RF} = 
(\hspmw{k-1}{p-1})_{\RF_1} \sqcup 
\underbrace{(\hspmw{k-1}{p-1})_{\RF_2^1} \sqcup 
(\hspmw{k-1}{p-1})_{\RF_2^2}}_{ = \, (\hspmw{k-1}{p-1})_{\RF_2}}. 
\end{equation}

Next, we set
\begin{align*}
& (\hspw{k}{p})_{ \CRR } := 
\bigl\{ (\bp,M) \in \hspw{k}{p} \mid n_{(k,*)}(\bp)=0 \bigr\}, \\
& (\hspw{k}{p})_{ \CRS } := 
\bigl\{ (\bp,M) \in \hspw{k}{p} \mid n_{(k,*)}(\bp) \ge 1 \bigr\}; 
\end{align*}
here we consider the following conditions on $(\bp,M) \in \hspw{k}{p}$: 
\begin{equation} \label{eq:cRS}
\begin{cases}
(\RR) & n_{(k,*)}(\bp)=0, \\
(\RS) & n_{(k,*)}(\bp) \ge 1. 
\end{cases}
\end{equation}
For $(\bp,M) \in (\hspw{k}{p})_{\RS} = (\hspw{k}{p})_{ \CRS }$, 
we define $b(\bp) := \max \{b \ge k+1 \mid (k,b) \in \bp\bigr\}$. We set
\begin{align*}
& (\hspw{k}{p})_{
 \begin{subarray}{c}
 \CRS \\
 \CRSa
 \end{subarray} } := 
\bigl\{ (\bp,M) \in (\hspw{k}{p})_{\RS} \mid (k,b(\bp)) \in M \bigr\}, \\ 
& (\hspw{k}{p})_{
  \begin{subarray}{c}
  \CRS \\
  \CRSb
  \end{subarray} } := 
\bigl\{ (\bp,M) \in (\hspw{k}{p})_{\RS} \mid (k,b(\bp)) \not\in M \bigr\}; 
\end{align*}
here we consider the following conditions on 
$(\bp,M) \in \hspw{k}{p}$: 
\begin{equation} \label{eq:cSn}
\begin{cases}
(\RS_1) & n_{(k,*)}(\bp) \ge 1 \text{ and } (k,b(\bp)) \in M, \\
(\RS_2) & n_{(k,*)}(\bp) \ge 1 \text{ and } (k,b(\bp)) \not\in M. 
\end{cases}
\end{equation}
Let $(\hspw{k}{p})_{\RS_1^1}$ and $(\hspw{k}{p})_{\RS_1^2}$ denote 
the subsets of $\hspw{k}{p}$ consisting of the elements $(\bp,M)$ 
satisfying the following conditions ($\RS_1^1$) and ($\RS_1^2$), respectively: 
\begin{equation} \label{eq:cS1n}
\begin{cases}
(\RS_1^1) & \underbrace{
n_{(k,*)}(\bp) \ge 1,\, (k,b(\bp)) \in M }_{%
\text{condition $\RS_1$; see \eqref{eq:cSn}} }, \text{ and 
$(k,b(\bp))$ is the final label of $\bp_{(\ast,b(\bp))}$},  \\[7mm]
(\RS_1^2) & n_{(k,*)}(\bp) \ge 1, (k,b(\bp)) \in M, \text{ and 
$(k,b(\bp))$ is not the final label of $\bp_{(\ast,b(\bp))}$}. 
\end{cases}
\end{equation}

\begin{ex} \label{ex:RS}
Let $k=3$ and $p=2$. 
We use the notation and setting of Example~\ref{ex2}. 
It is easily verified that
\begin{equation*}
\begin{array}{c|c|c}
\bp & \Mark_{2}(\bp) & \\ \hline\hline
(w \,;\, (3,6)_{\SB}, (1,5)_{\SB}) & \{(3,6),(1,5)\} & \RS_{1}^{1} \\ \hline
(w \,;\, (3,6)_{\SB}, (1,5)_{\SB}, (2,5)_{\SB}) & \{(3,6),(1,5)\} & \RS_{1}^{1} \\ \hline
(w \,;\, (3,6)_{\SB}, (1,5)_{\SB}, (2,5)_{\SB}) & \{(3,6),(2,5)\} & \RS_{1}^{1} \\ \hline
(w \,;\, (3,6)_{\SB}, (1,5)_{\SB}, (2,5)_{\SB}, (3,4)_{\SQ}) & \{(3,6),(1,5)\} & \RS_{1}^{1} \\ \hline
(w \,;\, (3,6)_{\SB}, (1,5)_{\SB}, (2,5)_{\SB}, (3,4)_{\SQ}) & \{(3,6),(2,5)\} & \RS_{1}^{1} \\ \hline

(w \,;\, (3,6)_{\SB}, (1,5)_{\SB}, (3,4)_{\SQ}) & \{(3,6),(1,5)\} & \RS_{1}^{1} \\ \hline
(w \,;\, (3,6)_{\SB}, (2,5)_{\SB}) & \{(3,6),(2,5)\} & \RS_{1}^{1} \\ \hline
(w \,;\, (3,6)_{\SB}, (2,5)_{\SB}, (3,4)_{\SQ}) & \{(3,6),(2,5)\} & \RS_{1}^{1} \\ \hline
(w \,;\, (1,5)_{\SB}, (2,5)_{\SB}) & \{(1,5),(2,5)\} & \RR \\ \hline
(w \,;\, (1,5)_{\SB}, (2,5)_{\SB}, (3,4)_{\SQ}) & \{(1,5),(2,5)\} & \RS_{2} \\ \hline
(w \,;\, (1,5)_{\SB}, (2,5)_{\SB}, (3,4)_{\SQ}) & \{(1,5),(3,4)\} & \RS_{1}^{1} \\ \hline
(w \,;\, (1,5)_{\SB}, (2,5)_{\SB}, (3,4)_{\SQ}, (1,4)_{\SB}) & \{(1,5),(3,4)\} & \RS_{1}^{2} \\ \hline
(w \,;\, (1,5)_{\SB}, (2,5)_{\SB}, (3,4)_{\SQ}, (1,4)_{\SB},(2,4)_{\SB}) & \{(1,5),(3,4)\} & \RS_{1}^{2} \\ \hline
(w \,;\, (1,5)_{\SB}, (2,5)_{\SB}, (3,4)_{\SQ}, (2,4)_{\SB}) & \{(1,5),(3,4)\} & \RS_{1}^{2} \\ \hline
(w \,;\, (1,5)_{\SB}, (3,4)_{\SQ}) & \{(1,5),(3,4)\} & \RS_{1}^{1} \\ \hline
(w \,;\, (1,5)_{\SB}, (3,4)_{\SQ}, (1,4)_{\SB}) & \{(1,5),(3,4)\} & \RS_{1}^{2} \\ \hline
(w \,;\, (1,5)_{\SB}, (3,4)_{\SQ}, (1,4)_{\SB}, (2,4)_{\SB}) & \{(1,5),(3,4)\} & \RS_{1}^{2} \\ \hline
(w \,;\, (1,5)_{\SB}, (3,4)_{\SQ}, (2,4)_{\SB}) & \{(1,5),(3,4)\} & \RS_{1}^{2} \\ \hline
(w \,;\, (2,5)_{\SB}, (3,4)_{\SQ}) & \{(2,5),(3,4)\} & \RS_{1}^{1} \\ \hline
(w \,;\, (2,5)_{\SB}, (3,4)_{\SQ}, (2,4)_{\SB}) & \{(2,5),(3,4)\} & \RS_{1}^{2} \\ \hline
(w \,;\, (3,4)_{\SQ}, (1,4)_{\SB}) & \{(3,4),(1,4)\} & \RS_{1}^{2} \\ \hline
(w \,;\, (3,4)_{\SQ}, (1,4)_{\SB}, (2,4)_{\SB}) & \{(3,4),(1,4)\} & \RS_{1}^{2} \\ \hline
(w \,;\, (3,4)_{\SQ}, (2,4)_{\SB}) & \{(3,4),(2,4)\} & \RS_{1}^{2}
\end{array}
\end{equation*}
\end{ex}

%
\begin{dfn} \label{dfn:kap2}
For $(\bp,M) \in (\hspw{k}{p})_{\RS_1^2}$, 
we define $\jp(\bp) \ge 0$ and $f_{\jp}'=f_{\jp}'(\bp) \ge k+1$ for $0 \le \jp \le \jp(\bp)$ 
by the following algorithm: 
\begin{enu}
\item[(1)'] Set $f'_{0}:=b(\bp)$; 
note that $\bp_{(*,f'_{0})}=\bp_{(*,b(\bp))} \ne \emptyset$.

\item[(2)'] Assume that we have defined $f'_{\jp}$ in such a way that 
$\bp_{(*,f'_{\jp})} \ne \emptyset$ for $\jp \ge 0$. 
We write the final label of $\bp_{(*,f'_{\jp})}$ as 
$(a,f'_{\jp})$, with $1 \le a \le k-1$. 

\begin{enu}
\item[(2a)'] If the set $\{ f' > f'_{\jp} \mid (a,f') \in \bp \}$ 
is empty, then we set $\jp(\bp):=\jp$, and end the algorithm. 

\item[(2b)'] If the set $\{ f' > f'_{\jp} \mid (a,f') \in \bp \}$ 
is nonempty, then we define $f'_{\jp+1}$ to be the minimum element of this set, 
and go back to the beginning of (2)'. 
\end{enu}
\end{enu}
\noindent
Then we define $\kappa''(\bp)$ to be the final label of 
$\bp_{(*,f'_{\jp(\bp)})}$. 
\end{dfn}

\begin{ex} \label{ex:S12*}
Let $k=3$ and $p=2$. 
We use the notation and setting of Examples~\ref{ex2} and \ref{ex:RS}. 
For conditions $\RS_{1}^{\ta}$ and $\RS_{1}^{\tb}$, see \eqref{eq:cS12a} below. 
It is easily verified that
\begin{equation*}
\begin{array}{c|c|c|c}
\bp & \Mark_{2}(\bp) & \kappa''(\bp) & \\ \hline\hline
(w \,;\, (1,5)_{\SB}, (2,5)_{\SB}, (3,4)_{\SQ}, (1,4)_{\SB}) 
& \{(1,5),(3,4)\} & (1,5) & \RS_{1}^{\ta} \\ \hline
(w \,;\, (1,5)_{\SB}, (2,5)_{\SB}, (3,4)_{\SQ}, (1,4)_{\SB},(2,4)_{\SB}) 
& \{(1,5),(3,4)\} & (2,5) & \RS_{1}^{\tb} \\ \hline
(w \,;\, (1,5)_{\SB}, (2,5)_{\SB}, (3,4)_{\SQ}, (2,4)_{\SB}) 
& \{(1,5),(3,4)\} & (2,5) & \RS_{1}^{\tb} \\ \hline
(w \,;\, (1,5)_{\SB}, (3,4)_{\SQ}, (1,4)_{\SB}) 
& \{(1,5),(3,4)\} & (1,5) & \RS_{1}^{\ta} \\ \hline
(w \,;\, (1,5)_{\SB}, (3,4)_{\SQ}, (1,4)_{\SB}, (2,4)_{\SB}) 
& \{(1,5),(3,4)\} & (2,4) & \RS_{1}^{\tb} \\ \hline
(w \,;\, (1,5)_{\SB}, (3,4)_{\SQ}, (2,4)_{\SB}) 
& \{(1,5),(3,4)\} & (2,4) & \RS_{1}^{\tb} \\ \hline
(w \,;\, (2,5)_{\SB}, (3,4)_{\SQ}, (2,4)_{\SB}) 
& \{(2,5),(3,4)\} & (2,5) & \RS_{1}^{\ta} \\ \hline
(w \,;\, (3,4)_{\SQ}, (1,4)_{\SB}) 
& \{(3,4),(1,4)\} & (1,4) & \RS_{1}^{\ta} \\ \hline
(w \,;\, (3,4)_{\SQ}, (1,4)_{\SB}, (2,4)_{\SB}) 
& \{(3,4),(1,4)\} & (2,4) & \RS_{1}^{\tb} \\ \hline
(w \,;\, (3,4)_{\SQ}, (2,4)_{\SB}) 
& \{(3,4),(2,4)\} & (2,4) & \RS_{1}^{\ta}
\end{array}
\end{equation*}
\end{ex}

We set 
\begin{align*}
(\hspw{k}{p})_{\begin{subarray}{l} \RS_1^2 \\ \CRSaba \end{subarray}} & :=
\bigl\{ (\bp,M) \in (\hspw{k}{p})_{\RS_1^2} \mid 
\kappa''(\bp) \in M \bigr\}, \\
(\hspw{k}{p})_{\begin{subarray}{l} \RS_1^2 \\ \CRSabb \end{subarray}} & :=
\bigl\{ (\bp,M) \in (\hspw{k}{p})_{\RS_1^2} \mid 
\kappa''(\bp) \notin M \bigr\}; 
\end{align*}
here we consider the following conditions on $(\bp,M) \in \hspw{k}{p}$ 
(for condition $\RS_1^2$, see \eqref{eq:cS1n}): 
\begin{equation} \label{eq:cS12a}
\begin{cases}
(\RS_1^{\ta}) & (\RS_1^2) \text{ and } \kappa''(\bp) \in M, \\
(\RS_1^{\tb}) & (\RS_1^2) \text{ and } \kappa''(\bp) \not\in M. 
\end{cases}
\end{equation}
Observe that 
\begin{equation} \label{eq:RRS}
\hspw{k}{p} =  (\hspw{k}{p})_{\RR} \sqcup (\hspw{k}{p})_{\RS} =
 (\hspw{k}{p})_{\RR} \sqcup 
 \overbrace{
 (\hspw{k}{p})_{\RS_1^1} \sqcup 
 \underbrace{ 
 (\hspw{k}{p})_{\RS_1^{\ta}} \sqcup 
 (\hspw{k}{p})_{\RS_1^{\tb}} }_{ = \, (\hspw{k}{p})_{\RS_1^2} } }^{ = \, (\hspw{k}{p})_{\RS_1} }
 \sqcup 
 (\hspw{k}{p})_{\RS_2}.
\end{equation}

For $\bq=(\bp,M) \in \hspw{k}{p}$, we set 
\begin{equation} \label{eq:bEkp}
\bE{k}{p}{\bq}:=(-1)^{\ell(\bp)-p}\Q(\bp)\FG_{\ed(\bp)}^{\Q}; 
\end{equation}
see also Remark~\ref{rem:empty}. Then we define
\begin{equation*}
\CSE{k}{p}:=\sum_{\bq \in \hspw{k}{p}} \bE{k}{p}{\bq}, \qquad 
(\CSE{k}{p})_{\spadesuit}:=\sum_{\bq \in (\hspw{k}{p})_{\spadesuit}} \bE{k}{p}{\bq} 
\text{ for } \spadesuit \in \{\RR, \RS, \RS_1, \RS_2, \RS_1^1, \RS_1^2, \RS_1^{\ta}, \RS_1^{\tb} \}. 
\end{equation*}
We have 
%
%
\begin{equation} \label{eq:Ind5}
\CSE{k}{p} = 
(\CSE{k}{p})_{\RR} + 
(\CSE{k}{p})_{\RS_1^1} + 
(\CSE{k}{p})_{\RS_1^{\ta}} + 
(\CSE{k}{p})_{\RS_1^{\tb}} + 
(\CSE{k}{p})_{\RS_2}. 
\end{equation}

%
\subsection{Matching (3): Finishing the proof of Theorem~\ref{thm:pieri}.}
\label{subsec:mat3}
Here we list the conditions 
needed in Proposition~\ref{prop:mat3} below and in its proof: 
\begin{itemize}
\item For $((\bp,M) \mid \bm) \in \hspmw{k-1}{p}$, 
\begin{equation} \label{eq:cond3a}
\begin{cases}
(\RA_1) & n_{(k-1,k)}(\bp)=0 \text{ and } \bp_{(*,k)} = \emptyset, \\
(\RY_2) & \iota(\bm) \ne (k-1,k),\, \bm_{(*,k)} = \emptyset, \text{ and } \bm_{(k,*)} \ne \emptyset, \\
(\RE) & \bp_{(*,k)} \ne \emptyset,\,\kappa(\bp) \in M,\,\bm_{(*,k)}=\emptyset, 
        \text{ and } \bm_{(k,*)} \ne \emptyset, \\
(\RG) & \bm = \emptyset, \, \bp_{(*,k)} \ne \emptyset, \text{ and } \kappa(\bp) \in M. 
\end{cases}
\end{equation}

\item For $((\bp,M) \mid \bm) \in \hspmw{k-1}{p-1}$, 
\begin{equation} \label{eq:cond3b}
\begin{cases}
(\RE) & \bp_{(*,k)} \ne \emptyset,\,\kappa(\bp) \in M,\,\bm_{(*,k)}=\emptyset, 
\text{ and } \bm_{(k,*)} \ne \emptyset, \\
(\RF) & \bm=\emptyset,\,\bp_{(*,k)} \ne \emptyset, \text{ and }
\kappa(\bp) \notin M, \\
(\RF_1) & \bm=\emptyset, \, \bp_{(*,k)} \ne \emptyset,\,
\kappa(\bp) \notin M, \text{ and } n_{(a,*)}(\bp) = 1, \\
(\RF_2) & \bm=\emptyset, \, \bp_{(*,k)} \ne \emptyset,\,
\kappa(\bp) \notin M, \text{ and } n_{(a,*)}(\bp) \ge 2, \\
(\RF_2^1) & 
\bm=\emptyset, \, \bp_{(*,k)} \ne \emptyset,\,
\kappa(\bp) \notin M, \, n_{(a,*)}(\bp) \ge 2, \text{ and } 
\kappa'(\bp) \in M, \\
(\RF_2^2) & 
\bm=\emptyset, \, \bp_{(*,k)} \ne \emptyset,\,
\kappa(\bp) \notin M, \, n_{(a,*)}(\bp) \ge 2, \text{ and } 
\kappa'(\bp) \not\in M, 
\end{cases}
\end{equation}
where in conditions $\RF_1$, $\RF_2$, $\RF_2^1$, and $\RF_2^2$, 
we write $\kappa(\bp) = (a,k)$ for some $1 \le a \le k-1$. 
Also, for the definition of $\kappa'(\bp)$, see Definition~\ref{dfn:kap1}. 

\item For $(\bp, M) \in \hspw{k}{p}$, 
\begin{equation} \label{eq:cond3c}
\begin{cases}
(\RR) & n_{(k,*)}(\bp)=0, \\
(\RS) & n_{(k,*)}(\bp) \ge 1, \\
(\RS_1) & n_{(k,*)}(\bp) \ge 1 \text{ and } (k,b(\bp)) \in M, \\
(\RS_2) & n_{(k,*)}(\bp) \ge 1 \text{ and } (k,b(\bp)) \not\in M, \\
(\RS_1^1) & n_{(k,*)}(\bp) \ge 1,\, (k,b(\bp)) \in M, \text{ and } \\
          & \text{$(k,b(\bp))$ is the final label of $\bp_{(\ast,b(\bp))}$}, \\
(\RS_1^2) & n_{(k,*)}(\bp) \ge 1,\, (k,b(\bp)) \in M, \text{ and } \\
          & \text{$(k,b(\bp))$ is not the final label of $\bp_{(\ast,b(\bp))}$}, \\ 
(\RS_1^{\ta}) & (\RS_1^2) \text{ and } \kappa''(\bp) \in M, \\
(\RS_1^{\tb}) & (\RS_1^2) \text{ and } \kappa''(\bp) \not\in M, 
\end{cases}
\end{equation}
where $b(\bp) = \max \{b \ge k+1 \mid (k,b) \in \bp\bigr\}$. 
Also, for the definition of $\kappa''(\bp)$, see Definition~\ref{dfn:kap2}. 
\end{itemize}
%
%
\begin{prop}[to be proved in Section~\ref{sec:prfmat3}] \label{prop:mat3a}
The following equalities hold in $\K_{\infty}'$\,{\rm:} 
\begin{align*}
& (\CSF{k-1}{p})_{\RA_1\RY_2} = (\CSE{k}{p})_{\RS_2}, &
& (\CSF{k-1}{p})_{\RE} = (\CSE{k}{p})_{\RS_1^{\tb}} - (\CSF{k-1}{p-1})_{\RF_{2}^{2}}, \\
& (\CSF{k-1}{p})_{\RA_1 \emptyset} = (\CSE{k}{p})_{\RR}, &
& (\CSF{k-1}{p})_{\RG} = - (\CSF{k-1}{p-1})_{\RF_{1}}, \\
& (\CSF{k-1}{p-1})_{\RA_1\RY_2} = - (\CSE{k}{p})_{\RS_1^1}, &
& (\CSF{k-1}{p-1})_{\RE} = - (\CSE{k}{p})_{\RS_1^{\ta}} 
  + (\CSF{k-1}{p-1})_{\RF_{2}^{1}}. 
\end{align*}
\end{prop}

Now, we conclude that 
\begin{align*}
& \FG_{w}^{\Q}\G{k}{p} - \CSE{k}{p} \\[3mm]
& = (\CSF{k-1}{p})_{\RA_1\RY_2} + (\CSF{k-1}{p})_{\RE} + 
(\CSF{k-1}{p})_{\RA_1 \emptyset} + (\CSF{k-1}{p})_{\RG} \\
& - (\CSF{k-1}{p-1})_{\RA_1\RY_2} - (\CSF{k-1}{p-1})_{\RE} + (\CSF{k-1}{p-1})_{\RF} \\
& - (\CSE{k}{p})_{\RR} - (\CSE{k}{p})_{\RS_1^1} - (\CSE{k}{p})_{\RS_1^{\ta}} - 
  (\CSE{k}{p})_{\RS_1^{\tb}} - (\CSE{k}{p})_{\RS_2} \quad \text{by \eqref{eq:Ind4} and \eqref{eq:Ind5}} \\[3mm]
& = (\CSE{k}{p})_{\RS_2} + (\CSE{k}{p})_{\RS_1^{\tb}} - (\CSF{k-1}{p-1})_{\RF_{2}^{2}} + 
(\CSE{k}{p})_{\RR} - (\CSF{k-1}{p-1})_{\RF_{1}} \\
& + (\CSE{k}{p})_{\RS_1^1} + (\CSE{k}{p})_{\RS_1^{\ta}} 
  - (\CSF{k-1}{p-1})_{\RF_{2}^{1}} + (\CSF{k-1}{p-1})_{\RF} \\
& - (\CSE{k}{p})_{\RR} - (\CSE{k}{p})_{\RS_1^1} - (\CSE{k}{p})_{\RS_1^{\ta}} - 
  (\CSE{k}{p})_{\RS_1^{\tb}} - (\CSE{k}{p})_{\RS_2} \quad \text{by Proposition~\ref{prop:mat3a}} \\[3mm]
& = 0 \quad \text{by \eqref{eq:RF}}. 
\end{align*}
This completes the proof of Theorem~\ref{thm:pieri}.

%
\section{Proof of Proposition~\ref{prop:mat1a}.}
\label{sec:prfmat1}

Let $\g \in \{p-1,p\}$. 
Here again we list the conditions needed in this section: 
\begin{itemize}
\item For $((\bp,M) \mid \bm) \in \hspmw{k-1}{\g}$ with $g \in \{p-1,p\}$, 
\begin{equation*}
\begin{cases}
(\RA) & n_{(k-1,k)}(\bp)=0, \\
(\RB) & n_{(k-1,k)}(\bp)=1, \\
(\RB_1) & \text{$n_{(k-1,k)}(\bp)=1$ and $(k-1,k) \not\in M$}, \\
(\RB_2) & \text{$(k-1,k) \in M$ and $\kappa(\bp)=(k-1,k)$}, \\
(\RB_3) & \text{$(k-1,k) \in M$ and $\kappa(\bp) \ne (k-1,k)$}, \\
(\RX) & \iota(\bm) = (k-1,k), \\
(\RY) & \iota(\bm) \ne (k-1,k).
\end{cases}
\end{equation*}
\item For $((\bp,M) \mid \bm) \in \hspmw{k-2}{\g-1}$ with $g \in \{p-1,p\}$, 
\begin{equation*}
\begin{cases}
(\RC) & n_{(*,k-1)}(\bp) = 0, \\
(\RD) & n_{(*,k-1)}(\bp) \ge 1, \\
(\RD_1) & \text{see Definition~\ref{dfn:D**}; {\bf Algorithm $(\bp_{(*,k-1)}:(k-1,k))$} ends with} \\
& \text{a directed path of the form \eqref{eq:dec12-2}}, \\
(\RD_{11}) & \text{see Definition~\ref{dfn:D**}; ($\RD_1$) holds, and  
$\{i_{1},\dots,i_{s}\} \cap \{j_{1},\dots,j_{t}\} = \emptyset$}, \\
(\RD_{12}) & \text{see Definition~\ref{dfn:D**}; ($\RD_1$) holds, and  
$\{i_{1},\dots,i_{s}\} \cap \{j_{1},\dots,j_{t}\} \ne \emptyset$}, \\
(\RD_2) & \text{see Definition~\ref{dfn:D**}; {\bf Algorithm $(\bp_{(*,k-1)}:(k-1,k))$} ends with} \\
& \text{a directed path of the form \eqref{eq:dec12-3}}, \\
(\RX) & \iota(\bm) = (k-1,k), \\
(\RY) & \iota(\bm) \ne (k-1,k).
\end{cases}
\end{equation*}
\end{itemize}
Also, recall from \eqref{eq:bF} the definition of 
$\bF{\h}{\g}{\bq}$, $\bq \in \hspmw{\h}{\g}$, 
for $(\h,\g) \in \{(k-1,p-1),\,(k-1,p),\,(k-2,p-1),\,(k-2,p-2)\}$. 
Proposition~\ref{prop:mat1a} follows from the next proposition. 
%
%
\begin{prop} \label{prop:mat1}
Let $g \in \{p-1,p\}$. 
\begin{enu}
\item 
There exists a bijection 
$\pi_{1}:(\hspmw{k-1}{\g})_{\RA\RX} \rightarrow (\hspmw{k-1}{\g})_{\RB_1\RY}$ 
satisfying 
\begin{equation*}
\bF{k-1}{\g}{\pi_{1}(\bq)}=-\bF{k-1}{\g}{\bq} \quad 
\text{\rm for $\bq \in (\hspmw{k-1}{\g})_{\RA\RX}$}.
\end{equation*}

\item 
There exists a bijection 
$\pi_{2}:(\hspmw{k-1}{\g})_{\RA\RY} \rightarrow (\hspmw{k-1}{\g})_{\RB_1\RX}$ 
satisfying
\begin{equation*}
\bF{k-1}{\g}{\pi_{2}(\bq)} = - \Q_{k-1}\bF{k-1}{\g}{\bq}
\quad \text{\rm for $\bq \in (\hspmw{k-1}{\g})_{\RA\RY}$}.
\end{equation*} 

\item 
There exists a bijection 
$\pi_{3}:(\hspmw{k-1}{\g})_{\RB_2\RX} \rightarrow (\hspmw{k-2}{\g-1})_{\RC\RY}$ 
satisfying 
\begin{equation*}
\bF{k-2}{\g-1}{\pi_{3}(\bq)} = \Q_{k-1}^{-1}\bF{k-1}{\g}{\bq} 
\quad \text{\rm for $\bq \in (\hspmw{k-1}{\g})_{\RB_2\RX}$}.
\end{equation*} 

\item 
There exists a bijection 
$\pi_{4}:(\hspmw{k-1}{\g})_{\RB_2\RY} \rightarrow (\hspmw{k-2}{\g-1})_{\RC\RX}$ 
satisfying 
\begin{equation*}
\bF{k-2}{\g-1}{\pi_{4}(\bq)} = \bF{k-1}{\g}{\bq} 
\quad \text{\rm for $\bq \in (\hspmw{k-1}{\g})_{\RB_2\RY}$}.
\end{equation*} 

\item 
There exists a bijection 
$\pi_{5}:(\hspmw{k-1}{\g})_{\RB_3 \RX} \rightarrow (\hspmw{k-2}{\g-1})_{\RD_{11}\RY}$ 
satisfying 
\begin{equation*}
\bF{k-2}{\g-1}{\pi_{5}(\bq)} = \Q_{k-1}^{-1}\bF{k-1}{\g}{\bq}
\quad \text{\rm for $\bq \in (\hspmw{k-1}{\g})_{\RB_3\RX}$}. 
\end{equation*}

\item
There exists a bijection 
$\pi_{6}:(\hspmw{k-1}{\g})_{\RB_3 \RY} \rightarrow (\hspmw{k-2}{\g-1})_{\RD_{11}\RX}$ 
satisfying 
\begin{equation*}
\bF{k-2}{\g-1}{\pi_{6}(\bq)} = \bF{k-1}{\g}{\bq}
\quad \text{\rm for $\bq \in (\hspmw{k-1}{\g})_{\RB_3\RY}$}. 
\end{equation*} 

\item 
There exists a bijection 
$\pi_{7}:(\hspmw{k-2}{\g-1})_{\RD_{12}\RX} \rightarrow (\hspmw{k-2}{\g-1})_{\RD_2\RY}$ 
satisfying 
\begin{equation*}
\bF{k-2}{\g-1}{\pi_{7}(\bq)} = - \Q_{k-1}^{-1} \bF{k-2}{\g-1}{\bq} 
\quad \text{\rm for $\bq \in (\hspmw{k-2}{\g-1})_{\RD_{12}\RX}$}. 
\end{equation*}

\item 
There exists a bijection 
$\pi_{8}:(\hspmw{k-2}{\g-1})_{\RD_{12}\RY} \rightarrow (\hspmw{k-2}{\g-1})_{\RD_2\RX}$ 
satisfying 
\begin{equation*}
\bF{k-2}{\g-1}{\pi_{8}(\bq)} = - \bF{k-2}{\g-1}{\bq} 
\quad \text{\rm for $\bq \in (\hspmw{k-2}{\g-1})_{\RD_{12}\RY}$}.
\end{equation*}
\end{enu}
\end{prop}

\subsection{Proof of (1).}
\label{subsec:prfmat1-1}
Let $\bq=((\bp,M) \mid \bm) \in (\hspmw{k-1}{\g})_{\RA\RX}$.
We write $\bp$ and $\bm$ as: 
\begin{equation} \label{eq:pm}
\begin{split}
& \bp = (w\,;\,(a_{1},b_{1}),\dots,(a_{r},b_{r})), \\
& \bm = (\ed(\bp)\,;\,(c_{1},k),\dots,(c_{u},k),\bm_{(k,*)});
\end{split}
\end{equation}
note that $(a_{s},b_{s}) \ne (k-1,k)$ for any $1 \le s \le r$, 
and $c_{1} = k-1$. We define 
\begin{equation} \label{eq:mat1-1a}
\begin{split}
& \bp \ast (k-1,k)_{\kappa} := 
  (w\,;\,(a_{1},b_{1}),\dots,(a_{r},b_{r}),(k-1,k)), \\
& \bm \setminus (k-1,k)_{\iota} := 
  (\ed(\bp) \cdot (k-1,k)\,;\,(c_{2},k),\dots,(c_{u},k),\bm_{(k,*)}),
\end{split}
\end{equation}
and set $\pi_{1}(\bq):=((\bp \ast (k-1,k)_{\kappa},M) \mid \bm \setminus (k-1,k)_{\iota})$; 
we see that $\pi_{1}(\bq) \in (\hspmw{k-1}{\g})_{\RB_1\RY}$ and 
$\bF{k-1}{\g}{\pi_{1}(\bq)} = - \bF{k-1}{\g}{\bq}$. We show the bijectivity of the map 
$\pi_{1}:(\hspmw{k-1}{\g})_{\RA\RX} \rightarrow (\hspmw{k-1}{\g})_{\RB_1\RY}$ 
by giving its inverse. Let $\bq=((\bp,M) \mid \bm) \in (\hspmw{k-1}{\g})_{\RB_1\RY}$, 
with $\bp$ and $\bm$ as in \eqref{eq:pm}; 
note that $(a_{r},b_{r}) = (k-1,k)$ (see Remark~\ref{rem:B123}\,(1)) 
and $c_{1} \ne k-1$. We define
\begin{equation} \label{eq:mat1-1b}
\begin{split}
& \bp \setminus (k-1,k)_{\kappa} := (w\,;\,(a_{1},b_{1}),\dots,(a_{r-1},b_{r-1})), \\
& (k-1,k)_{\iota} \ast \bm 
  := (\ed(\bp) \cdot (k-1,k)\,;\,(k-1,k),(c_{1},k),\dots,(c_{u},k),\bm_{(k,*)}), 
\end{split}
\end{equation}
and set $\pi_{1}'(\bq):=((\bp \setminus (k-1,k)_{\kappa}, M) \mid (k-1,k)_{\iota} \ast \bm)$; 
we see that $\pi_{1}'(\bq) \in (\hspmw{k-1}{\g})_{\RA\RX}$ and 
$\bF{k-1}{\g}{\pi_{1}'(\bq)} = - \bF{k-1}{\g}{\bq}$. 
It is easily verified that $\pi_{1}'$ is the inverse of $\pi_{1}$. 
This proves part (1).

\subsection{Proof of (2).}
\label{subsec:prfmat1-2}
Let $\bq=((\bp,M) \mid \bm) \in (\hspmw{k-1}{\g})_{\RA\RY}$, 
with $\bp$ and $\bm$ as in \eqref{eq:pm}; 
note that $(a_{s},b_{s}) \ne (k-1,k)$ for any $1 \le s \le r$, 
and $c_{1} \ne k-1$. 
We set $\pi_{2}(\bq):=((\bp \ast (k-1,k)_{\kappa},M) \mid (k-1,k)_{\iota} \ast \bm)$, 
where $\bp \ast (k-1,k)_{\kappa}$ and $(k-1,k)_{\iota} \ast \bm$ are 
defined as in \eqref{eq:mat1-1a} and \eqref{eq:mat1-1b}, respectively; 
we see that $\pi_{2}(\bq) \in (\hspmw{k-1}{\g})_{\RB_1\RX}$ and 
$\bF{k-1}{\g}{\pi_{2}(\bq)} = - \Q_{k-1}\bF{k-1}{\g}{\bq}$.
Let us show the bijectivity of the map $\pi_{2}$. 
Let $\bq=((\bp,M) \mid \bm) \in (\hspmw{k-1}{\g})_{\RB_1\RX}$, 
with $\bp$ and $\bm$ as in \eqref{eq:pm}; 
note that $(a_{r},b_{r}) = (k-1,k)$ (see Remark~\ref{rem:B123}\,(1)), 
and $c_{1} = k-1$.
We set $\pi_{2}'(\bq):=( (\bp \setminus (k-1,k)_{\kappa},M) \mid \bm \setminus (k-1,k)_{\iota})$, 
where $\bp \setminus (k-1,k)_{\kappa}$ and $\bm \setminus (k-1,k)_{\iota}$ are defined as in \eqref{eq:mat1-1b} and \eqref{eq:mat1-1a}, respectively; 
we see that $\pi_{2}'(\bq) \in (\hspmw{k-1}{\g})_{\RA\RY}$ and 
$\bF{k-1}{\g}{\pi_{2}'(\bq)} = - \Q_{k-1}^{-1}\bF{k-1}{\g}{\bq}$. 
It is easily verified that $\pi_{2}'$ is the inverse of $\pi_{2}$.
This proves part (2).

\subsection{Proof of (3).}
\label{subsec:prfmat1-3}
Let $\bq=((\bp,M) \mid \bm) \in (\hspmw{k-1}{\g})_{\RB_2 \RX}$, 
with $\bp$ and $\bm$ as in \eqref{eq:pm}; 
note that $(a_{r},b_{r}) = (k-1,k)$, and $c_{1} = k-1$. We set 
\begin{equation*}
\pi_{3}(\bq):=( (\bp \setminus (k-1,k)_{\kappa}, M \setminus \{(k-1,k)\}) \mid \bm \setminus (k-1,k)_{\iota});
\end{equation*}
we see by Remark~\ref{rem:B123}\,(2) that $\pi_{3}(\bq) \in \hspm{k-2}{\g-1}_{\RC\RY}$ and 
$\bF{k-2}{\g-1}{\pi_{3}(\bq)} = \Q_{k-1}^{-1}\bF{k-1}{\g}{\bq}$. 
Let us show the bijectivity of the map $\pi_{3}$. 
Let $\bq=((\bp,M) \mid \bm) \in \hspm{k-2}{\g-1}_{\RC \RY}$. We set 
\begin{equation*}
\pi_{3}'(\bq):=( (\bp \ast (k-1,k)_{\kappa}, M \sqcup \{(k-1,k)\}) \mid (k-1,k)_{\iota} \ast \bm);
\end{equation*}
we see that $\pi_{3}'(\bq) \in (\hspmw{k-1}{\g})_{\RB_2\RX}$ and 
$\bF{k-2}{\g-1}{\pi_{3}'(\bq)} = \Q_{k-1}\bF{k-1}{\g}{\bq}$. 
It is easily verified that $\pi_{3}'$ is the inverse of $\pi_{3}$.
This proves part (3).

\subsection{Proof of (4).}
\label{subsec:prfmat1-4}
Let $\bq=((\bp,M) \mid \bm) \in (\hspmw{k-1}{\g})_{\RB_2\RY}$, with $\bp$ and $\bm$ as in \eqref{eq:pm}; 
note that $(a_{r},b_{r}) = (k-1,k)$, and $c_{1} \ne k-1$. We set 
\begin{equation*}
\pi_{4}(\bq):=((\bp \setminus (k-1,k)_{\kappa}, M \setminus \{(k-1,k)\}) \mid (k-1,k)_{\iota} \ast \bm); 
\end{equation*}
we see that $\pi_{4}(\bq) \in \hspm{k-2}{\g-1}_{\RC\RX}$ and 
$\bF{k-2}{\g-1}{\pi_{4}(\bq)} = \bF{k-1}{\g}{\bq}$. Let us show the bijectivity of the map $\pi_{4}$. 
Let $\bq=((\bp,M) \mid \bm) \in \hspm{k-2}{\g-1}_{\RC \RX}$. We set 
\begin{equation*}
\pi_{4}'(\bq):=( (\bp \ast (k-1,k)_{\kappa}, M \sqcup \{(k-1,k)\}) \mid \bm \setminus (k-1,k)_{\iota});
\end{equation*}
we see that $\pi_{4}'(\bq) \in (\hspmw{k-1}{\g})_{\RB_2\RY}$ and 
$\bF{k-1}{\g}{\pi_{4}'(\bq)} = \bF{k-2}{\g-1}{\bq}$. 
It is easily verified that $\pi_{4}'$ is the inverse of $\pi_{4}$.
This proves part (4).

\subsection{Proof of (5).}
\label{subsec:prfmat1-5}
Let $\bq=((\bp,M) \mid \bm) \in (\hspmw{k-1}{\g})_{\RB_3 \RX}$. 
We write $\bp$ and $\bm$ as:
%
%
\begin{equation} \label{eq:mat15a}
\begin{split}
& \bp = (w\,;\,\underbrace{\dots\dots\dots\dots\dots\dots}_{
\begin{subarray}{c}
\text{This segment contains} \\[1mm]
\text{no label of } \\[1mm]
\text{the form $(k-1,*)$;} \\[1mm]
\text{see Remark~\ref{rem:B123}\,(3).}
\end{subarray}},
\underbrace{(i_{1},k),\dots,(i_{s},k),(k-1,k),
(j_{1},k),\dots,(j_{t},k)}_{ =\,\bp_{(*,k)} }), \\[3mm]
& \bm = (\ed(\bp)\,;\,(c_{1},k),\dots,(c_{u},k),\bm_{(k,*)});
\end{split}
\end{equation}
note that $t \ge 1$, and $c_{1}=k-1$. 
Since $1 \le j_{1},\dots,j_{t} \le k-2$, 
we deduce from Lemma~\ref{lem:int}\,(2), applied to the segment 
$(k-1,k),(j_{1},k),\dots,(j_{t},k)$, that 
%
%
\begin{equation} \label{eq:mat15b}
(\underbrace{w\,;\,\dots\dots,
(i_{1},k),\dots,(i_{s},k),
(j_{1},k-1),\dots,(j_{t},k-1)}_{%
=:\,\psi_{\RB_3}(\bp)},(k-1,k))
\end{equation}
is a directed path. Also, we define $\vp_{\RB_3}(M)$ 
by replacing each label of the form $(j_{r},k)$, 
$1 \le r \le t$, in $M$ with $(j_{r},k-1)$, and then 
removing $(k-1,k) \in M$. We set 
\begin{equation*}
\pi_{5}(\bq):=( (\psi_{\RB_3}(\bp), \vp_{\RB_3}(M) ) \mid 
  \bm \setminus (k-1,k)_{\iota}); 
\end{equation*}
we see that $\pi_{5}(\bq) \in \hspm{k-2}{\g-1}_{\RD_{11}\RY}$ and 
$\bF{k-2}{\g-1}{\pi_{5}(\bq)} = \Q_{k-1}^{-1} \bF{k-1}{\g}{\bq}$. 
Let us show the bijectivity of the map $\pi_{5}$ by giving its inverse. 
Let $\bq=((\bp,M) \mid \bm) \in \hspm{k-2}{\g-1}_{\RD_{11}\RY}$, 
and assume that $\bp$ is of the form \eqref{eq:dec12-1}. 
Then we define $\psi_{\RD_{11}}(\bp)$ to be the directed path \eqref{eq:dec12-2}. 
Also, we define $\vp_{\RD_{11}}(M)$ by replacing each label of the form $(j_{r},k-1)$, 
$1 \le r \le t$, in $M$ with $(j_{r},k)$, and then adding $(k-1,k)$ to 
the resulting set. 
Since $\{i_{1},\dots,i_{s}\} \cap \{j_{1},\dots,j_{t}\} = \emptyset$ and $t \ge 1$, 
we can check that $(\psi_{\RD_{11}}(\bp), \vp_{\RD_{11}}(M)) \in (\hspw{k-1}{\g})_{\RB_3}$. 
We set 
\begin{equation*}
\pi_{5}'(\bq):=( (\psi_{\RD_{11}}(\bp), \vp_{\RD_{11}}(M)) \mid (k-1,k)_{\iota} \ast \bm);
\end{equation*}
we see that $\pi_{5}'(\bq) \in (\hspmw{k-1}{\g})_{\RB_3\RX}$ and 
$\bF{k-1}{\g}{\pi_{5}'(\bq)} = \Q_{k-1} \bF{k-2}{\g-1}{\bq}$. 
It is easily verified that $\pi_{5}'$ is the inverse of $\pi_{5}$.
This proves part (5).

\subsection{Proof of (6).}
\label{subsec:prfmat1-6}
Let $\bq=((\bp,M) \mid \bm) \in (\hspmw{k-1}{\g})_{\RB_3\RY}$, 
with $\bp$ and $\bm$ as in \eqref{eq:mat15a}; 
note that $t \ge 1$, and $c_{1} \ne k-1$. 
Define $\psi_{\RB_3}(\bp)$ and $\vp_{\RB_3}(M)$ as in the proof of (5), 
and set
\begin{equation*}
\pi_{6}(\bq):=( (\psi_{\RB_3}(\bp), \vp_{\RB_3}(M) ) \mid 
  (k-1,k)_{\iota} \ast \bm); 
\end{equation*}
we see that $\pi_{6}(\bq) \in \hspm{k-2}{\g-1}_{\RD_{11}\RX}$ 
(note that $\bF{k-2}{\g-1}{\pi(\bq)} = \bF{k-1}{\g}{\bq}$). 
Let us show the bijectivity of the map $\pi_{6}$ by giving its inverse. 
Let $\bq=((\bp,M) \mid \bm) \in \hspm{k-2}{\g-1}_{\RD_{11}\RX}$. 
Define $\psi_{\RD_{11}}(\bp)$ and $\vp_{\RD_{11}}(M)$ as in the proof of (5). 
We set 
\begin{equation*}
\pi_{6}'(\bq):=( (\psi_{\RD_{11}}(\bp), \vp_{\RD_{11}}(M)) \mid \bm \setminus (k-1,k)_{\iota});
\end{equation*}
we see that $\pi_{6}'(\bq) \in (\hspmw{k-1}{\g})_{\RB_2\RY}$ and 
$\bF{k-1}{\g}{\pi_{6}'(\bq)} = \bF{k-2}{\g-1}{\bq}$. 
It is easily verified that $\pi_{6}'$ is the inverse of $\pi_{6}$.
This proves part (6). 
%
\subsection{Proof of (7).}
\label{subsec:prfmat1-7}
Let $\bq=((\bp,M) \mid \bm) \in (\hspmw{k-2}{\g-1})_{\RD_{12}\RX}$. 
Assume that $\bp$ is of the form \eqref{eq:dec12-1}; 
recall from the definition that 
$\{i_{1},\dots,i_{s}\} \cap \{j_{1},\dots,j_{t}\} \ne \emptyset$. 
We set
\begin{equation} \label{eq:s(p)}
s(\bp):=\max \bigl\{ 1 \le s' \le s \mid i_{s'} \in \{j_{1},\dots,j_{t}\} \bigr\}. 
\end{equation}
Let $1 \le u \le t$ be such that $i_{s(\bp)} = j_{u}=:a$. 

\begin{ex} \label{ex:mat1-7a}
Recall from Example~\ref{ex:D**} that
\begin{equation*}
(\bp,M):=
( (32514 \,;\, 
  (1,5)_{\SB}, (2,5)_{\SB}, (3,4)_{\SQ}, (1,4)_{\SB},(2,4)_{\SB}), \{(1,5),(3,4)\} ) 
\in (\hspw{3}{2})_{\RD_{12}}. 
\end{equation*}
In this case, $s = 2$, $i_{1}=1$, $i_{2}=2$ and $t=3$, $j_{1}=3$, $j_{2}=1$, $j_{3}=2$. 
Hence it follows that $s(\bp)=2$ and $u=3$. 
\end{ex}

\begin{clm} \label{c471} We have $u=t$. \end{clm}

\noindent
{\it Proof of Claim~\ref{c471}.} Suppose, for a contradiction, that 
$u < t$. By condition (P2) on $\bp$, we have $j_{u+1} > j_{u}$. 
Recall from \eqref{eq:dec12-2} that 
\begin{equation*}
\begin{split}
\bigl( w\,;\, & \dots\dots, 
  (i_{1},k),\dots,\overbrace{(i_{s(\bp)},k)}^{=\,(a,k)},\dots,(i_{s},k),(k-1,k), \\
& (j_{1},k-1),\dots,\underbrace{(j_{u},k-1)}_{=\,(a,k-1)},
  (j_{u+1},k-1),\dots,(j_{t},k-1) \bigr)
\end{split}
\end{equation*}
is a directed path. Applying Lemma~\ref{lem:int}\,(2) repeatedly to the segment 
$(i_{1},k),\dots,(i_{s},k),(k-1,k)$ in the directed path above, 
we deduce that 
\begin{equation*}
\begin{split}
\bigl( w\,;\, & \dots\dots, (k-1,k), 
  (i_{1},k-1),\dots,\overbrace{(i_{s(\bp)},k-1)}^{=\,(a,k-1)},\dots,(i_{s},k-1), \\
& (j_{1},k),\dots,\underbrace{(j_{u},k)}_{=\,(a,k)},
  (j_{u+1},k),\dots,(j_{t},k) \bigr)
\end{split}
\end{equation*}
is a directed path. By Lemma~\ref{lem:int}\,(1) and 
the definition \eqref{eq:s(p)} of $s(\bp)$, 
\begin{equation*}
\begin{split}
\bigl( w\,;\, & \dots\dots, (k-1,k), 
  (i_{1},k-1),\dots,(i_{s(\bp)-1},k-1),(j_{1},k),\dots,(j_{u-1},k), \\
& \underbrace{(i_{s(\bp)},k-1)}_{=\,(a,k-1)},\underbrace{(j_{u},k)}_{=\,(a,k)},
  (j_{u+1},k),\dots,(j_{t},k), (i_{s(\bp)+1},k-1),\dots,(i_{s},k-1)\bigr)
\end{split}
\end{equation*}
is a directed path, which has the segment 
$(a,k-1),(a,k),(j_{u+1},k)$. However, since 
$a = j_{u} < j_{u+1}$, this contradicts Lemma~\ref{lem:LS29a}.
Hence we obtain $u = t$, as desired. Next, suppose, for a contradiction, 
that there exists $1 \le s' < s(\bp)$ such that 
$i_{s'} \in \{j_{1},\dots,j_{t}=j_{u}\}$; note that $i_{s'} \ne i_{s(\bp)} = a$ by (P0). 
Let $1 \le t' \le t$ be such that $i_{s'}=j_{t'}$. 
By the same argument as above, we can easily show that $t'=t$, 
and hence $i_{s'}=j_{t'} = j_{t} = a$, which is a contradiction. Hence we conclude that 
$\bigl\{ 1 \le s' \le s \mid i_{s'} \in \{j_1,\dots,j_t\} \bigr\} = \{s(\bp)\}$. 
This proves the claim. \bqed

\medskip

To summarize, we conclude that 
the element $\bp \in (\sfpw{k-2}{\g-1})_{\RD_{12}}$ is of the form:
%
%
\begin{equation} \label{eq:pD12}
\bp = \bigl( w\,;\,\dots\dots,
  \underbrace{(i_{1},k),\dots,\overbrace{(i_{s(\bp)},k)}^{=\,(a,k)},\dots,(i_{s},k)}_{%
  =\,\bp_{(*,k)} },
  \underbrace{(j_{1},k-1),\dots,\overbrace{(j_{t},k-1)}^{=\,(a,k-1)}}_{%
  =\,\bp_{(*,k-1)} } \bigr),
\end{equation}
with $\{i_{1},\dots,i_{s}\} \cap \{j_{1},\dots,j_{t}\} = \{a\}$. 
By the definition \eqref{eq:s(p)} of $s(\bp)$ and Lemma~\ref{lem:int}\,(1), we see that 
\begin{equation} \label{eq:pD12b}
\begin{split}
\bigl( w\,;\,\dots\dots,
  (i_{1},k),\dots,\, & (i_{s(\bp)-1},k),
  (j_{1},k-1),\dots,(j_{t-1},k-1), \\
& \underbrace{(i_{s(\bp)},k)}_{=\,(a,k)},
  \underbrace{(j_{t},k-1)}_{=\,(a,k-1)},(i_{s(\bp)+1},k),\dots,(i_{s},k) 
  \bigr)
\end{split}
\end{equation}
is a directed path. 
Applying Lemma~\ref{lem:int}\,(3) to the segment $(a,k),(a,k-1)$, 
we deduce that 
\begin{equation*}
\begin{split}
\bigl( w\,;\,\dots\dots,
  (i_{1},k),\dots,\, & (i_{s(\bp)-1},k),
  (j_{1},k-1),\dots,(j_{t-1},k-1), \\
& (a,k-1),(k-1,k),(i_{s(\bp)+1},k),\dots,(i_{s},k) 
  \bigr)
\end{split}
\end{equation*}
is a directed path. Similarly, by using 
Lemma~\ref{lem:int}\,(2) repeatedly, we deduce that 
\begin{equation} \label{eq:pD12a}
\begin{split}
\bigl( w\,;\,\dots\dots,
  (i_{1},k),\dots,\, & (i_{s(\bp)-1},k),
  (j_{1},k-1),\dots,(j_{t-1},k-1), \\
& \underbrace{(a,k-1)}_{%
  \begin{subarray}{c}
  =\,(j_{t},k-1) \\
  =\,(i_{s(\bp)},k-1)
  \end{subarray}},(i_{s(\bp)+1},k-1),\dots,(i_{s},k-1),(k-1,k)
  \bigr)
\end{split}
\end{equation}
is a directed path. Now we define $\psi_{\RD_{12}}(\bp)$ to be the 
directed path obtained by removing the final label $(k-1,k)$ 
from the directed path \eqref{eq:pD12a}. 
Also, we define $\vp_{\RD_{12}}(M)$ by replacing each label of the form $(i_{r},k)$, 
$s(\bp) \le r \le s$, in $M$ with $(i_{r},k-1)$. We set 
\begin{equation*}
\pi_{7}(\bq):=((\psi_{\RD_{12}}(\bp),\vp_{\RD_{12}}(M)) \mid \bm \setminus (k-1,k)_{\iota});
\end{equation*}
we see by \eqref{eq:pD12b} and \eqref{eq:pD12a} that 
$\pi_{7}(\bq) \in \hspm{k-2}{\g-1}_{\RD_2\RY}$, and that 
$\bF{k-2}{\g-1}{\pi_{7}(\bq)} = - \Q_{k-1}^{-1} \bF{k-2}{\g-1}{\bq}$. 

\begin{ex} \label{ex:mat1-7b}
Keep the notation and setting of Example~\ref{ex:mat1-7a}. Recall the element 
\begin{equation*}
(\bp,M)=
( (32514 \,;\, 
  (1,5)_{\SB}, (2,5)_{\SB}, (3,4)_{\SQ}, (1,4)_{\SB},(2,4)_{\SB}), \{(1,5),(3,4)\} ) 
\in (\hspw{3}{2})_{\RD_{12}}. 
\end{equation*}
By applying Lemma~\ref{lem:int}\,(1) to the segment 
$(2,5)_{\SB},(3,4)_{\SQ}$, and then to the segment $(2,5)_{\SB},(1,4)_{\SB}$, 
we obtain the following directed path: 
\begin{equation*}
(32514 \,;\, 
  (1,5)_{\SB}, (3,4)_{\SQ}, (1,4)_{\SB}, (2,5)_{\SB}, (2,4)_{\SB}). 
\end{equation*}
Applying Lemma~\ref{lem:int}\,(3) to the segment 
$(2,5)_{\SB},(2,4)_{\SB}$ in the directed path above, 
we obtain the following directed path
\begin{equation*}
(32514 \,;\, 
  (1,5)_{\SB}, (3,4)_{\SQ}, (1,4)_{\SB}, (2,4)_{\SB}, (4,5)_{\SB}).
\end{equation*}
In the definition above, we apply Lemma~\ref{lem:int}\,(2) repeatedly 
in order to move $(4,5)_{\SB}$ to the rightmost of the directed path. 
However, in this case, we do not need this procedure. Hence it follows that 
\begin{align*}
& \psi_{\RD_{12}}(\bp) = (32514 \,;\, 
  (1,5)_{\SB}, (3,4)_{\SQ}, (1,4)_{\SB}, (2,4)_{\SB}), \\
& \vp_{\RD_{12}}(M) = \{ (1,5), (3,4) \}; 
\end{align*}
note that $(1,5) \in M$ is not replaced by $(1,4)$ 
since $s(\bp)= 2 = s$. 
\end{ex}

Let us show the bijectivity of the map $\pi_{7}$ by giving its inverse. 
Let $\bq=((\bp,M) \mid \bm) \in \hspm{k-2}{\g-1}_{\RD_{2}\RY}$, with $\bp$ as in \eqref{eq:dec12-1}. 
Recall the definition of $t(\bp)$ from Definition~\ref{dfn:D**}.
Since $1 \le j_{1},\dots,j_{t} \le k-2$ are all distinct by (P0), 
by applying Lemma~\ref{lem:int}\,(1) to the directed path \eqref{eq:dec12-3}, we deduce that 
\begin{equation} \label{eq:D2a}
\begin{split}
  (w\,;\, & \dots\dots,(i_{1},k),\dots,(i_{s},k),(j_{t(\bp)},k),
  (j_{t(\bp)+1},k),\dots,(j_{t},k), \\
& (j_{1},k-1),\dots,(j_{t(\bp)-1},k-1),(j_{t(\bp)},k-1))
\end{split}
\end{equation}
is a directed path; let us denote this directed path by $\psi_{\RD_2}(\bp)$. 

\begin{ex} \label{ex:mat1-7c}
Keep the notation and setting of Examples~\ref{ex:D**} and \ref{ex:mat1-7b}. 
Let 
\begin{equation*}
(\bp',M') : =
( (32514 \,;\, 
  (1,5)_{\SB}, (3,4)_{\SQ}, (1,4)_{\SB},(2,4)_{\SB}), \{(1,5),(3,4)\} ) 
\in (\hspw{3}{2})_{\RD_{2}}; 
\end{equation*}
note that $s = 1$, $i_{1}=5$ and $t=3$, $j_{1}=3$, $j_{2}=1$, $j_{3}=2$. 
By running {\bf Algorithm $(\bp_{(*,4)}:(4,5))$} for $\bp'$, we obtain 
\begin{equation*}
( (32514 \,;\, 
  (1,5)_{\SB}, (3,4)_{\SQ}, (1,4)_{\SB},(2,5)_{\SB}, (2,4)_{\SB}); 
\end{equation*}
hence it follows that $t(\bp')=3$. By using Lemma~\ref{lem:int}\,(1) repeatedly, 
we can move $(2,5)_{\SB}$ to the left of the segment $(3,4)_{\SQ},(1,4)_{\SB}$ as
\begin{equation*}
( (32514 \,;\, 
  (1,5)_{\SB}, (2,5)_{\SB}, (3,4)_{\SQ}, (1,4)_{\SB},(2,4)_{\SB}) = \psi_{\RD_{2}}(\bp'); 
\end{equation*}
notice that this directed path is $\bp$ in Example~\ref{ex:mat1-7b}. 
\end{ex}

\begin{clm} \label{c472}
The directed path $\psi_{\RD_2}(\bp)$ is an element of 
$\sfp{k-2}{\g-1}_{\RD_{12}}$.
\end{clm}

\noindent
{\it Proof of Claim~\ref{c472}.}
First, we show that $i_{s} < j_{t(\bp)}$, 
from which it follows that $\psi_{\RD_2}(\bp) \in \sfp{k-2}{\g-1}_{\RD}$. 
Assume that in the directed path \eqref{eq:D2a}, 
the transposition $(i_{s},k)$ is applied to $v$. Then, the directed path 
\begin{equation*}
\begin{split}
(v\,;\, & (i_{s},k),(j_{t(\bp)},k),
  (j_{t(\bp)+1},k),\dots,(j_{t},k), \\
& (j_{1},k-1),\dots,(j_{t(\bp)-1},k-1),(j_{t(\bp)},k-1))
\end{split}
\end{equation*}
is an element of $\SP^{k-2}(v)$.
Applying Lemma~\ref{lem:LS29}\,(2) (with $k$ replaced by $k-2$) 
to the first, second, and last label of this directed path, 
we obtain $i_{s} < j_{t(\bp)}$, as desired. Next, we consider the directed path 
\begin{equation*}
\begin{split}
(w\,;\, & \dots\dots,(i_{1},k),\dots,(i_{s},k),(j_{t(\bp)},k),
  (j_{t(\bp)+1},k),\dots,(j_{t},k), \\
& \underbrace{(j_{1},k-1),\dots,(j_{t(\bp)-1},k-1),(j_{t(\bp)},k-1)}_{=:\,\bs},(k-1,k)), 
\end{split}
\end{equation*}
and run {\bf Algorithm $(\bs:(k-1,k))$} for this directed path; 
it ends with a directed path either of the form: 
\begin{equation} \label{eq:D2x}
\begin{split}
(w\,;\, & \dots\dots,(i_{1},k),\dots,(i_{s},k),(j_{t(\bp)},k),
  (j_{t(\bp)+1},k),\dots,(j_{t},k), \\
& (k-1,k),(j_{1},k),\dots,(j_{t(\bp)-1},k),(j_{t(\bp)},k)),
\end{split}
\end{equation}
or of the from:
\begin{equation*}
\begin{split}
  (w \,;\, & \dots\dots,(i_{1},k),\dots,(i_{s},k),
  (j_{t(\bp)},k),(j_{t(\bp)+1},k),\dots,(j_{t},k), \\
& (j_{1},k-1),\dots,(j_{t'-1},k-1),(j_{t'},k),(j_{t'},k-1),(j_{t'+1},k),\dots,(j_{t(\bp)},k))
\end{split}
\end{equation*}
for some $1 \le t' \le t(\bp)$. 
Suppose, for a contradiction, that the latter case happens. 
Then there exists a directed path of the form: 
\begin{equation*}
\begin{split}
(w\,;\, & \dots\dots,(i_{1},k),\dots,(i_{s},k),(j_{t(\bp)},k),(j_{t(\bp)+1},k),\dots,(j_{t},k), \\
& (j_{t'},k),(j_{t'+1},k).\dots,(j_{t(\bp)},k),
  (j_{1},k-1),\dots,(j_{t'-1},k-1),(j_{t'},k-1)); 
\end{split}
\end{equation*}
notice that this directed path has the segment 
\begin{equation*}
(j_{t(\bp)},k),(j_{t(\bp)+1},k),\dots,(j_{t},k),(j_{t'},k),(j_{t'+1},k).\dots,(j_{t(\bp)},k)
\end{equation*}
whose labels are all contained in $\{ (a,k) \mid 1 \le a \le k-2\}$. 
This contradicts Lemma~\ref{lem:noak}. Hence the former case happens, and so 
$\psi_{\RD_2}(\bp)$ is an element of $\sfp{k-2}{\g-1}_{\RD_{12}}$, as desired. 
This proves the claim. \bqed

\medskip

Also, we define $\vp_{\RD_{2}}(M)$ by replacing each label of the form $(j_{r},k-1)$, 
$t(\bp) \le r \le t$, in $M$ with $(j_{r},k)$. We set 
\begin{equation*}
\pi_{7}'(\bq):=((\psi_{\RD_{2}}(\bp),\vp_{\RD_{2}}(M)) \mid (k-1,k)_{\iota} \ast \bm);
\end{equation*}
we see that $\pi_{7}'(\bq) \in \hspm{k-2}{\g-1}_{\RD_{12}\RX}$, and that 
$\bF{k-2}{\g-1}{\pi_{7}'(\bq)} = - \Q_{k-1} \bF{k-2}{\g-1}{\bq}$. 
It is easily verified that $\pi_{7}'$ is the inverse of $\pi_{7}$.
This proves part (7). 
%
\subsection{Proof of (8).}
\label{subsec:prfmat1-8}
Let $\bq=((\bp,M) \mid \bm) \in \hspm{k-2}{g-1}_{\RD_{12}\RY}$. 
Define $\psi_{\RD_{12}}(\bp)$ and $\vp_{\RD_{12}}(M)$ as in the proof of (7), 
and set 
\begin{equation*}
\pi_{8}(\bq):=((\psi_{\RD_{12}}(\bp),\vp_{\RD_{12}}(M)) \mid (k-1,k)_{\iota} \ast \bm);
\end{equation*}
we see that $\pi_{8}(\bq) \in \hspm{k-2}{\g-1}_{\RD_{2}\RX}$, 
and that $\bF{k-2}{\g-1}{\pi_{8}(\bq)} = - \bF{k-2}{\g-1}{\bq}$.
Let us show the bijectivity of the map $\pi_{8}$ by giving its inverse. 
Let $\bq=((\bp,M) \mid \bm) \in \hspm{k-2}{\g-1}_{\RD_{2}\RX}$. 
Define $\psi_{\RD_{2}}(\bp)$ and $\vp_{\RD_{2}}(M)$ as in the proof of (7), 
and set 
\begin{equation*}
\pi_{8}'(\bq):=( (\psi_{\RD_{2}}(\bp), \vp_{\RD_{2}}(M)) \mid \bm \setminus (k-1,k)_{\iota});
\end{equation*}
we see that $\pi_{8}'(\bq) \in (\hspmw{k-1}{\g})_{\RD_{12}\RY}$, and that 
$\bF{k-2}{\g-1}{\pi_{8}'(\bq)} = - \bF{k-2}{\g-1}{\bq}$. 
It is easily verified that $\pi_{8}'$ is the inverse of $\pi_{8}$.
This proves part (8).

%
\section{Proof of Proposition~\ref{prop:mat2a}.}
\label{sec:prfmat2}

Let $\g \in \{p-1,p\}$. 
Here again we list the conditions needed in this section: 
\begin{itemize}
\item For $((\bp,M) \mid \bm) \in \hspmw{k-1}{\g}$ with $g \in \{p-1,p\}$, 
\begin{equation*}
\begin{cases}
(\RA) & n_{(k-1,k)}(\bp)=0, \\
(\RA_1) & n_{(k-1,k)}(\bp)=0 \text{ and } \bp_{(*,k)} = \emptyset, \\
(\RA_2) & n_{(k-1,k)}(\bp)=0, \, \bp_{(*,k)} \ne \emptyset, \text{ and } \kappa(\bp) \not\in M, \\
(\RA_3) & n_{(k-1,k)}(\bp)=0, \, \bp_{(*,k)} \ne \emptyset, \text{ and } \kappa(\bp) \in M, \\
(\RB_2) & \text{$(k-1,k) \in M$ and $\kappa(\bp)=(k-1,k)$}, \\
(\RB_3) & \text{$(k-1,k) \in M$ and $\kappa(\bp) \ne (k-1,k)$}, \\
(\RB_3^1) & (k-1,k) \in M, \, \kappa(\bp) \ne (k-1,k), \, \kappa(\bp) \notin M, \\
(\RB_3^2) & (k-1,k) \in M, \, \kappa(\bp) \ne (k-1,k), \, \kappa(\bp) \in M, \\
(\RD_2) & \text{see Definition~\ref{dfn:D**}; {\bf Algorithm $(\bp_{(*,k-1)}:(k-1,k))$} ends with} \\
& \text{a directed path of the form \eqref{eq:dec12-3}}, \\
(\RT) & \bp_{(*,k)} \cap \bm_{(*,k)} \ne \emptyset, \\
(\RU) & \bp_{(*,k)} \cap \bm_{(*,k)} = \emptyset, \\
(\RT_1) & \bp_{(*,k)} \cap \bm_{(*,k)} \ne \emptyset \text{ and } \iota(\bm) \in \bp_{(*,k)}, \\
(\RT_2) & \bp_{(*,k)} \cap \bm_{(*,k)} \ne \emptyset \text{ and } \iota(\bm) \notin \bp_{(*,k)}, \\
(\RT_3) & \bp_{(*,k)} \cap \bm_{(*,k)} \ne \emptyset, \, 
          \iota(\bm) \in \bp_{(*,k)}, \text{ and } \kappa(\bp) \prec \iota(\bm), \\
(\RT_4) & \bp_{(*,k)} \cap \bm_{(*,k)} \ne \emptyset \text{ and } 
          (\iota(\bm) \notin \bp_{(*,k)} \text{ or } \kappa(\bp) \succ \iota(\bm)), \\
(\RY) & \iota(\bm) \ne (k-1,k), \\
(\RY_1) & \iota(\bm) \ne (k-1,k) \text{ and } \bm_{(*,k)} = \emptyset, \\
(\RY_2) & \iota(\bm) \ne (k-1,k),\, \bm_{(*,k)} = \emptyset, \text{ and } \bm_{(k,*)} \ne \emptyset, \\
(\RY_3) & \iota(\bm) \ne (k-1,k) \text{ and } \bm_{(*,k)} \ne \emptyset. 
\end{cases}
\end{equation*}
\end{itemize}
Also, recall from \eqref{eq:BA} and \eqref{eq:BB1}--\eqref{eq:decB23} that 
\begin{align*}
\BA & :=
(\hspmw{k-1}{\g})_{\RA_1 \RY_3} \sqcup 
(\hspmw{k-1}{\g})_{\RA_2 \RY} \sqcup (\hspmw{k-1}{\g})_{\RA_3 \RY_3}, \\[1mm]
\BB_1 & :=
  (\hspmw{k-1}{\g})_{\RB_{2}\RY_{3}}^{\RT_{1}} \sqcup 
  (\hspmw{k-1}{\g})_{\RB_{3}^{1}\RY_{3}}^{\RT_{3}} \sqcup 
  (\hspmw{k-1}{\g})_{\RB_{3}^{2}\RY_{3}}^{\RT_{1}}, \\[1mm]
\BB_2 & :=
  (\hspmw{k-1}{\g})_{\RB_{2}\RY_{3}}^{\RT_{2}} \sqcup 
  (\hspmw{k-1}{\g})_{\RB_{3}^{1}\RY_{3}}^{\RT_{4}} \sqcup 
  (\hspmw{k-1}{\g})_{\RB_{3}^{2}\RY_{3}}^{\RT_{2}}, \\[1mm] 
\BB_3 & := 
(\hspmw{k-1}{\g})_{\RB_{2} \RY_3}^{\RU} \sqcup (\hspmw{k-1}{\g})_{\RB_{3} \RY_3}^{\RU}, \\[1mm]
\BB_4 & :=(\hspmw{k-1}{\g})_{\RB_{3}^{1} \RY_1} \sqcup \BB_3, 
\end{align*}
and from \eqref{eq:bF} the definition of 
$\bF{\h}{\g}{\bq}$, $\bq \in \hspmw{\h}{\g}$, 
for $(\h,\g) \in \{(k-1,p-1),\,(k-1,p),\,(k-2,p-1),\,(k-2,p-2)\}$. 
Proposition~\ref{prop:mat2a} follows from the next proposition. 
%
%
\begin{prop} \label{prop:mat2}
Let $g \in \{p-1,p\}$. 
\begin{enu}
\item 
There exists a bijection $\theta_{1}:\BA \rightarrow \BA$ satisfying the condition that 
\begin{equation*}
\bF{k-1}{\g}{\theta_{1}(\bq)} = - \bF{k-1}{\g}{\bq} \quad \text{\rm for $\bq \in \BA$}.
\end{equation*} 

\item There exists a bijection $\theta_{2}:\BB_{2} \rightarrow \BB_{2}$ 
satisfying the condition that
\begin{equation*}
\bF{k-1}{\g}{\theta_{2}(\bq)} = - \bF{k-1}{\g}{\bq} 
\quad \text{\rm for $\bq \in \BB_{2}$}. 
\end{equation*} 

\item There exists a bijection $\theta_{3}:\BB_{4} \rightarrow \BB_{4}$ 
satisfying the condition that
\begin{equation*}
\bF{k-1}{\g}{\theta_{3}(\bq)} = - \bF{k-1}{\g}{\bq} \quad \text{\rm for $\bq \in \BB_{4}$}. 
\end{equation*} 

\item There exists a bijection $\theta_{4}:\hspm{k-2}{\g-1}_{\RD_2\RY} \rightarrow \BB_{1}$ 
satisfying the condition that
\begin{equation*}
\bF{k-1}{\g}{\theta_{4}(\bq)}=\Q_{k-1} \bF{k-2}{\g-1}{\bq} 
\quad \text{\rm for $\bq \in \hspm{k-2}{\g-1}_{\RD_2\RY}$}. 
\end{equation*} 
\end{enu}
\end{prop}

\subsection{Proof of (1).}
\label{subsec:prfmat2-1}
Let $\bq=((\bp,M) \mid \bm) \in \BA = 
(\hspmw{k-1}{\g})_{\RA_1 \RY_3} \sqcup 
(\hspmw{k-1}{\g})_{\RA_2 \RY} \sqcup (\hspmw{k-1}{\g})_{\RA_3 \RY_3}$. 
Let $(a,k)$ be the final label of the $(*,k)$-segment $\bp_{(*,k)}$
of $\bp$; if $\bp_{(*,k)} = \emptyset$, then we set $a:=0$. 
Let $(b,k)$ be the initial label of the $(*,k)$-segment $\bm_{(*,k)}$
of $\bm$; if $\bm_{(*,k)} = \emptyset$, then we set $b:=0$. 
Note that $0 \le a,\,b \le k-2$. Also, it follows from Lemma~\ref{lem:noak} 
that if $b > 0$, then $(b,k) \notin \bp_{(*,k)}$. We define
\begin{equation*} 
\theta_{1}(\bq):=
\begin{cases}
((\bp \setminus (a,k)_{\kappa},M) \mid (a,k)_{\iota} \ast \bm) \\[2mm]
\hspace*{10mm} 
\text{if $\bq \in (\hspmw{k-1}{\g})_{\RA_2 \RY}$ and $a > b$}, \\[3mm]
((\bp \ast (b,k)_{\kappa},M) \mid \bm \setminus (b,k)_{\iota}) \\[2mm]
\hspace*{10mm}
\text{if $\bq \in (\hspmw{k-1}{\g})_{\RA_1 \RY_3} \sqcup (\hspmw{k-1}{\g})_{\RA_3 \RY_3}$, or } \\[2mm]
\hspace*{10mm} 
\text{if $\bq \in (\hspmw{k-1}{\g})_{\RA_2 \RY}$ and $a < b$}. 
\end{cases}
\end{equation*}
We see that $\theta_{1}(\bq) \in \BA$, and 
$\theta_{1}(\theta_{1}(\bq))=\bq$. Furthermore, we deduce that 
$\bF{k-1}{\g}{\theta_{1}(\bq)}=-\bF{k-1}{\g}{\bq}$. 
This proves part (1). 

\subsection{Proof of (2).}
\label{subsec:prfmat2-2}
Let $\bq=((\bp,M) \mid \bm) \in \BB_{2}=
  (\hspmw{k-1}{\g})_{\RB_{2}\RY_{3}}^{\RT_{2}} \sqcup 
  (\hspmw{k-1}{\g})_{\RB_{3}^{1}\RY_{3}}^{\RT_{4}} \sqcup 
  (\hspmw{k-1}{\g})_{\RB_{3}^{2}\RY_{3}}^{\RT_{2}}$. 
Let $(a,k)$ be the final label of $\bp_{(*,k)}^{(k-1,k)}$; 
if $\bp_{(*,k)}^{(k-1,k)} = \emptyset$, then we set $a:=0$. 
Let $(b,k)$ be the initial label of $\bm_{(*,k)}$; 
if $\bm_{(*,k)} = \emptyset$, then we set $b:=0$. 
We define
\begin{equation*} 
\theta_{2}(\bq):=
\begin{cases}
((\bp \setminus (a,k)_{\kappa},M) \mid (a,k)_{\iota} \ast \bm) \\[2mm]
\hspace*{10mm} 
\text{if $\bq \in (\hspmw{k-1}{\g})_{\RB_{3}^{2} \RY_3}^{\RT_2}$ and $a > b$}, \\[3mm]
((\bp \ast (b,k)_{\kappa},M) \mid \bm \setminus (b,k)_{\iota}) \\[2mm]
\hspace*{10mm}
\text{if $\bq \in (\hspmw{k-1}{\g})_{\RB_{2} \RY_3}^{\RT_2} \sqcup 
      (\hspmw{k-1}{\g})_{\RB_{3}^{1} \RY_3}^{\RT_{4}}$, or } \\[2mm]
\hspace*{10mm} 
\text{if $\bq \in (\hspmw{k-1}{\g})_{\RB_{3}^{2} \RY_3}^{\RT_2}$ and $a < b$}.
\end{cases}
\end{equation*}
We see that $\theta_{2}(\bq) \in \BB_{2}$, 
and $\theta_{2}(\theta_{2}(\bq))=\bq$. Furthermore, we deduce that 
$\bF{k-1}{\g}{\theta_{2}(\bq)}=-\bF{k-1}{\g}{\bq}$.
This proves part (2). 

\subsection{Proof of (3).}
\label{subsec:prfmat2-3}
Let $\bq=((\bp,M) \mid \bm) \in \BB_{4}=
(\hspmw{k-1}{\g})_{\RB_{3}^{1} \RY_1} \sqcup 
(\hspmw{k-1}{\g})_{\RB_{2} \RY_3}^{\RU} \sqcup (\hspmw{k-1}{\g})_{\RB_{3} \RY_3}^{\RU}$; 
recall that $(\hspmw{k-1}{\g})_{\RB_{3} \RY_3}^{\RU} = 
(\hspmw{k-1}{\g})_{\RB_{3}^{1} \RY_3}^{\RU} \sqcup (\hspmw{k-1}{\g})_{\RB_{3}^{2} \RY_3}^{\RU}$. 
Let $(a,k)$ be the final label of $\bp_{(*,k)}^{(k-1,k)}$; 
if $\bp_{(*,k)}^{(k-1,k)} = \emptyset$, then we set $a:=0$. 
Let $(b,k)$ be the initial label of $\bm_{(*,k)}$; 
if $\bm_{(*,k)} = \emptyset$, then we set $b:=0$. 
We define
\begin{equation*} 
\theta_{3}(\bq):=
\begin{cases}
((\bp \setminus (a,k)_{\kappa},M) \mid (a,k)_{\iota} \ast \bm) \\[2mm]
\hspace*{10mm} 
\text{if $\bq \in (\hspmw{k-1}{\g})_{\RB_{3}^{1} \RY_1}$, or } \\[2mm]
\hspace*{10mm} 
\text{if $\bq \in (\hspmw{k-1}{\g})_{\RB_{3}^{1} \RY_3}^{\RU}$ and $a > b$}, \\[3mm]
((\bp \ast (b,k)_{\kappa},M) \mid \bm \setminus (b,k)_{\iota}) \\[2mm]
\hspace*{10mm} 
\text{if $\bq \in (\hspmw{k-1}{\g})_{\RB_{2} \RY_3}^{\RU} \sqcup 
      (\hspmw{k-1}{\g})_{\RB_{3}^{2} \RY_3}^{\RU}$, or } \\[2mm]
\hspace*{10mm} 
\text{if $\bq \in (\hspmw{k-1}{\g})_{\RB_{3}^{1} \RY_3}^{\RU}$ and $a < b$}.
\end{cases}
\end{equation*}
We see that $\theta_{3}(\bq) \in \BB_{4}$, 
and $\theta_{3}(\theta_{3}(\bq))=\bq$. Furthermore, we deduce that 
$\bF{k-1}{\g}{\theta_{3}(\bq)}=-\bF{k-1}{\g}{\bq}$. 
This proves part (3). 

\subsection{Proof of (4).}
\label{subsec:prfmat2-4}

Let $\bq=((\bp,M) \mid \bm) \in \hspm{k-2}{\g-1}_{\RD_2\RY}$, 
and write $\bp$ and $\bm$ as:
%
%
\begin{equation} \label{eq:51a1}
\bp = (w\,;\,\dots\dots,
\overbrace{(i_{1},k),\dots,(i_{s},k)}^{ =\,\bp_{(*,k)} },
\overbrace{(j_{1},k-1),\dots,(j_{t},k-1)}^{ =\,\bp_{(*,k-1)} }), 
\end{equation}
with $t \ge 1$, and 
%
%
\begin{equation} \label{eq:51a2}
\bm = (\ed(\bp)\,;\,
\overbrace{(c_{1},k),\dots,(c_{u},k)}^{ =\,\bm_{(*,k)} },
\overbrace{(k,d_{r}),\dots,(k,d_{1})}^{ =\,\bm_{(k,*)} });
\end{equation}
if $u=0$, i.e., $\bm_{(*,k)} = \emptyset$, then we set $c_{1}:=0$. 
Note that $0 \le c_{1} \le k-2$. We consider the directed path 
\begin{equation*}
\begin{split}
\bp_{1}:=
  (w\,;\,& \dots\dots,(i_{1},k),\dots,(i_{s},k), \\
& (j_{1},k-1),\dots,(j_{t},k-1),(k-1,k),(k-1,k) );
\end{split}
\end{equation*}
notice that $\ed(\bp_{1})=\ed(\bp)$ and $\Q(\bp_{1}) = Q_{k-1}Q(\bp)$. 
Recall from \eqref{eq:D2x} that 
\begin{equation*}
\begin{split}
(w\,;\, & \dots\dots,(i_{1},k),\dots,(i_{s},k),(j_{t(\bp)},k),
  (j_{t(\bp)+1},k),\dots,(j_{t},k), \\
& (k-1,k),(j_{1},k),\dots,(j_{t(\bp)-1},k),(j_{t(\bp)},k))
\end{split}
\end{equation*}
is a directed path; 
note that $i_{s} < j_{t(\bp)}$ (see the comment preceding \eqref{eq:D2x}). 
We claim that $j_{t(\bp)} > c_{1}$. If $c_{1}=0$, then the claim is obvious. 
Assume that $c_{1} > 0$. Then, 
\begin{equation} \label{eq:51a3}
\begin{split}
(w\,;\, & \dots\dots,(i_{1},k),\dots,(i_{s},k),(j_{t(\bp)},k),
  (j_{t(\bp)+1},k),\dots,(j_{t},k), \\
& (k-1,k),(j_{1},k),\dots,(j_{t(\bp)-1},k),(j_{t(\bp)},k),(c_1,k))
\end{split}
\end{equation}
is a directed path. By using Lemma~\ref{lem:int}\,(2) repeatedly, 
we see that
\begin{equation*}
\begin{split}
(w\,;\, & \dots\dots,(i_{1},k),\dots,(i_{s},k),(k-1,k),\\
& (j_{t(\bp)},k-1),
  (j_{t(\bp)+1},k-1),\dots,(j_{t},k-1), \\
& (j_{1},k),\dots,(j_{t(\bp)-1},k),(j_{t(\bp)},k),(c_1,k))
\end{split}
\end{equation*}
is a directed path; 
note that $c_{1} \notin \{j_1,\dots,j_{t(\bp)}\}$ by Lemma~\ref{lem:noak}. 
Hence we deduce by Lemma~\ref{lem:LS29b} that
$j_{t(\bp)} > c_{1}$, as desired. 
Define a directed path $\bp'$ by removing 
the segment $(j_{t(\bp)},k),(c_1,k)$ from the directed path \eqref{eq:51a3}. 
Also, define $M'$ by replacing each label of the form $(j_{t'},k-1)$, 
$t(\bp) \le t' \le t$, in $M$ with $(j_{t'},k)$, and then adding $(k-1,k)$ to 
the resulting set. We set 
\begin{align*}
(j_{t(\bp)},k)_{\iota} \ast \bm := 
(\ed(\bp) \cdot (j_{t(\bp)},k) \,;\,(j_{t(\bp)},k),(c_{1},k),\dots,(c_{u},k),
  (k,d_{r}),\dots,(k,d_{1})). 
\end{align*}
We can check that 
\begin{equation*}
\theta_{4}(\bq) := ((\bp',M') \mid (j_{t(\bp)},k)_{\iota} \ast \bm) 
\in \BB_{1} = 
  (\hspmw{k-1}{\g})_{\RB_{2}\RY_{3}}^{\RT_{1}} \sqcup 
  (\hspmw{k-1}{\g})_{\RB_{3}^{1}\RY_{3}}^{\RT_{3}} \sqcup 
  (\hspmw{k-1}{\g})_{\RB_{3}^{2}\RY_{3}}^{\RT_{1}}; 
\end{equation*}
note that $\bF{k-1}{\g}{\theta_{4}(\bq)}=\Q_{k-1}\bF{k-2}{\g-1}{\bq}$.

We show the bijectivity of the map $\theta_{4}$ by giving its inverse. 
Let $\bq=((\bp,M) \mid \bm) \in \BB_{1}$, 
and write $\bp$ and $\bm$ as:
\begin{equation} \label{eq:12e2}
\begin{split}
& \bp = (w\,;\,\dots\dots,
\overbrace{(i_{1},k),\dots,(i_{s},k),(k-1,k),(j_{1},k),\dots,(j_{t},k)}^{ =\,\bp_{(\ast,k)} } ), \\[1mm] 
& \bm = ( \ed(\bp)\,;\,
\underbrace{(c_{1},k),\dots,(c_{u},k)}_{%
 =\,\bm_{(\ast,k)} }, (k,d_{r}),\dots,(k,d_{1}) ), 
\end{split}
\end{equation}
where $s,u \ge 1$, $t,r \ge 0$, $1 \le c_{1} \le k-2$, and 
$c_{1} \in \{i_{1},\dots,i_{s}\}$ (see Remark~\ref{rem:B23Y1}). 
Let $1 \le s' \le s$ be such that $i_{s'}=c_{1}$. 
We consider 
\begin{equation*}
(w\,;\,\dots\dots,
(i_{1},k),\dots,\underbrace{(i_{s'},k)}_{=\,(c_{1},k)},\dots,(i_{s},k),(k-1,k),
(j_{1},k),\dots,(j_{t},k), (c_{1},k)). 
\end{equation*}
By Lemma~\ref{lem:int}\,(2), 
\begin{equation*}
\begin{split}
(w\,;\, & \dots\dots,
  (i_{1},k),\dots,(i_{s'-1},k),(k-1,k), \\ 
& \underbrace{(i_{s'},k-1)}_{=\,(c_{1},k-1)},\dots,(i_{s},k-1),
  (j_{1},k),\dots,(j_{t},k), (c_{1},k))
\end{split}
\end{equation*}
is a directed path. 
Using Lemma~\ref{lem:int}\,(1), we obtain the directed path
\begin{equation*}
\begin{split}
(w\,;\, & \dots\dots,
  (i_{1},k),\dots,(i_{s'-1},k),(k-1,k), \\ 
& (j_{1},k),\dots,(j_{t},k), 
  \underbrace{(i_{s'},k-1)}_{=\,(c_{1},k-1)},
  \underbrace{(c_{1},k)}_{=\,(i_{s'},k)}, 
  (i_{s'+1},k-1), \dots,(i_{s},k-1)). 
\end{split}
\end{equation*}
By Lemma~\ref{lem:int}\,(2), we see that 
\begin{equation*}
\begin{split}
(w\,;\, & \dots\dots,
  (i_{1},k),\dots,(i_{s'-1},k),(k-1,k), \\ 
& (j_{1},k),\dots,(j_{t},k), 
  (k-1,k),(i_{s'},k-1), 
  (i_{s'+1},k-1), \dots,(i_{s},k-1))
\end{split}
\end{equation*}
is a directed path. 
Then, by Lemma~\ref{lem:int}\,(2), 
%
%
\begin{equation} \label{eq:51a4}
\begin{split}
(w\,;\, & \dots\dots,
  (i_{1},k),\dots,(i_{s'-1},k),(k-1,k),(k-1,k), \\ 
& (j_{1},k-1),\dots,(j_{t},k-1), \underbrace{(i_{s'},k-1)}_{=\,(c_{1},k-1)}, 
  (i_{s'+1},k-1), \dots,(i_{s},k-1))
\end{split}
\end{equation}
is a directed path. 
We define a directed path $\bp''$ by removing 
the segment $(k-1,k),(k-1,k)$ from the directed path \eqref{eq:51a4};
note that $\ed(\bp'')=\ed(\bp)$ and $\Q(\bp'') = \Q_{k-1}^{-1}\Q(\bp)$. 
Recall that if $t > 0$ and $n_{(j_{t},*)}(\bp) \ge 2$, 
then $j_{t} < c_{1} = i_{s'}$. 
Also, define $M''$ by replacing each label of the form $(i_{s''},k)$, 
$s' \le s'' \le s$, in $M$ by $(i_{s''},k-1)$, and then removing $(k-1,k)$ from 
the resulting set. We set 
\begin{align*}
\bm \setminus (c_{1},k)_{\iota} : =
(\ed(\bp) \cdot (c_{1},k) \,;\,(c_{2},k),\dots,(c_{u},k),
  (k,d_{r}),\dots,(k,d_{1})). 
\end{align*}
We can check that 
\begin{equation*}
\theta_{4}'(\bq) := ((\bp'',M'') \mid \bm \setminus (c_{1},k)_{\iota}) 
\in \hspm{k-2}{\g-1}_{\RD_{2}\RY}; 
\end{equation*}
note that $\bF{k-2}{\g-1}{\theta_{4}'(\bq)}=\Q_{k-1}^{-1}\bF{k-1}{\g}{\bq}$. 
It is easily verified that $\theta_{4}'$ is the inverse of $\theta_{4}$. 
This proves part (4).

%
\section{Proof of Proposition~\ref{prop:mat3a}.}
\label{sec:prfmat3}

Here again we list the conditions needed in this section: 
\begin{itemize}
\item For $((\bp,M) \mid \bm) \in \hspmw{k-1}{p}$, 
\begin{equation*}
\begin{cases}
(\RA_1) & n_{(k-1,k)}(\bp)=0 \text{ and } \bp_{(*,k)} = \emptyset, \\
(\RY_2) & \iota(\bm) \ne (k-1,k),\, \bm_{(*,k)} = \emptyset, \text{ and } \bm_{(k,*)} \ne \emptyset, \\
(\RE) & \bp_{(*,k)} \ne \emptyset,\,\kappa(\bp) \in M,\,\bm_{(*,k)}=\emptyset, 
        \text{ and } \bm_{(k,*)} \ne \emptyset, \\
(\RG) & \bm = \emptyset, \, \bp_{(*,k)} \ne \emptyset, \text{ and } \kappa(\bp) \in M. 
\end{cases}
\end{equation*}

\item For $((\bp,M) \mid \bm) \in \hspmw{k-1}{p-1}$, 
\begin{equation*}
\begin{cases}
(\RE) & \bp_{(*,k)} \ne \emptyset,\,\kappa(\bp) \in M,\,\bm_{(*,k)}=\emptyset, 
\text{ and } \bm_{(k,*)} \ne \emptyset, \\
(\RF) & \bm=\emptyset,\,\bp_{(*,k)} \ne \emptyset, \text{ and }
\kappa(\bp) \notin M, \\
(\RF_1) & \bm=\emptyset, \, \bp_{(*,k)} \ne \emptyset,\,
\kappa(\bp) \notin M, \text{ and } n_{(a,*)}(\bp) = 1, \\
(\RF_2) & \bm=\emptyset, \, \bp_{(*,k)} \ne \emptyset,\,
\kappa(\bp) \notin M, \text{ and } n_{(a,*)}(\bp) \ge 2, \\
(\RF_2^1) & 
\bm=\emptyset, \, \bp_{(*,k)} \ne \emptyset,\,
\kappa(\bp) \notin M, \, n_{(a,*)}(\bp) \ge 2, \text{ and } 
\kappa'(\bp) \in M, \\
(\RF_2^2) & 
\bm=\emptyset, \, \bp_{(*,k)} \ne \emptyset,\,
\kappa(\bp) \notin M, \, n_{(a,*)}(\bp) \ge 2, \text{ and } 
\kappa'(\bp) \not\in M, 
\end{cases}
\end{equation*}
where in conditions $\RF_1$, $\RF_2$, $\RF_2^1$, and $\RF_2^2$, 
we write $\kappa(\bp) = (a,k)$ for some $1 \le a \le k-1$. 
Also, for the definition of $\kappa'(\bp)$, see Definition~\ref{dfn:kap1}. 

\item For $(\bp, M) \in \hspw{k}{p}$, 
\begin{equation*}
\begin{cases}
(\RR) & n_{(k,*)}(\bp)=0, \\
(\RS) & n_{(k,*)}(\bp) \ge 1, \\
(\RS_1) & n_{(k,*)}(\bp) \ge 1 \text{ and } (k,b(\bp)) \in M, \\
(\RS_2) & n_{(k,*)}(\bp) \ge 1 \text{ and } (k,b(\bp)) \not\in M, \\
(\RS_1^1) & n_{(k,*)}(\bp) \ge 1,\, (k,b(\bp)) \in M, \text{ and } \\
          & \text{$(k,b(\bp))$ is the final label of $\bp_{(\ast,b(\bp))}$}, \\
(\RS_1^2) & n_{(k,*)}(\bp) \ge 1,\, (k,b(\bp)) \in M, \text{ and } \\
          & \text{$(k,b(\bp))$ is not the final label of $\bp_{(\ast,b(\bp))}$}, \\ 
(\RS_1^{\ta}) & (\RS_1^2) \text{ and } \kappa''(\bp) \in M, \\
(\RS_1^{\tb}) & (\RS_1^2) \text{ and } \kappa''(\bp) \not\in M, 
\end{cases}
\end{equation*}
where $b(\bp) = \max \{b \ge k+1 \mid (k,b) \in \bp\bigr\}$. 
Also, for the definition of $\kappa''(\bp)$, see Definition~\ref{dfn:kap2}. 
\end{itemize}
Recall from \eqref{eq:bF} the definition of 
$\bF{\h}{\g}{\bq}$, $\bq \in \hspmw{\h}{\g}$, and 
from \eqref{eq:bEkp} the definition of 
$\bE{k}{p}{\bq}$, $\bq \in \hspw{k}{p}$. 
Proposition~\ref{prop:mat3a} follows from the next proposition. 
%
%
\begin{prop} \label{prop:mat3} \mbox{}
\begin{enu}
\item There exists a bijection 
$\chi_{1}:(\hspmw{k-1}{p})_{\RA_1\RY_2} \rightarrow (\hspw{k}{p})_{\RS_2}$ 
satisfying the condition that 
\begin{equation*}
\bE{k}{p}{\chi_{1}(\bq)} = \bF{k-1}{p}{\bq} 
\quad \text{\rm for $\bq \in (\hspmw{k-1}{p})_{\RA_1\RY_2}$}.
\end{equation*} 

\item There exists a bijection 
$\chi_{2}:(\hspmw{k-1}{p})_{\RE} \rightarrow 
(\hspw{k}{p})_{\RS_1^{\tb}} \sqcup (\hspmw{k-1}{p-1})_{\RF_2^2}$ 
satisfying the conditions that 
\begin{align*}
& \bE{k}{p}{\chi_{2}(\bq)} = \bF{k-1}{p}{\bq} 
\quad \text{\rm for $\bq \in (\hspmw{k-1}{p})_{\RE}$ such that 
$\chi_{2}(\bq) \in (\hspw{k}{p})_{\RS_1^{\tb}}$}, \\
& \bF{k-1}{p-1}{\chi_{2}(\bq)} = 
- \bF{k-1}{p}{\bq} \quad \text{\rm for 
$\bq \in (\hspmw{k-1}{p})_{\RE}$ such that 
$\chi_{2}(\bq) \in (\hspmw{k-1}{p-1})_{\RF_2^2}$}.
\end{align*} 

\item There exists a bijection 
$\chi_{3}:(\hspmw{k-1}{p})_{\RA_1\emptyset} \rightarrow (\hspw{k}{p})_{\RR}$ 
satisfying the condition that 
\begin{equation*} 
\bE{k}{p}{\chi_{3}(\bq)} = \bF{k-1}{p}{\bq}
\quad \text{\rm for $\bq \in (\hspmw{k-1}{p})_{\RA_1\emptyset}$}. 
\end{equation*} 

\item There exists a bijection 
$\chi_{4}:(\hspmw{k-1}{p})_{\RG} \rightarrow (\hspmw{k-1}{p-1})_{\RF_1}$ 
satisfying the condition that 
\begin{equation*} 
\bF{k-1}{p-1}{\chi_{4}(\bq)} = - \bF{k-1}{p}{\bq} 
\quad \text{\rm for $\bq \in (\hspmw{k-1}{p})_{\RG}$}.
\end{equation*} 

\item There exists a bijection
 $\chi_{5}:(\hspmw{k-1}{p-1})_{\RA_1\RY_2} \rightarrow (\hspw{k}{p})_{\RS_1^1}$ 
satisfying the condition that 
\begin{equation*} 
\bE{k}{p}{\chi_{5}(\bq)} = - \bF{k-1}{p-1}{\bq} \quad 
\text{for \rm $\bq \in (\hspmw{k-1}{p-1})_{\RA_1\RY_2}$}. 
\end{equation*} 

\item There exists a bijection 
$\chi_{6}:(\hspmw{k-1}{p-1})_{\RE} \rightarrow 
(\hspw{k}{p})_{\RS_1^{\ta}} \sqcup (\hspmw{k-1}{p-1})_{\RF_2^1}$ 
satisfying the conditions that 
\begin{align*}
& \bE{k}{p}{\chi_{6}(\bq)} = 
  - \bF{k-1}{p-1}{\bq} \quad \text{\rm for 
  $\bq \in (\hspmw{k-1}{p-1})_{\RE}$ such that 
  $\chi_{6}(\bq) \in (\hspw{k}{p})_{\RS_1^{\ta}}$}, \\
& \bF{k-1}{p-1}{\chi_{6}(\bq)} = 
  \bF{k-1}{p-1}{\bq} \quad \text{\rm for 
  $\bq \in (\hspmw{k-1}{p-1})_{\RE}$ such that 
  $\chi_{6}(\bq) \in (\hspmw{k-1}{p-1})_{\RF_2^1}$}. 
\end{align*}
\end{enu}
\end{prop}
In order to prove Proposition~\ref{prop:mat3}, we freely make use of 
two procedures, that is, insertion and deletion, which are defined 
in Appendix~\ref{sec:ins}; 
we also use the notation of Appendix~\ref{sec:ins}. 

\subsection{Proofs of (1) and (5).}
\label{subsec:prfmat3-15}

Let $\g \in \{p-1,p\}$. 
Let $\bq=((\bp,M) \mid \bm) \in (\hspmw{k-1}{\g})_{\RA_1\RY_2}$, and write $\bp$ and $\bm$ as:
\begin{equation} \label{eq:31p}
\bp = (w\,;\,\bp_{(*,d)}, \bp_{(*,d-1)},\dots,\bp_{(*,k+1)},\bp_{(*,k)}), 
\end{equation}
\begin{equation} \label{eq:31m}
\bm = (\ed(\bp)\,;\,(k,d_{r}),\dots,(k,d_{1})),
\end{equation}
for $d \ge d_{r} > \cdots > d_{1} \ge k+1$; 
note that $r \ge 1$. We define 
\begin{equation} \label{eq:zeta}
(\bp \leftarrow \bm):=
( \cdots ((\bp \leftarrow (k,d_{r})) \leftarrow (k,d_{r-1})) 
 \leftarrow \cdots \leftarrow (k,d_{1}) ); 
\end{equation}
note that $\bp \leftarrow \bm$ is the directed path obtained 
by adding $(k,d_{t})$ to the end of $\bp_{(*,d_{t})}$ in $\bp$ (of the form \eqref{eq:31p})
for $1 \le t \le r$. 
If $\g = p$, then we set  $\chi_{1}(\bq):=(\bp \leftarrow \bm,M)$; 
it is easily seen that $\chi_{1}(\bq) \in \hsp{k}{p}_{\RS_2}$, and 
$\bE{k}{p}{\chi_{1}(\bq)} = \bF{k-1}{p}{\bq}$. Similarly, 
if $\g = p-1$, then we set  $\chi_{5}(\bq):=(\bp \leftarrow \bm,M \sqcup \{(k,d_{r})\})$; 
it is easily seen that $\chi_{5}(\bq) \in \hsp{k}{p}_{\RS_1^1}$, and 
$\bE{k}{p}{\chi_{5}(\bq)} = - \bF{k-1}{p-1}{\bq}$. 

We show the bijectivity of the maps $\chi_{1}$ and $\chi_{5}$ 
by giving their inverses. 
Let $\bq=(\bp,M) \in \hsp{k}{p}_{\RS_2} \sqcup \hsp{k}{p}_{\RS_1^1}$. 
Let 
\begin{equation} \label{eq:31d}
\{d_{r} > \cdots > d_{1} \} = \{ d \ge k+1 \mid (k,d) \in \bp \}; 
\end{equation}
note that $(k,d_{t})$ is the final label of 
$\bp_{(*,d_{t})}$ for $1 \le t \le r$. Then we set
\begin{equation} \label{eq:xi}
\begin{split}
\xi(\bq) & := ( \cdots ((\bp \rightarrow (k,d_{1})) \rightarrow (k,d_{2})) 
 \rightarrow \cdots \rightarrow (k,d_{r}) ), \\
\mu(\bq) & := (\ed(\bp)\,;\,(k,d_{r}),\dots,(k,d_{1}));
\end{split}
\end{equation}
observe that $\xi(\bq)$ is the directed path obtained from $\bp$ 
by removing $(k,d_{t})$ from the end of $\bp_{(*,d_{t})}$ in $\bp$ 
for $1 \le t \le r$. 
%
%
If $\bq \in \hsp{k}{p}_{\RS_2}$, then we set 
$\chi_{1}'(\bq):=((\xi(\bq),M) \mid \mu(\bq))$; 
it is easily verified that $\chi_{1}'(\bq) \in (\hspmw{k-1}{p})_{\RA_1\RY_2}$, and 
$\chi_{1}'$ is the inverse of $\chi_{1}$. Similarly, 
if $\bq \in \hsp{k}{p}_{\RS_1^1}$, then we set 
$\chi_{5}'(\bq):=((\xi(\bq),M \setminus \{(k,d_{r})\}) \mid \mu(\bq))$; 
it is easily verified that $\chi_{5}'(\bq) \in (\hspmw{k-1}{p-1})_{\RA_1\RY_2}$, and 
$\chi_{5}'$ is the inverse of $\chi_{5}$. This proves parts (1) and (5). 

\subsection{Proofs of (2) and (6).}
\label{subsec:prfmat3-26}

Let $\g \in \{p-1,p\}$. 
Let $\bq=((\bp,M) \mid \bm) \in (\hspmw{k-1}{\g})_{\RE}$, and 
write $\bp$ and $\bm$ as in \eqref{eq:31p} and \eqref{eq:31m}, respectively 
(see also Remark~\ref{rem:can2}). We define 
\begin{equation*}
\zeta_{t}(\bp):=( \cdots ((\bp \leftarrow (k,d_{r})) \leftarrow (k,d_{r-1})) 
 \leftarrow \cdots \leftarrow (k,d_{t}) ) \quad \text{for $1 \le t \le r$},  
\end{equation*}
and $\zeta_{1}(\bp):=(\bp \leftarrow \bm)$. 

Assume that in the sequence of insertions for obtaining $\bp \leftarrow \bm$, 
a directed path of the form \eqref{eq:Ia} appears when $(k,d_{u})$ is inserted for some $1 \le u \le r$; 
take the largest $u$ satisfying this condition.
Then there exist segments $\bs_{u}',\bs_{u+1}',\dots,\bs_{r-1}',\bs_{r}'$ 
in $\bp_{(*,k)}$ satisfying the following conditions: 
\begin{itemize}
\item[(1)] $\iota(\bs_{u}') = \iota(\bp_{(*,k)})$, 
$\kappa(\bs_{r}') = \kappa(\bp_{(*,k)})$, and 
$\kappa(\bs_{t}') = \iota(\bs_{t+1}')$ for $u \le t \le r-1$; 

\item[(2)] $\zeta_{u}(\bp)$ is the directed path obtained from $\bp$ by removing $\bp_{(*,k)}$, then 
adding $\bs_{t}$ to the end of $\bp_{(*,d_{t})}$ in $\bp$ for $u+1 \le t \le r$, and 
adding $(k,d_{u}),\bs_{u}$ to the end of $\bp_{(*,d_{u})}$, 
where $\bs_{t}$ is defined by replacing $(i,k)$ in $\bs_{t}'$ with $(i,d_{t})$ for $u \le t \le r$. 
\end{itemize}

\noindent
Also, we deduce that $\zeta_{1}(\bp)=(\bp \leftarrow \bm)$ is the directed path obtained by adding 
$(k,d_{t})$ to the end of $\bp_{(*,d_{t})}$ in $\zeta_{u}(\bp)$ for $1 \le t < u$. 
We set $K_{1}:= \bp_{(*,k)} \cap M$; note that for each $(i,k) \in K_{1}$ with $(i,k) \ne \kappa(\bp)$, 
there exists a unique $u+1 \le t_{i} \le r$ such that $(i,k) \in \bs_{t_{i}}$ 
and $(i,k) \ne \kappa(\bs_{t_{i}})$. We set $K_{2}:=\bigl\{ (i,d_{t_i}) \mid 
\text{$(i,k) \in K_{1}$ with $(i,k) \ne \kappa(\bp)$} \bigr\}$, and then 
\begin{equation*}
M_{\bq}:=
 \begin{cases}
 (M \setminus K_{1}) \sqcup K_{2} \sqcup \{(k,d_{u})\}
 & \text{if $\g = p$}, \\[1mm]
 (M \setminus K_{1}) \sqcup K_{2} \sqcup \{(k,d_{u}),\kappa(\bs_{r})\}
 & \text{if $\g = p-1$}. 
\end{cases}
\end{equation*}
We deduce that if $\g = p$, then 
$\chi_{2}(\bq):=(\bp \leftarrow \bm,M_{\bq}) \in (\hspw{k}{p})_{\RS_{1}^{\tb}}$ and 
$\bE{k}{p}{\chi_{2}(\bq)} = \bF{k-1}{p}{\bq}$, and that 
if $\g = p-1$, then 
$\chi_{6}(\bq):=(\bp \leftarrow \bm,M_{\bq}) \in (\hspw{k}{p})_{\RS_{1}^{\ta}}$ and 
$\bE{k}{p}{\chi_{6}(\bq)} = - \bF{k-1}{p-1}{\bq}$. 

Assume that in the sequence of insertions for obtaining $\bp \leftarrow \bm$, 
a directed path of the form \eqref{eq:Ia} does not appear when $(k,d_{t})$ is inserted for any $1 \le t \le r$. 
Then there exist segments $\bs_{0}',\bs_{1}',\dots,\bs_{r-1}',\bs_{r}'$ 
in $\bp_{(*,k)}$ satisfying the following conditions: 
\begin{itemize}
\item[(1)'] $\iota(\bs_{0}') = \iota(\bp_{(*,k)})$, 
$\kappa(\bs_{r}') = \kappa(\bp_{(*,k)})$, and 
$\kappa(\bs_{t}') = \iota(\bs_{t+1}')$ for $0 \le t \le r-1$; 

\item[(2)'] $\zeta_{1}(\bp)$ is the directed path obtained by removing 
$(\bs_{1}' \cup \cdots \cup \bs_{r}') \setminus \{\iota(\bs_{1}')\}$ from $\bp_{(*,k)}$, and then adding 
$\bs_{t}$ to the end of $\bp_{(*,d_{t})}$ in $\bp$ for $1 \le t \le r$, 
where $\bs_{t}$ is defined by replacing $(i,k)$ in $\bs_{t}'$ with $(i,d_{t})$ 
for $1 \le t \le r$. 
\end{itemize}
We set $K_{1}:= (\bs_{1}' \cup \cdots \cup \bs_{r}') \cap M$; 
note that for each $(i,k) \in K_{1}$, 
there exists a unique $1 \le t_{i} \le r$ such that $(i,k) \in \bs_{t_{i}}'$ 
and $(i,k) \ne \kappa(\bs_{t_{i}}')$. We set $K_{2}:=\bigl\{ (i,d_{t_i}) \mid 
(i,k) \in K_{1} \bigr\}$, and
\begin{equation*}
M_{\bq}:=
 \begin{cases}
 (M \setminus K_{1}) \sqcup (K_{2} \setminus \{\kappa(\bs_{r})\})
 & \text{if $\g = p$}, \\[1mm]
 (M \setminus K_{1}) \sqcup K_{2}
 & \text{if $\g = p-1$}. 
\end{cases}
\end{equation*}
We deduce that if $\g = p$, then 
$\chi_{2}(\bq):=( (\bp \leftarrow \bm,M_{\bq}) \mid \emptyset) \in (\hspmw{k-1}{p-1})_{\RF_{2}^{2}}$ and 
$\bF{k-1}{p-1}{\chi_{2}(\bq)} = - \bF{k-1}{p}{\bq}$, and that 
if $\g = p-1$, then 
$\chi_{6}(\bq):=( (\bp \leftarrow \bm,M_{\bq}) \mid \emptyset) \in (\hspmw{k-1}{p-1})_{\RF_{2}^{1}}$ and 
$\bF{k-1}{p-1}{\chi_{6}(\bq)} = \bF{k-1}{p-1}{\bq}$. 


Let us show the bijectivity of the maps $\chi_{2}$ and $\chi_{6}$ 
by giving their inverses. First, let $\bq=(\bp,M) \in 
(\hspw{k}{p})_{\RS_{1}^{\tb}} \sqcup (\hspw{k}{p})_{\RS_{1}^{\ta}}$. 
Recall from Definition~\ref{dfn:kap2} the definitions of $\jp(\bp)$ and 
$f'_{\jp}=f'_{\jp}(\bp)$ for $0 \le \jp \le \jp(\bp)$; observe that 
$f'_{0} < f'_{1} < \cdots < f'_{\jp(\bp)}$. Also, let 
\begin{equation*}
\{d_{u} > \cdots > d_{1} \} = \{ d \ge k+1 \mid (k,d) \in \bp \}; 
\end{equation*}
notice that $d_{u}=b(\bp)=f_{0}'$. 
We set $r:=u+\jp(\bp)$, and $d_{u+\jp}:=f_{\jp}'$ for $0 \le \jp \le \jp(\bp)$. 
Then we define
\begin{equation*}
\begin{split}
\xi(\bq) & :=( \cdots ((\bp \rightarrow (k,d_{1})) \rightarrow (k,d_{2})) 
 \rightarrow \cdots \rightarrow (k,d_{r}) ), \\
\mu(\bq) & :=(\ed(\bp)\,;\,(k,d_{r}),\dots,(k,d_{1})).
\end{split}
\end{equation*}
For each label $(i,k)$ in the $(*,k)$-segment $\xi(\bq)_{(*,k)}$ of $\xi(\bq)$, 
there exists a unique $d(i) \in \{d_{s} \mid u \le s \le r \}$ 
satisfying the conditions that $(i,d(i)) \in \bp$ and that 
$(i,d(i)) \ne \kappa(\bp_{(*,d(i))})$ if $(i,k) \ne \kappa(\xi(\bq))$. 
We set $K_{2}':=M \cap \bigl\{(i,d(i)) \mid (i,k) \in \xi(\bq) \bigr\}$, 
$K_{1}':=\bigl\{ (i,k) \in \xi(\bq) \mid (i,d(i)) \in K_{2}' \bigr\}$, 
and then define 
\begin{equation*}
M^{\bq}:= 
 \begin{cases}
 (M \setminus ( K_{2}' \sqcup \{ (k,d_{u}) \})) \sqcup (K_{1}' \sqcup \{\kappa(\xi(\bq))\})
 & \text{if $(\bp,M) \in \hsp{k}{p}_{\RS_{1}^{\tb}}$}, \\[2mm]
 (M \setminus ( K_{2}' \sqcup \{ (k,d_{u}) \})) \sqcup K_{1}'
 & \text{if $(\bp,M) \in \hsp{k}{p}_{\RS_{1}^{\ta}}$}. 
 \end{cases} 
\end{equation*}
If $(\bp,M) \in \hsp{k}{p}_{\RS_{1}^{\tb}}$, then 
we set $\chi_{2}'(\bq):=((\xi(\bq),M^{\bq}) \mid \mu(\bq))$; 
we see that $\chi_{2}'(\bq) \in (\hspmw{k-1}{p})_{\RE}$.
Similarly, if $(\bp,M) \in \hsp{k}{p}_{\RS_{1}^{\ta}}$, then 
we set $\chi_{6}'(\bq):=((\xi(\bq),M^{\bq}) \mid \mu(\bq))$; 
we see that $\chi_{6}'(\bq) \in (\hspmw{k-1}{p-1})_{\RE}$.

Next, let $\bq=((\bp,M) \mid \emptyset) \in 
(\hspmw{k-1}{p-1})_{\RF_{2}^{2}} \sqcup (\hspmw{k-1}{p-1})_{\RF_{2}^{1}}$. 
Recall from Definition~\ref{dfn:kap1} the definitions of $\ip(\bp)$ and 
$f_{\ip}=f_{\ip}(\bp)$ for $0 \le \ip \le \ip(\bp)$; observe that 
$k=f_{0} < f_{1} < \cdots < f_{\ip(\bp)}$. 
We set $r:=\ip(\bp)$, and $d_{s}:=f_{s}$ for $0 \le s \le r=\ip(\bp)$. 
Then we define
\begin{equation*}
\begin{split}
\xi(\bq) & :=( \cdots ((\bp \rightarrow (k,d_{1})) \rightarrow (k,d_{2})) 
 \rightarrow \cdots \rightarrow (k,d_{r}) ), \\
\mu(\bq) & :=(\ed(\bp)\,;\,(k,d_{r}),\dots,(k,d_{1})).
\end{split}
\end{equation*}
For each label $(i,k)$ in the $(*,k)$-segment $\xi(\bq)_{(*,k)}$ of $\xi(\bq)$, 
there exists a unique $d(i) \in \{d_{s} \mid 0 \le s \le r \}$ 
satisfying the conditions that $(i,d(i)) \in \bp$ and that 
$(i,d(i)) \ne \kappa(\bp_{(*,d(i))})$ if $(i,k) \ne \kappa(\xi(\bq))$. 
We set $K_{2}':=M \cap \bigl\{(i,d(i)) \mid (i,k) \in \xi(\bq), (i,k) \notin \bp \bigr\}$, 
$K_{1}':=\bigl\{ (i,k) \in \xi(\bq) \mid (i,d(i)) \in K_{2}' \bigr\}$, 
and then define 
\begin{equation*}
M^{\bq}:= 
 \begin{cases}
 (M \setminus ( K_{2}' \sqcup \{ (k,d_{u}) \})) \sqcup (K_{1}' \sqcup \{\kappa(\xi(\bq))\})
 & \text{if $\bq = ((\bp,M) \mid \emptyset) \in (\hspmw{k-1}{p-1})_{\RF_{2}^{2}}$}, \\[2mm]
 (M \setminus ( K_{2}' \sqcup \{ (k,d_{u}) \})) \sqcup K_{1}'
 & \text{if $\bq = ((\bp,M) \mid \emptyset) \in (\hspmw{k-1}{p-1})_{\RF_{2}^{1}}$}. 
 \end{cases} 
\end{equation*}
If $\bq=((\bp,M) \mid \emptyset) \in (\hspmw{k-1}{p-1})_{\RF_{2}^{2}}$, then 
we set $\chi_{2}'(\bq):=((\xi(\bq),M^{\bq}) \mid \mu(\bq))$; 
we see that $\chi_{2}'(\bq) \in (\hspmw{k-1}{p})_{\RE}$.
Similarly, if $\bq=((\bp,M) \mid \emptyset) \in (\hspmw{k-1}{p-1})_{\RF_{2}^{1}}$, then 
we set $\chi_{6}'(\bq):=((\xi(\bq),M^{\bq}) \mid \mu(\bq))$; 
we see that $\chi_{6}'(\bq) \in (\hspmw{k-1}{p-1})_{\RE}$.
Hence we obtain the maps $\chi_{2}'$ and $\chi_{6}'$, 
which are the inverses of the maps $\chi_{2}$ and $\chi_{6}$, 
respectively. This proves parts (2) and (6). 

\subsection{Proof of (3).}
\label{subsec:prfmat3-3}

For $\bq=((\bp,M) \mid \emptyset) \in (\hspmw{k-1}{p})_{\RA_1\emptyset}$, 
we set $\chi_{3}(\bq) := (\bp,M)$. It is easily seen that 
$\chi_{3}(\bq) \in \hsp{k}{p}_{\RR}$, and $\bE{k}{p}{\chi_{3}(\bq)} = \bF{k-1}{p}{\bq}$. 
Also, we deduce that the map $\chi_{3}$ is bijective. This proves part (3).

\subsection{Proof of (4).}
\label{subsec:prfmat3-4}

For $\bq=((\bp,M) \mid \emptyset) \in (\hspmw{k-1}{p})_{\RG}$, 
we set $\chi_{4}(\bq) := ((\bp,M \setminus \{\kappa(\bp)\}) \mid \emptyset)$. 
It is easily seen that 
$\chi_{4}(\bq) \in (\hspmw{k-1}{p-1})_{\RF_1}$, and $\bF{k-1}{p-1}{\chi_{4}(\bq)} = - \bF{k-1}{p}{\bq}$. 
Also, we deduce that the map $\chi_{4}$ is bijective. This proves part (4).

\appendix
%
%
\section{Some lemmas on directed paths in the quantum Bruhat graph.}
\label{sec:lemdp}
The following lemma is used in the proof of Claim~\ref{c472} 
(in order to obtain the inequality $i_{s} < j_{t(\bp)}$).
%
%
\begin{lem}[{cf. \cite[Lemma~2.9]{LS}}] \label{lem:LS29} \mbox{}
\begin{enu}
\item There does not exist a directed path of the form\,{\rm:} 
\begin{equation} \label{eq:LS29-1a}
(v\,;\,(j,m),(i,m),(i,l))
\end{equation}
in $\QBG(S_{\infty})$ for any $v \in S_{\infty}$ and $1 \le i < j < l < m$. 

\item For all $w \in S_{\infty}$ and $1 \le i < j \le k < l < m$, 
no element $\bp \in \SP^{k}(w)$ has 
a segment of the form $(j,m),\dots,(i,m),\dots,(i,l)$. 
\end{enu}
\end{lem}

\begin{proof}
(1) Suppose, for a contradiction, that 
there exists a directed path of the form \eqref{eq:LS29-1a}.
%
%
In what follows, we use Lemma~\ref{lem:edge} frequently without mentioning it; 
note that 
$(v \cdot (j,m))(i) = v(i)$, $(v \cdot (j,m))(m) = v(j)$, 
$(v \cdot (j,m))(j) = v(m)$, $(v \cdot (j,m))(l) = v(l)$, and that 
$(v \cdot (j,m)(i,m))(i) = v(j)$, $(v \cdot (j,m)(i,m))(l) = v(l)$, 
$(v \cdot (j,m)(i,m))(j) = v(m)$.

\paragraph{\bf Case 1.} 
Assume that the edge corresponding to $(j,m)$ is a Bruhat edge; 
in this case, we have
\begin{equation} \label{eq:c1}
v(j) < v(m), \qquad v(l) \not\in [v(j),v(m)].
\end{equation}

\paragraph{\bf Subcase 1.1.} 
Assume that the edge corresponding to $(i,m)$ is a Bruhat edge; 
in this case, we have
\begin{equation} \label{eq:c11}
v(i) < v(j), \qquad 
v(m), v(l) \not\in [v(i),v(j)].
\end{equation}
Combining \eqref{eq:c1} and \eqref{eq:c11}, we see that
$v(i) < v(j) < v(m)$, and that either $v(l) < v(i)$ or $v(m) < v(l)$ holds. 

\paragraph{\bf Subsubcase 1.1.1.} 
Assume that the edge corresponding to $(i,l)$ is a Bruhat edge; 
in this case, we have
\begin{equation} \label{eq:c111}
v(j) < v(l), \qquad 
v(m) \not\in [v(j),v(l)].
\end{equation}
Then we obtain $v(i) < v(j) < v(m) < v(l)$, 
which contradicts $v(m) \not\in [v(j),v(l)]$. 

\paragraph{\bf Subsubcase 1.1.2.} 
Assume that the edge corresponding to $(i,l)$ is a quantum edge; 
in this case, we have
\begin{equation} \label{eq:c112}
v(j) > v(l), \qquad 
v(m) \in [v(l),v(j)].
\end{equation}
Then we obtain $v(l) < v(i) < v(j) < v(m)$, 
which contradicts $v(m) \in [v(l),v(j)]$.

\paragraph{\bf Subcase 1.2.} 
Assume that the edge corresponding to $(i,m)$ is a quantum edge; 
in this case, we have
\begin{equation} \label{eq:c12}
v(i) > v(j), \qquad 
v(m), v(l) \in [v(j),v(i)].
\end{equation}
Combining \eqref{eq:c1} and \eqref{eq:c12}, we see that
$v(j) < v(m) < v(l) < v(i)$. 

\paragraph{\bf Subsubcase 1.2.1.} 
Assume that the edge corresponding to $(i,l)$ is a Bruhat edge. 
In this case, \eqref{eq:c111} holds, 
which contradicts $v(j) < v(m) < v(l) < v(i)$. 

\paragraph{\bf Subsubcase 1.2.2.} 
Assume that the edge corresponding to $(i,l)$ is a quantum edge. 
In this case, \eqref{eq:c112} holds, 
which contradicts $v(j) < v(m) < v(l) < v(i)$.

\paragraph{\bf Case 2.} 
Assume that the edge corresponding to $(j,m)$ is a quantum edge; 
in this case, we have
\begin{equation} \label{eq:c2}
v(j) > v(m), \qquad v(l) \in [v(m),v(j)].
\end{equation}

\paragraph{\bf Subcase 2.1.} 
Assume that the edge corresponding to $(i,m)$ is a Bruhat edge; 
in this case, \eqref{eq:c11} holds. 
Combining \eqref{eq:c2} and \eqref{eq:c11}, we see that
$v(m) < v(l) < v(i) < v(j)$. 

\paragraph{\bf Subsubcase 2.1.1.} 
Assume that the edge corresponding to $(i,l)$ is a Bruhat edge. 
In this case, \eqref{eq:c111} holds, 
which contradicts $v(m) < v(l) < v(i) < v(j)$. 

\paragraph{\bf Subsubcase 2.1.2.} 
Assume that the edge corresponding to $(i,l)$ is a quantum edge. 
In this case, \eqref{eq:c112} holds, 
which contradicts $v(m) < v(l) < v(i) < v(j)$.

\paragraph{\bf Subcase 2.2.} 
Assume that the edge corresponding to $(i,m)$ is a quantum edge; 
in this case, \eqref{eq:c12} holds. 
Combining \eqref{eq:c2} and \eqref{eq:c12}, we see that
$v(m) < v(j) < v(i)$, which contradicts 
$v(m) \in [v(j),v(i)]$. 

This proves part (1). 

(2) By using part (1), we can prove part (2) 
by exactly the same argument as for \cite[Lemma~2.9]{LS}. 
This completes the proof of the lemma. 
\end{proof}
The following lemma is used in the proof of Claim~\ref{c471} 
(in order to obtain the equality $u=t$).
%
%
\begin{lem} \label{lem:LS29a}
There does not exist a directed path of the form\,{\rm:} 
\begin{equation} \label{eq:LS29a-1a}
(v\,;\,(i,l),(i,m),(j,m))
\end{equation}
in $\QBG(S_{\infty})$ for any $v \in S_{\infty}$ 
and $1 \le i < j < l < m$. 
\end{lem}

\begin{proof}
Suppose, for a contradiction, that 
there exists a directed path of the form \eqref{eq:LS29a-1a}. 
Let $n \in \BZ_{\ge 1}$ be such that $n > m$ and $v \in S_{n}$, 
and let $\lng \in S_{n}$ be the longest element. 
Then, by multiplying the directed path $\bp$ 
by $w_{\circ}$ on the left, we obtain a directed path 
\begin{equation*}
(\lng \ed(\bp)\,;\,(j,m),(i,m),(i,l)),
\end{equation*}
which contradicts Lemma~\ref{lem:LS29}.
This proves the lemma. 
\end{proof}

The following lemma is used in Section~\ref{subsec:prfmat2-4} 
(in order to obtain the inequality $j_{t(\bp)} > c_{1}$).
%
%
\begin{lem} \label{lem:LS29b}
There does not exist a directed path of the form\,{\rm:} 
\begin{equation} \label{eq:LS29b}
\bigl( v\,;\,(a,k-1),(b_{1},k-1),\dots,(b_{s},k-1),(a_{1},k),\dots,(a_{t},k),(a,k),(b,k) \bigr)
\end{equation}
in $\QBG(S_{\infty})$ for any $v \in S_{\infty}$, $s,t \ge 0$, $1 \le a < b \le k-1$, and 
$1 \le a_{1},\dots,a_{t},b_{1},\dots,b_{s} \le k-1$
such that 
$a,a_{1},\dots,a_{t},b_{1},\dots,b_{s}$ are all distinct, 
and $b \notin \{a_{1},\dots,a_{t}\}$. 
\end{lem}

\begin{proof}
Suppose, for a contradiction, that 
there exists a directed path of the form \eqref{eq:LS29b}; 
we take a shortest one, say $\bp$, among them. 
By Lemma~\ref{lem:LS29a}, we have $s+t \ge 1$. 
Also, by Lemma~\ref{lem:int}\,(1), we see that 
\begin{equation*}
\bigl( v'\,;\,(a,k-1),(b_{1},k-1),\dots,(b_{s},k-1),(a,k),(b,k) \bigr), 
\end{equation*}
with $v'=v \cdot (a_{1},k) \cdots (a_{t},k)$, is a directed path. 
Hence we deduce that $t=0$ (and so $s \ge 1$) by the shortestness of $\bp$. 
If $b \notin \{b_{1},\dots,b_{s}\}$, then 
we see by Lemma~\ref{lem:int}\,(1) that 
\begin{equation*}
\bigl( v\,;\,(a,k-1),(a,k),(b,k),(b_{1},k-1),\dots,(b_{s},k-1) \bigr)
\end{equation*}
is a directed path, and hence so is $(v\,;\,(a,k-1),(a,k),(b,k))$. 
However, this contradicts Lemma~\ref{lem:LS29a}. 
Therefore, it follows that $b \in \{b_{1},\dots,b_{s}\}$. 
By the same argument as above, we obtain $b_{s}=b$. 
Thus, $\bp$ is of the form: 
\begin{equation*}
\bp = 
\bigl( v\,;\,(a,k-1),(b_{1},k-1),\dots,(b_{s-1},k-1),(b,k-1),(a,k),(b,k) \bigr). 
\end{equation*}
Since $b \ne a$, we see by Lemma~\ref{lem:int}\,(1) that 
\begin{equation*}
\bigl( v\,;\,(a,k-1),(b_{1},k-1),\dots,(b_{s-1},k-1),(a,k),(b,k-1),(b,k) \bigr)
\end{equation*}
is a directed path. Also, we see by Lemma~\ref{lem:int}\,(3) that 
\begin{equation*}
\bigl( v\,;\,(a,k-1),(b_{1},k-1),\dots,(b_{s-1},k-1),(a,k),(k-1,k),(b,k-1) \bigr)
\end{equation*}
is a directed path, and hence so is
\begin{equation*}
\bigl( v\,;\,(a,k-1),(b_{1},k-1),\dots,(b_{s-1},k-1),(a,k),(k-1,k) \bigr). 
\end{equation*}
However, this contradicts the shortestness of $\bp$. 
This proves the lemma. 
\end{proof}

The following lemma is used, for example, 
in Section~\ref{subsec:prfmat2-1} (in order to define the bijection $\theta_{1}$).
%
%
\begin{lem} \label{lem:noak}
Let $k \ge 3$. 
There does not exist a directed path of the form\,{\rm:} 
\begin{equation} \label{eq:noak1}
\bp=(v\,;\,(a,k),(b_{1},k),\dots,(b_{s},k),(a,k))
\end{equation}
in $\QBG(S_{\infty})$ for any $v \in S_{\infty}$, $s \ge 0$, and $1 \le a,b_{1},\dots,b_{s} \le k-2$. 
\end{lem}

\begin{proof}
We prove the assertion of the lemma by induction on $s$. 
Since $1 \le a \le k-2$, 
the assertion is obvious if $s=0$. 
Let us prove the assertion for $s=1$. 
Suppose, for a contradiction, that 
$\bp=(v\,;\,(a,k),(b,k),(a,k))$ is a directed path 
for some $v \in S_{\infty}$ and $1 \le a,b \le k-2$; 
it is obvious that $a \ne b$. If $a > b$, then 
it follows from Lemma~\ref{lem:int}\,(2) that 
$(v\,;\,(b,a),(a,k),(a,k))$ is a directed path, 
which contradicts the assumption that $a \le k-2$. 
If $a < b$, then we see by Lemma~\ref{lem:int}\,(2) that 
$(v\,;\,(b,k),(a,b),(a,k))$ is a directed path. 
Hence it follows from Lemma~\ref{lem:int}\,(3) that 
$(v\,;\,(b,k),(b,k),(a,b))$ is a directed path, 
which contradicts the assumption that $b \le k-2$.
This proves the assertion for $s=1$. 

Let us assume that $s \ge 2$. Suppose, for a contradiction, that 
there exists a directed path $\bp$ of the form \eqref{eq:noak1}, and 
take a shortest one among them; by the shortestness, 
we see that $a, b_{1},\dots,b_{s}$ are all distinct. 
If $b_{1} > b_{2}$, then 
it follows from Lemma~\ref{lem:int}\,(2), 
applied to the segment $(b_{1},k),(b_{2},k)$, that
\begin{equation*}
(v\,;\,(a,k),(b_{2},b_{1}),(b_{1},k),(b_{3},k),\dots,(b_{s},k),(a,k))
\end{equation*}
is a directed path. Since $\{a,k\} \cap \{b_{1},b_{2}\} = \emptyset$, 
we deduce by Lemma~\ref{lem:int}\,(1) that 
\begin{equation*}
(v'\,;\,(a,k),(b_{1},k),(b_{3},k)\dots,(b_{s},k),(a,k)), \qquad 
\text{with $v':=v \cdot (b_{2},b_{1})$}, 
\end{equation*}
is a directed path, which contradicts the shortestness of the directed path $\bp$. 
If $b_{1} < b_{2}$, then we see by Lemma~\ref{lem:int}\,(2) that
\begin{equation*}
(v\,;\,(a,k),(b_{2},k),(b_{1},b_{2}),(b_{3},k),\dots,(b_{s},k),(a,k))
\end{equation*}
is a directed path. Since $a,b_{1},\dots,b_{s},k$ are all distinct, 
we can move $(b_{1},b_{2})$ directly to the right of $(b_{3},k),\dots,(b_{s},k),(a,k)$; 
it follows from Lemma~\ref{lem:int}\,(1) that 
\begin{equation*}
(v\,;\,(a,k),(b_{2},k),(b_{3},k),\dots,(b_{s},k),(a,k),(b_{1},b_{2}))
\end{equation*}
is a directed path. In particular, 
\begin{equation*}
(v\,;\,(a,k),(b_{2},k),(b_{3},k),\dots,(b_{s},k),(a,k))
\end{equation*}
is also a directed path, which contradicts the shortestness of the directed path $\bp$. 
This proves the lemma. 
\end{proof}

The following lemma is used at the end of 
the proof of Lemma~\ref{lem:LS217b} below.
\begin{lem}[{cf. \cite[Lemma~2.17]{LS}}] \label{lem:LS217a}
For any $v \in S_{\infty}$ and $1 \le i < j < k < l < m$, 
there does not exist a directed path of the form\,{\rm:} 
\begin{equation} \label{eq:LS217a}
(v\,;\,(i,m),(j,m),(j,l),(i,k)) 
\end{equation}
in $\QBG(S_{\infty})$.
\end{lem}

\begin{proof}
Suppose, for a contradiction, that 
there exists a directed path $\bp$ of the form \eqref{eq:LS217a}. 
We write $\bp$ as: 
\begin{equation} \label{eq:LS217a1}
\bp:  v = v_{0} \edge{(i,m)} v_{1} \edge{(j,m)} v_{2} \edge{(j,l)} v_{3} \edge{(i,k)} v_{4}. 
\end{equation}
Observe that 
\begin{align*}
& 
v_{1}(j) = v(j), \quad v_{1}(m) = v(i), \quad 
v_{1}(k) = v(k) , \quad v_{1}(l) = v(l), \\[3mm]
& 
v_{2}(j) = v(i), \quad v_{2}(l) = v(l), \quad v_{2}(k) = v(k), \\[3mm]
& 
v_{3}(i) = v(m), \quad v_{3}(k) = v(k), \quad v_{3}(j) = v(l). 
\end{align*}

\paragraph{\bf Case 1.}
%
Assume that the first edge $v_{0} \edge{(i,m)} v_{1}$ in $\bp$ is a Bruhat edge. 
In this case, we have
\begin{equation} \label{eq:217c1}
v(i) < v(m), \qquad v(j), v(k), v(l) \notin [v(i), v(m)]. 
\end{equation}

\paragraph{\bf Subcase 1.1.}
%
Assume that the second edge $v_{1} \edge{(j,m)} v_{2}$ is a Bruhat edge. 
In this case, we have
\begin{equation} \label{eq:217c11}
v(j) < v(i), \qquad v(k), v(l) \notin [v(j), v(i)]. 
\end{equation}

\paragraph{\bf Subsubcase 1.1.1.}
%
Assume that the third edge $v_{2} \edge{(j,l)} v_{3}$ is a Bruhat edge. 
In this case, we have
\begin{equation} \label{eq:217c111}
v(i) < v(l), \qquad v(k) \notin [v(i), v(l)]. 
\end{equation}
From \eqref{eq:217c1}, \eqref{eq:217c11}, and \eqref{eq:217c111}, 
we deduce that $v(j) < v(i) < v(m) < v(l)$, and that
either $v(k) > v(l)$ or $v(k) < v(j)$ holds. 
If $v(k) > v(l)$, then $v(m) < v(k)$. 
Hence the final edge $v_{3} \edge{(i,k)} v_{4}$ is a Bruhat edge. 
However, since $v(l) \in [v(m),v(k)]$, this is a contradiction. 
If $v(k) < v(j)$, then $v(k) < v(m)$. 
Hence the final edge $v_{3} \edge{(i,k)} v_{4}$ is a quantum edge. 
However, since $v(l) \not\in [v(k),v(m)]$, this is a contradiction. 

\paragraph{\bf Subsubcase 1.1.2.}
%
Assume that the third edge $v_{2} \edge{(j,l)} v_{3}$ is a quantum edge. 
In this case, we have
\begin{equation} \label{eq:217c112}
v(i) > v(l), \qquad v(k) \in [v(l), v(i)]. 
\end{equation}
From \eqref{eq:217c1}, \eqref{eq:217c11}, and \eqref{eq:217c112}, 
we deduce that $v(l) < v(k) < v(j) < v(i) < v(m)$, 
which implies that 
the final edge $v_{3} \edge{(i,k)} v_{4}$ is a quantum edge. 
However, since $v(l) \not\in [v(k),v(m)]$, this is a contradiction.

\paragraph{\bf Subcase 1.2.}
%
Assume that the second edge $v_{1} \edge{(j,m)} v_{2}$ is a quantum edge. 
In this case, we have
\begin{equation} \label{eq:217c12}
v(j) > v(i), \qquad v(k), v(l) \in [v(i), v(j)]. 
\end{equation}
Since $v(i) < v(l)$, it follows that 
the third edge $v_{2} \edge{(j,l)} v_{3}$ is a Bruhat edge. 
Hence \eqref{eq:217c111} holds. 
From \eqref{eq:217c1}, \eqref{eq:217c12}, and \eqref{eq:217c111}, 
we deduce that $v(i) < v(m) < v(l) < v(k) < v(j)$. 
Since $v(m) < v(k)$, the final edge $v_{3} \edge{(i,k)} v_{4}$ is a Bruhat edge. 
However, since $v(l) \in [v(m),v(k)]$, this is a contradiction.

\paragraph{\bf Case 2.}
%
Assume that the first edge $v_{0} \edge{(i,m)} v_{1}$ in $\bp$ is a quantum edge. 
In this case, we have
\begin{equation} \label{eq:217c2}
v(i) > v(m), \qquad v(j), v(k), v(l) \in [v(m), v(i)]. 
\end{equation}
Since $v(j) < v(i)$, the second edge $v_{1} \edge{(j,m)} v_{2}$ is a Bruhat edge, 
and hence \eqref{eq:217c11} holds. 
Since $v(i) > v(l)$, the third edge $v_{2} \edge{(j,l)} v_{3}$ is a quantum edge, 
and hence \eqref{eq:217c112} holds. 
From \eqref{eq:217c2}, \eqref{eq:217c11}, and \eqref{eq:217c112}, 
we deduce that $v(m) < v(l) < v(k) < v(j) < v(i)$. 
Since $v(m) < v(k)$, the final edge $v_{3} \edge{(i,k)} v_{4}$ is a Bruhat edge. 
However, since $v(l) \in [v(m),v(k)]$, this is a contradiction. 
This completes the proof of the lemma. 
\end{proof}

The following lemma is used in the definition of deletion procedure 
in Section~\ref{subsec:deletion} below.
%
%
\begin{lem}[{cf. \cite[Lemma~2.17]{LS}}] \label{lem:LS217b}
For any $w \in S_{\infty}$ and $1 \le i, j < k \le l < m$, 
there does not exist an element $\bp \in \SP^{k-1}(w)$ having a segment $\bs$ of the form\,{\rm:} 
\begin{equation} \label{eq:LS217b}
(i,m),\, \dots, \, (j,m), \, \dots, \, (j,l), \, \dots, \,(i,k) 
\end{equation}
in which any label of the form $(i,d)$, 
with $k \le d \le m$, does not appear between $(i,m)$ and $(i,k)$. 
\end{lem}

\begin{proof}
Suppose, for a contradiction, that 
for some $w \in S_{\infty}$ and $1 \le i, j < k \le l < m$, 
there exists an element of $\SP^{k-1}(w)$ having a segment of the form 
\eqref{eq:LS217b}; we take a shortest one, say $\bp$, among them. 
By Lemma~\ref{lem:LS29b}, we see that $i < j$; in particular, $i \le k-2$. 
By the shortestness, $\bp$ is identical to $\bs$, that is, 
\begin{equation*}
\bp = (w\,;\,(i,m),\, \dots, \, (j,m), \, \dots, \, (j,l), \, \dots, \,(i,k)). 
\end{equation*}
We write the segment of the directed path $\bp$ between $(j,l)$ and $(i,k)$ as:
\begin{equation*}
(j,l),(b_{1},c_{1}),\dots,(b_{t},c_{t}),(a_{1},k),\dots,(a_{s},k),(i,k), 
\end{equation*}
with $s, t \ge 0$ and $l \ge c_{1} \ge \cdots \ge c_{t} > k$; 
we set $a_{s+1}:=i$. 
Suppose, for a contradiction, that $s \ge 1$. 
If $a_{u} < a_{u+1}$ for some $1 \le u \le s$, 
then we deduce by Lemma~\ref{lem:int}\,(2), applied to the 
segment $(a_{u},k),(a_{u+1},k)$ in $\bp$, that 
\begin{equation*}
\begin{split}
(w\,;\, & (i,m),\, \dots, \, (j,m), \, \dots, \, (j,l), \, \dots, \, \\
 & \dots, (a_{u-1},k), (a_{u+1},k), (a_{u},a_{u+1}), (a_{u+2},k), \dots, (a_{s+1},k))
\end{split}
\end{equation*}
is a directed path. By moving $(a_{u},a_{u+1})$ to the end of this directed path 
(by Lemma~\ref{lem:int}\,(1)) and then removing it, we deduce that 
\begin{equation*}
\begin{split}
(w\,;\, & (i,m),\, \dots, \, (j,m), \, \dots, \, (j,l), \, \dots, \, \\
 & \dots, (a_{u-1},k), (a_{u+1},k), (a_{u+2},k), \dots, (a_{s+1},k))
\end{split}
\end{equation*}
is also a directed path; 
it is easily seen that this directed path is an element of $\SP^{k-1}(w)$, 
which contradicts the shortestness of $\bp$. 
Thus we get $a_{1} > a_{2} > \cdots > a_{s} > a_{s+1}=i$, 
which implies that $n_{(a_{u},*)}(\bp) = 1$ for all $1 \le u \le s$. 
Hence we can move the segment $(a_{1},k),\dots,(a_{s},k)$ to the beginning of $\bp$, 
and obtain an element of $\SP^{k-1}(w')$, with $w':=w \cdot (a_{1},k) \cdots (a_{s},k)$. 
The resulting directed path has a segment of the form \eqref{eq:LS217b}, and is shorter than $\bp$; 
this contradicts the shortestness of $\bp$. Hence we obtain $s=0$, as desired. 
Next, suppose, for a contradiction, that $t \ge 1$. 
By Lemma~\ref{lem:int}\,(1), together with the fact that 
$i \notin \{b_{1},\dots,b_{t}\}$ and $s = 0$, 
we can move the segment $(b_{1},c_{1}),\dots,(b_{t},c_{t})$ to the end of the directed path $\bp$; 
by removing this segment, we obtain a directed path 
which has a segment of the form \eqref{eq:LS217b}, and is shorter than $\bp$. 
This contradicts the shortestness of $\bp$. Hence we obtain $t=0$, as desired. 
Since $s=t=0$, the label $(i,k)$ is next to $(j,l)$. 
By (P2) for $\bp$ and the fact that $i < j$, we deduce that $l > k$. 

By exactly the same argument as above, we find that 
there exists no label between $(j,m)$ and $(j,l)$ in $\bp$; we write $\bp$ as: 
\begin{equation*}
\bp = (w\,;\, (i,m),(d_{1},m),(d_{2},m),\dots,(d_{r},m), (j,m),(j,l),(i,k)), 
\end{equation*}
with $r \ge 0$. Suppose, for a contradiction, that $r \ge 1$. 
If $i > d_{1}$, then we see by Lemma~\ref{lem:int}\,(2) that 
\begin{equation*}
(w'\,;\, (i,m),(d_{2},m),\dots,(d_{r},m), (j,m),(j,l),(i,k))
\end{equation*}
is an element of $\SP^{k-1}(w')$, with $w':=w \cdot (d_{1},i)$. 
This contradicts the shortestness of $\bp$. 
If $i < d_{1}$, then  we see by Lemma~\ref{lem:int}\,(2) that 
\begin{equation*}
(w\,;\, (d_{1},m),(i,d_{1}),(d_{2},m),\dots,(d_{r},m), (j,m),(j,l),(i,k))
\end{equation*}
is a directed path. 
By using Lemma~\ref{lem:int}\,(1) repeatedly, we deduce that 
\begin{equation*}
(w\,;\, (d_{1},m),(d_{2},m),\dots,(d_{r},m), (j,m),(j,l),(i,d_{1}),(i,k))
\end{equation*}
is a directed path. By Lemma~\ref{lem:int}\,(3), 
\begin{equation*}
(w\,;\, (d_{1},m),(d_{2},m),\dots,(d_{r},m), (j,m),(j,l),(d_{1},k),(i,d_{1}))
\end{equation*}
is a directed path, and hence so is
\begin{equation*}
(w\,;\, (d_{1},m),(d_{2},m),\dots,(d_{r},m), (j,m),(j,l),(d_{1},k)); 
\end{equation*}
note that this directed path is an element of $\SP^{k-1}(w)$ 
having a segment of the form \eqref{eq:LS217b}, with $i$ replaced by $d_{1}$. 
This contradicts the shortestness of $\bp$. Therefore, we conclude that $r = 0$, as desired, 
and hence that $\bp$ is of the form: 
\begin{equation*}
\bp = (w\,;\, (i,m),(j,m),(j,l),(i,k)). 
\end{equation*}
However, since $1 \le i < j < k < l < m$, 
this contradicts Lemma~\ref{lem:LS217a}. 
This proves the lemma. 
\end{proof}

%
\section{Insertion and deletion.}
\label{sec:ins}

We explain two procedures, that is, insertion and deletion, 
which are needed in the proof of Proposition~\ref{prop:mat3}.
%
%
\subsection{Insertion.}
\label{subsec:ins}
Let $w \in S_{\infty}$ and $k \ge 1$. 
Let $\bp= (w\,;\,(a_{1},b_{1}),\dots,(a_{r},b_{r}))$ be a directed path 
in $\QBG(S_{\infty})$ starting at $w$ and satisfying the following conditions:
\begin{enu}
\item[(P0)'] $(a_{i},b_{i}) \in \SL_{k-1} \cup \SL_{k}$ for all $1 \le i \le r$, and 
$n_{(a,b)}(\bp) \in \{0,1\}$ for each $(a,b) \in \SL_{k-1} \cup \SL_{k}$. 
Also, if $n_{(k,*)}(\bp) \ge 1$, then $n_{(*,k)}(\bp) = 0$; 
\item[(P1)'] $b_{1} \ge b_{2} \ge \cdots \ge b_{r}$; 
\item[(P2)'] If $r \ge 3$, and if $a_{j}=a_{i}$ 
for some $1 \le j < i \le r-1$, 
then $(a_{i},b_{i}) \prec (a_{i+1},b_{i+1})$. 
\end{enu}

\noindent
We write $\bp$ as:
\begin{equation*}
\bp=(w\,;\,\dots\dots,\bp_{(*,k+2)},\bp_{(*,k+1)},
\underbrace{(i_{1},k),\dots,(i_{s},k)}_{ \text{$=\bp_{(*,k)}$; possibly, $\emptyset$} }). 
\end{equation*}
Assume that $d \ge k+1$ satisfies the following conditions: 
\begin{enu}

\item[(C1)] 
\begin{equation} \label{eq:Ins0}
(w\,;\,\dots\dots,\bp_{(*,k+2)},\bp_{(*,k+1)},
\underbrace{(i_{1},k),\dots,(i_{s},k)}_{ \text{$=\bp_{(*,k)}$; possibly, $\emptyset$} },(k,d))
\end{equation}
is a directed path; 

\item[(C2)] If $n_{(k,*)}(\bp) \ge 1$, then 
$d < \min \{c \ge k+1 \mid (k,c) \in \bp \}$; 

\item[(C3)] If $n_{(k,*)}(\bp) = 0$ and $s \ge 1$, then 
$(i_{s},l) \notin \bp_{(*,l)}$ for any $k+1 \le l \le d$. 

\end{enu}

\noindent
Now, we define a directed path $\bp \leftarrow (k,d)$ as follows. 
Run {\bf Algorithm $(\bp_{(*,k)}:(k,d))$} for the directed path \eqref{eq:Ins0}; 
this algorithm ends with a directed path $\bp_{1}$ either of the form \eqref{eq:Ia} or of the form \eqref{eq:Ib}: 
\begin{equation} \label{eq:Ia}
(w\,;\,\underbrace{\dots\dots,\bp_{(*,k+2)},\bp_{(*,k+1)}}_{\heartsuit},
(k,d), (i_{1},d),\dots,(i_{s},d)), 
\end{equation}
\begin{equation} \label{eq:Ib}
\begin{split}
(w\,;\, & \dots\dots,\bp_{(*,k+2)},\bp_{(*,k+1)}, \\
& (i_{1},k),\dots,(i_{t-1},k),(i_{t},d),(i_{t},k),(i_{t+1},d),\dots,(i_{s},d)) \quad 
\text{for some $1 \le t \le s$}. 
\end{split}
\end{equation}

\paragraph{\bf Case 1.}
%
If $n_{(k,*)}(\bp) \ge 1$, then we see by (P0)' that $s=0$, 
and hence $\bp_{1}$ is of the form \eqref{eq:Ia}. Also, by (C2), we can move 
$(k,d)$ directly to the right of $\bp_{(*,d)}$ in $\bp_{1}$ as follows: 
\begin{equation*}
(w\,;\,\dots\dots,\bp_{(*,d+1)},\bp_{(*,d)},(k,d),\bp_{(*,d-1)},\dots,\bp_{(*,k+1)}).
\end{equation*}
We call the procedure, which assigns $\bp \leftarrow (k,d)$ to $\bp$, an insertion; 
notice that the resulting directed path $\bp \leftarrow (k,d)$ satisfies (P0)', (P1)', (P2)', 
with $n_{(k,*)}(\bp \leftarrow (k,d)) > 1$. 

\paragraph{\bf Case 2.}
%
Assume next that $n_{(k,*)}(\bp) = 0$, and that
$\bp_{1}$ is of the form \eqref{eq:Ia}. 

\begin{clm} \label{c:ins1}
We have $(i_{u},l) \notin \bp_{(*,l)}$ 
for any $1 \le u \le s$ and $k+1 \le l \le d$.
\end{clm}

\noindent
{\it Proof of Claim~\ref{c:ins1}.} Suppose, for a contradiction, that
there exist $1 \le u \le s$ and $k+1 \le l \le d$ such that 
$(i_{u},l) \in \bp_{(*,l)}$; notice that $1 \le u < s$ by condition (C3). 
Let $(a,l)$ be the rightmost label in the segment $\heartsuit$ in $\bp_{1}$ 
(of the form \eqref{eq:Ia}) such that 
$a \in \{i_{1},\dots,i_{s}\}$; note that $k+1 \le l \le d$ by our assumption. 
Let $1 \le u \le s$ be such that $(a,l) = (i_{u},l)$: 
\begin{equation*}
\bp_{1}=
(w\,;\,\dots\dots,(i_{u},l),
 \underbrace{\dots\dots}_{\diamondsuit},
 (k,d), (i_{1},d),\dots,(i_{u},d),(i_{u+1},d),(i_{u+2},d),\dots,(i_{s},d)), 
\end{equation*}
where in the segment $\diamondsuit$, 
a label of the form $(i_{u},m)$ does not exist 
for any $1 \le u \le s$ and $k+1 \le m \le l$. 
By condition (P2)' on $\bp$, we see that $i_{u} < i_{u+1}$. 
Suppose first that $l < d$. By Lemma~\ref{lem:int}\,(1), 
we deduce that 
\begin{equation*}
(w\,;\,\dots\dots,
 (k,d), (i_{1},d),\dots,(i_{u},l),(i_{u},d),(i_{u+1},d),
 \underbrace{\dots\dots}_{\diamondsuit},(i_{u+2},d),\dots,(i_{s},d))
\end{equation*}
is a directed path, which has the segment 
$(i_{u},l),(i_{u},d),(i_{u+1},d)$. 
Since $l < d$ and $i_{u} < i_{u+1}$, this contradicts Lemma~\ref{lem:LS29a}. 
Suppose next that $l = d$. We write $\bp_{(*,d)}$ in $\bp_{1}$ as:
\begin{equation*}
\begin{split}
(w\,;\,& \dots,\overbrace{(a_{1},d),\dots,(a_{t},d),(i_{u},d),
  (b_{1},d),\dots,(b_{q},d)}^{=\,\bp_{(*,d)}},\bp_{(*,d-1)},\dots\dots, \\
& \bp_{(*,k+1)},(k,d), (i_{1},d),\dots,(i_{u},d),(i_{u+1},d),\dots,(i_{s},d)). 
\end{split}
\end{equation*}
By Lemma~\ref{lem:int}\,(1), we see that 
\begin{equation*}
\begin{split}
(w\,;\,&\dots,\overbrace{(a_{1},d),\dots,(a_{t},d),(i_{u},d),
  (b_{1},d),\dots,(b_{q},d)}^{=\,\bp_{(*,d)}}, \\
& (k,d),(i_{1},d),\dots,(i_{u},d),(i_{u+1},d),\dots,(i_{s},d), 
  \bp_{(*,d-1)}, \dots, \bp_{(*,k+1)})
\end{split}
\end{equation*}
is a directed path. 
Hence it follows from Lemma~\ref{lem:int}\,(2) that 
\begin{equation*}
\begin{split}
(w\,;\, & \dots\dots,(a_{1},d),\dots,(a_{t},d),(k,d),(i_{u},k),
  (b_{1},k),\dots,(b_{q},k), \\
& (i_{1},d),\dots,(i_{u},d),(i_{u+1},d),\dots,(i_{s},d),
  \bp_{(*,d-1)}, \dots, \bp_{(*,k+1)})
\end{split}
\end{equation*}
is a directed path. 
Then, by Lemma~\ref{lem:int}\,(1), 
we deduce that 
\begin{equation*}
\begin{split}
(w\,;\, & \dots\dots,(a_{1},d),\dots,(a_{t},d),(k,d),(i_{1},d),\dots,(i_{u-1},d), \\
& (i_{u},k),(i_{u},d),(i_{u+1},d),
  (b_{1},k),\dots,(b_{q},k),(i_{u+2},d),\dots,(i_{s},d),
  \bp_{(*,d-1)}, \dots, \bp_{(*,k+1)})
\end{split}
\end{equation*}
is a directed path, which has the segment 
$(i_{u},k),(i_{u},d),(i_{u+1},d)$. 
Since $k < d$ and $i_{u} < i_{u+1}$, 
this contradicts Lemma~\ref{lem:LS29a}. 
This proves the claim. \bqed

\medskip

By Lemma~\ref{lem:int}\,(1), together with this claim, 
we can move the segment 
$(k,d), (i_{1},d),\dots,(i_{s},d)$ in $\bp_{1}$ (of the form \eqref{eq:Ia}) 
directly to the right of $\bp_{(*,d)}$ as follows: 
\begin{equation*}
\begin{split}
(w\,;\, & \dots\dots,\bp_{(*,d+1)},\\
& \bp_{(*,d)},(k,d), (i_{1},d),\dots,(i_{s},d),
  \bp_{(*,d-1)},\dots,\bp_{(*,k+1)}).
\end{split}
\end{equation*}
We call the procedure, which assigns $\bp \leftarrow (k,d)$ to $\bp$, an insertion; 
notice that the resulting directed path $\bp \leftarrow (k,d)$ satisfies (P0)', (P1)', (P2)', 
with $n_{(k,*)}(\bp \leftarrow (k,d)) = 1$. 
%
%
\begin{ex} \label{ex:ins2}
Assume that $k = 6$, and let 
\begin{equation*}
\bp := (62173845; (5,7)_{\SB}, (2,7)_{\SB}, (3,7)_{\SB}, (4,6)_{\SB}, (1,6)_{\SB}). 
\end{equation*}
\begin{enu}
\item Assume that $d = 7$. Let us compute $\bp \leftarrow (6,7)$. 
Running {\bf Algorithm $(\bp_{(*,6)}:(6,7))$} for $\bp$, we obtain 
\begin{equation*}
\bp_{1} = ( 62173845 \,;\, \underbrace{(5,7)_{\SB}, (2,7)_{\SB}, (3,7)_{\SB}}_{=\, \bp_{(*,7)}}, 
\underbrace{(6,7)_{\SQ}, (4,7)_{\SB}, (1,7)_{\SB}}_{=: \, \bs}); 
\end{equation*}
in particular, $\bp_{1}$ is of the form \eqref{eq:Ia}. 
Since the segment $\bs$ is to the right of $\bp_{(*,d)} =\bp_{(*,7)}$, 
it follows that $\bp \leftarrow (6,7)$ is identical to 
the directed path above. 

\item Assume that $d = 9$. Let us compute $\bp \leftarrow (6,9)$. 
Running {\bf Algorithm $(\bp_{(*,6)}:(6,9))$} for $\bp$, we obtain 
\begin{equation*}
\bp_{1} = ( 62173845 \,;\, \underbrace{(5,7)_{\SB}, (2,7)_{\SB}, (3,7)_{\SB}}_{=\, \bp_{(*,7)}}, 
\underbrace{(6,9)_{\SQ}, (4,9)_{\SB}, (1,9)_{\SB}}_{=: \, \bs}); 
\end{equation*}
in particular, $\bp_{1}$ is of the form \eqref{eq:Ia}. 
We move the segment $\bs$ directly to the right of $\bp_{(*,d)} = \bp_{(*,9)}=\emptyset$, 
and obtain the directed path
\begin{equation*}
(\bp \leftarrow (6,9)) = ( 62173845 \,;\, 
(6,9)_{\SQ}, (4,9)_{\SB}, (1,9)_{\SB}, (5,7)_{\SB}, (2,7)_{\SB}, (3,7)_{\SB}).
\end{equation*}
\end{enu}
\end{ex}

\paragraph{\bf Case 3.}
%
Assume that $\bp_{1}$ is of the form \eqref{eq:Ib}; 
note that $n_{(k,*)}(\bp) = 0$ in this case. 
By the same argument as for Claim~\ref{c:ins1}, we deduce that 
%
%
\begin{equation} \label{eq:Ibc1}
(i_{u},l) \notin \bp_{(*,l)} \quad 
\text{for any $t \le u \le s$ and $k+1 \le l < d$}.
\end{equation}
By Lemma~\ref{lem:int}\,(1) and \eqref{eq:Ibc1}, 
we can move $(i_{t},d), (i_{t+1},d),\dots,(i_{s},d)$ 
directly to the right of $\bp_{(*,d)}$ as follows: 
\begin{align*}
& (w\,;\,\dots\dots,\bp_{(*,d+1)},
   \overbrace{%
   \bp_{(*,d)},(i_{t},d), (i_{t+1},d),\dots,(i_{s},d)}^{\text{the $(*,d)$-segment of this directed path}},
   \bp_{(*,d-1)},\dots \\ 
& \qquad \dots,\bp_{(*,k+1)},(i_{1},k),\dots,(i_{t-1},k),(i_{t},k)); 
\end{align*}
we call the procedure, which assigns $\bp \leftarrow (k,d)$ to $\bp$, an insertion. 

\begin{clm} \label{c:ins2}
The directed path $\bp':= (\bp \leftarrow (k,d))$ satisfies 
{\rm(P0)'}, {\rm (P1)'}, {\rm (P2)'}, 
with $n_{(k,*)}(\bp') = 0$. 
\end{clm}

\noindent
{\it Proof of Claim~\ref{c:ins2}.} 
It is obvious that $n_{(k,*)}(\bp') = 0$, and 
$\bp'$ satisfies (P0)' and (P1)'. 
Also, if $\bp_{(*,d)} = \emptyset$, then 
it is obvious that $\bp'$ satisfies (P2)'. 
Assume that $\bp_{(*,d)} \ne \emptyset$. 
By Lemma~\ref{lem:noak}, we deduce that 
$(i_{u},d) \notin \bp_{(*,d)}$ for any $p \le u \le s$. 
Let $(i,d)$ be the final label of $\bp_{(*,d)}$, and 
assume that $(i,d)$ is applied to $v \in W$. Then we see that 
\begin{align*}
(v\,;\, & (i,d), (i_{t},d), (i_{t+1},d),\dots,(i_{s},d), \\
& \bp_{(*,d-1)},\dots\dots,\bp_{(*,k+1)},(i_{1},k),\dots,(i_{t-1},k),(i_{t},k))
\end{align*}
is an element of $\SP^{k-1}(v)$. By Lemma~\ref{lem:LS29}\,(2), applied to 
the first, second, and last label of the directed path above, we deduce that 
$i < i_{t}$. Hence we conclude that $\bp'$ satisfies (P2)', 
as desired. This proves the claim. 
\nolinebreak[4] \bqed
%
%
\begin{ex} \label{ex:ins3}
Assume that $k = 4$ and $d = 5$. 
\begin{enu}
\item Let 
$\bp := ( 32514 \,;\, (1,5)_{\SB}, (3,4)_{\SQ}, (1,4)_{\SB}, (2,4)_{\SB})$. 
Let us compute $\bp \leftarrow (4,5)$. 
Running {\bf Algorithm $(\bp_{(*,4)}:(4,5))$} for $\bp$, we obtain 
\begin{equation*}
\bp_{1} = ( 32514 \,;\, (1,5)_{\SB}, (3,4)_{\SQ}, (1,4)_{\SB}, (2,5)_{\SB}, (2,4)_{\SB}); 
\end{equation*}
in particular, $\bp_{1}$ is of the form \eqref{eq:Ib}. 
We move $(2,5)_{\SB}$ to the right of $\bp_{(*,d)} = \bp_{(*,5)}$ as follows: 
\begin{equation*}
(\bp \leftarrow (4,5)) = 
( 32514 \,;\, (1,5)_{\SB}, (2,5)_{\SB}, (3,4)_{\SQ}, (1,4)_{\SB}, (2,4)_{\SB}). 
\end{equation*}

\item  Let 
$\bp := ( 32514 \,;\, (3,4)_{\SQ}, (1,4)_{\SQ}, (2,4)_{\SB})$. 
Let us compute $\bp \leftarrow (4,5)$. 
Running {\bf Algorithm $(\bp_{(*,4)}:(4,5))$} for $\bp$, we obtain 
\begin{equation*}
\bp_{1} = ( 32514 \,;\, (3,4)_{\SQ}, (1,5)_{\SB}, (1,4)_{\SB}, (2,5)_{\SB}); 
\end{equation*}
in particular, $\bp_{1}$ is of the form \eqref{eq:Ib}. 
We move $(1,5)_{\SB}$ and $(2,5)_{\SB}$ to 
the right of $\bp_{(*,d)} = \bp_{(*,5)} = \emptyset$ as follows: 
\begin{equation*}
(\bp \leftarrow (4,5)) = 
( 32514 \,;\, (1,5)_{\SB}, (2,5)_{\SB}, (3,4)_{\SQ}, (1,4)_{\SB}). 
\end{equation*}
\end{enu}
\end{ex}

%
\subsection{Deletion.}
\label{subsec:deletion}

Let $k \ge 1$. 
Let $\bp$ be a directed path starting at $w \in S_{\infty}$ and 
satisfying conditions (P0)', (P1)', and (P2)'. 
In addition, we assume that $\bp$ satisfies the following condition: 
\begin{enu}
\item[(P3)'] 
If $n_{(k,*)}(\bp) = 0$, 
then $\kappa(\bp)=(a,k)$ for some $1 \le a \le k-1$, 
and $n_{(a,*)}(\bp) \ge 2$. 
\end{enu}
Now we define $d(\bp) \ge k+1$, and a directed path 
$\bp \rightarrow (k,d(\bp))$ as follows. 

\paragraph{\bf Case 1.}
%
Assume that $n_{(k,*)}(\bp) \ge 1$; 
recall from (P1)' that $n_{(*,k)}(\bp)=0$ in this case. 
We define
\begin{equation}
d(\bp):=\min \{ d \ge k+1 \mid (k,d) \in \bp \}. 
\end{equation}
We write $\bp_{(*,d(\bp))}$ as: 
\begin{equation*}
(w\,;\, \dots,\bp_{(*,d(\bp)+1)},
\underbrace{\bs,(k,d(\bp)), (i_{1},d(\bp)),\dots,(i_{s},d(\bp))}_{ =\,\bp_{(*,d(\bp))} },
  \bp_{(*,d(\bp)-1)},\dots,\bp_{(*,k+1)}), 
\end{equation*}
with $s \ge 0$. Note that 
$(k,l) \notin \bp_{(*,l)}$ for any $k+1 \le l \le d(\bp)-1$ 
by the definition of $d(\bp)$. 
For each $1 \le u \le s$, since $i_{u} < k$, 
it follows from Lemma~\ref{lem:LS29b} that 
$(i_{u},l) \notin \bp_{(*,l)}$ for any $k+1 \le l \le d(\bp)-1$.
Hence, by Lemma~\ref{lem:int}, 
we can move the segment $(k,d(\bp)), (i_{1},d(\bp)),\dots,(i_{s},d(\bp))$ 
to the end of $\bp$ as follows: 
\begin{equation*}
(w\,;\, \dots,\bp_{(*,d(\bp)+1)},
 \bs, \bp_{(*,d(\bp)-1)},\dots,\bp_{(*,k+1)},(k,d(\bp)), (i_{1},d(\bp)),\dots,(i_{s},d(\bp))). 
\end{equation*}
Then, by Lemma~\ref{lem:int}\,(1), 
\begin{equation*}
(w\,;\, \dots,\bp_{(*,d(\bp)+1)},
 \bs, \bp_{(*,d(\bp)-1)},\dots,\bp_{(*,k+1)},(i_{1},k),\dots,(i_{s},k),(k,d(\bp)))
\end{equation*}
is a directed path. We define a path $\bp \rightarrow (k,d(\bp))$ to be 
the directed path obtained from this directed path by removing the final edge labeled by $(k,d(\bp))$, and call the procedure, which assigns $\bp \rightarrow (k,d(\bp))$ to $\bp$, a deletion;
observe that the resulting directed path $\bp \rightarrow (k,d(\bp))$ satisfies 
(P0)', (P1)', (P2)'. In addition, we see that 
$\bp \rightarrow (k,d(\bp))$ and $(k,d(\bp))$ satisfy (C1), (C2), (C3). 
Thus, the directed path $(\bp \rightarrow (k,d(\bp))) \leftarrow (k,d(\bp))$ is defined;
we deduce that a directed path of the form \eqref{eq:Ia} appears in the procedure, and that the resulting directed path is 
identical to $\bp$. Conversely, assume that $\bp$ and $(k,d)$ satisfy 
(P0)', (P1)', (P2)' and (C1), (C2), (C3), and that a directed path of the form \eqref{eq:Ia} appears 
in the insertion procedure for obtaining $\bp \leftarrow (k,d)$. We see that the resulting directed path 
$\bp \leftarrow (k,d)$ satisfies (P0)', (P1)', (P2)', (P3)', 
and that $n_{(k,*)}(\bp \leftarrow (k,d)) = 1$. 
Also, it is easily verified that 
$d(\bp \leftarrow (k,d)) = d$, and 
$((\bp \leftarrow (k,d)) \rightarrow (k,d)) = \bp$. 
%
%
\begin{ex} \label{ex:del1}
Assume that $k=6$, $d=9$, and let
\begin{equation*}
\bp':=( 62173845 \,;\, 
(6,9)_{\SQ}, (4,9)_{\SB}, (1,9)_{\SB}, (5,7)_{\SB}, (2,7)_{\SB}, (3,7)_{\SB}),
\end{equation*}
which is the directed path $\bp \leftarrow (6,9)$ in Example~\ref{ex:ins2}\,(2). 
Let us compute $\bp' \rightarrow (6,9)$. 
We have $d(\bp')=9$, and $s = 2$, $i_{1}=4$, $i_{2}=5$. 
By using Lemma~\ref{lem:int}, 
we move the segment $(6,9)_{\SQ}, (4,9)_{\SB}, (1,9)_{\SB}$ to the end of $\bp$ as follows:
\begin{equation*}
( 62173845 \,;\, 
(5,7)_{\SB}, (2,7)_{\SB}, (3,7)_{\SB}, (6,9)_{\SQ}, (4,9)_{\SB}, (1,9)_{\SB}). 
\end{equation*}
Then, by using Lemma~\ref{lem:int}\,(1), we obtain the following directed path: 
\begin{equation*}
( 62173845 \,;\, 
(5,7)_{\SB}, (2,7)_{\SB}, (3,7)_{\SB}, (4,6)_{\SB}, (1,6)_{\SB}, (6,9)_{\SQ}). 
\end{equation*}
By removing $(6,9)_{\SQ}$, we get
\begin{equation*}
(\bp' \rightarrow (6,9)) = ( 62173845 \,;\, 
(5,7)_{\SB}, (2,7)_{\SB}, (3,7)_{\SB}, (4,6)_{\SB}, (1,6)_{\SB}), 
\end{equation*}
which is the directed path $\bp$ in Example~\ref{ex:ins2}. 
\end{ex}

\paragraph{\bf Case 2.}
%
Assume that $n_{(k,*)}(\bp) = 0$; 
recall from (P3)' that 
$\kappa(\bp)=(a,k)$ for some $1 \le a \le k-1$, 
and that $n_{(a,*)}(\bp) \ge 2$ in this case. We define
\begin{equation}
d(\bp):=\min \{ d \ge k+1 \mid (a,d) \in \bp \}. 
\end{equation}
We write $\bp$ as: 
\begin{equation*}
\bp = (w\,;\, \dots, 
 \underbrace{\bs,(a,d(\bp)),(j_{1},d(\bp)),\dots,(j_{t},d(\bp))}_{ =\,\bp_{(*,d(\bp))} },
 \dots\dots,\underbrace{(i_{1},k),\dots,(i_{s},k),(a,k)}_{ =\,\bp_{(*,k)} }),
\end{equation*}
where $s, t \ge 0$. It follows from Lemma~\ref{lem:LS217b} that 
$(j_{u},d) \notin \bp_{(*,d)}$ for any $1 \le u \le t$ and 
$k \le d < d(\bp)$. Hence, by Lemma~\ref{lem:int}, 
we can move the segment $(a,d(\bp)),(j_{1},d(\bp)),\dots,(j_{t},d(\bp))$ as follows: 
\begin{equation*}
(w \,;\, \dots, \bs, \dots\dots, (i_{1},k),\dots,(i_{s},k), 
(a,d(\bp)),(a,k),(j_{1},d(\bp)),\dots,(j_{t},d(\bp))). 
\end{equation*}
Then, by using Lemma~\ref{lem:int}\,(3) and (2), 
we deduce that 
\begin{equation*}
(w\,;\, \dots, \bs, \dots\dots, 
  \underbrace{(i_{1},k),\dots,(i_{s},k),(a,k),}_{=\,\bp_{(*,k)}} \\
 (j_{1},k),\dots,(j_{t},k),\,(k,d(\bp)))
\end{equation*}
is a directed path. Now we define a path $\bp \rightarrow (k,d(\bp))$ to be 
the directed path obtained from this directed path by removing the final edge 
labeled by $(k,d(\bp))$, and call the procedure, which assigns $\bp \rightarrow (k,d(\bp))$ to $\bp$, a deletion.
As in Case 1, we deduce that 
$((\bp \rightarrow (k,d(\bp))) \leftarrow (k,d(\bp)))=\bp$ and 
$((\bp \leftarrow (k,d)) \rightarrow (k,d)) = \bp$. 
%
%
\begin{ex} \label{ex:del2}
Assume that $k=4$, $d = 5$, and let 
\begin{equation*}
\bp' : = 
( 32514 \,;\, (1,5)_{\SB}, (2,5)_{\SB}, (3,4)_{\SQ}, (1,4)_{\SB}, (2,4)_{\SB}), 
\end{equation*}
which is the directed path $\bp \leftarrow (4,5)$ in Example~\ref{ex:ins3}\,(1). 
Let us compute $\bp' \rightarrow (4,5)$. We have $a=2$, $d(\bp') = 5$, and $t=0$. 
By using Lemma~\ref{lem:int}, we move $(2,5)_{\SB}$ as follows: 
\begin{equation*}
( 32514 \,;\, (1,5)_{\SB}, (3,4)_{\SQ}, (1,4)_{\SB}, (2,5)_{\SB}, (2,4)_{\SB}). 
\end{equation*}
Using Lemma~\ref{lem:int}\,(3) (in this case, we do not need to use Lemma~\ref{lem:int}\,(2)), 
we obtain the following directed path: 
\begin{equation*}
( 32514 \,;\, (1,5)_{\SB}, (3,4)_{\SQ}, (1,4)_{\SB}, (2,4)_{\SB}, (4,5)_{\SB}). 
\end{equation*}
By removing $(4,5)_{\SB}$, we get
\begin{equation*}
(\bp' \rightarrow (4,5)) = 
( 32514 \,;\, (1,5)_{\SB}, (3,4)_{\SQ}, (1,4)_{\SB}, (2,4)_{\SB}), 
\end{equation*}
which is the directed path $\bp$ in Example~\ref{ex:ins3}\,(1). 
\end{ex}

%
\section{Examples.}
\label{sec:example}

Recall the notation from Example~\ref{ex:algo}.

\begin{ex}[{cf. \cite[Example~7.4]{LM}}] \label{ex1}
Let us compute $\FG_{321}^{\Q}\FG_{231}^{\Q} = \FG_{321}^{\Q}\G{2}{2}$ 
by using Theorem~\ref{thm:pieri}. We can verify that 
the set $\SP^{2}(w)$ for $w = 321$ consists of the following 14 elements:
\begin{equation*}
\begin{array}{c|c|c}
\bp & \Mark_{2}(\bp) & \ed(\bp) \\ \hline\hline
(w \,;\, \emptyset) & \emptyset & 321 \\ \hline
(w \,;\, (1,4)_{\SB}) & \emptyset & 4213 \\ \hline
(w \,;\, (1,4)_{\SB}, (2,4)_{\SB}) & \{(1,4),(2,4)\} & 4312 \\ \hline
(w \,;\, (1,4)_{\SB}, (2,4)_{\SB}, (1,3)_{\SQ}) & \{(1,4),(2,4)\} & 1342 \\ \hline
(w \,;\, (1,4)_{\SB}, (2,4)_{\SB}, (1,3)_{\SQ}, (2,3)_{\SB}) & \{(1,4),(2,4)\} & 1432 \\ \hline
(w \,;\, (1,4)_{\SB}, (2,4)_{\SB}, (2,3)_{\SQ}) & \{(1,4),(2,4)\} & 4132 \\ \hline
(w \,;\, (1,4)_{\SB}, (1,3)_{\SQ}) & \emptyset & 1243 \\ \hline
(w \,;\, (1,4)_{\SB}, (1,3)_{\SQ}, (2,3)_{\SB}) & \{(1,4),(2,3)\} & 1423 \\ \hline
(w \,;\, (1,4)_{\SB}, (2,3)_{\SQ}) & \{(1,4),(2,3)\} & 4123 \\ \hline
(w \,;\, (2,4)_{\SB}) & \emptyset & 3412 \\ \hline
(w \,;\, (2,4)_{\SB},(2,3)_{\SQ}) & \emptyset & 3142 \\ \hline
(w \,;\, (1,3)_{\SQ}) & \emptyset & e \\ \hline
(w \,;\, (1,3)_{\SQ}, (2,3)_{\SB}) & \{(1,3),(2,3)\} & 132 \\ \hline
(w \,;\, (2,3)_{\SQ}) & \emptyset & 312 
\end{array}
\end{equation*}
Therefore, $\sfp{2}{2}$ (and $\hsp{2}{2}$) consists of $7$ elements, and so it follows that
\begin{align*}
\FG_{321}^{\Q}\FG_{231}^{\Q} = \,& 
\FG_{4312}^{\Q}-\Q_1\Q_2\FG_{1342}^{\Q}+\Q_1\Q_2\FG_{1432}^{\Q}-\Q_{2}\FG_{4132}^{\Q} \\
& - \Q_1\Q_2\FG_{1423}^{\Q} + \Q_2\FG_{4123}^{\Q} + \Q_1\Q_2\FG_{132}^{\Q}. 
\end{align*}
\end{ex}

\begin{ex} \label{ex2}
Let us compute $\FG_{32514}^{\Q}\FG_{1342}^{\Q} = \FG_{32514}^{\Q}\G{3}{2}$ 
by using Theorem~\ref{thm:pieri}. We can verify that 
the set $\SP^{3}(w)$ for $w = 32514$ consists of the following 26 elements:
\begin{equation*}
\begin{array}{c|c|c}
\bp & \Mark_{2}(\bp) & \ed(\bp) \\ \hline\hline
(w \,;\, \emptyset) & \emptyset & 32514 \\ \hline
(w \,;\, (3,6)_{\SB}) & \emptyset & 326145 \\ \hline
(w \,;\, (3,6)_{\SB}, (1,5)_{\SB}) & \{(3,6),(1,5)\} & 426135 \\ \hline
(w \,;\, (3,6)_{\SB}, (1,5)_{\SB}, (2,5)_{\SB}) & \{(3,6),(1,5)\},\{(3,6),(2,5)\} & 436125 \\ \hline
(w \,;\, (3,6)_{\SB}, (1,5)_{\SB}, (2,5)_{\SB}, (3,4)_{\SQ}) & \{(3,6),(1,5)\},\{(3,6),(2,5)\} & 431625 \\ \hline
(w \,;\, (3,6)_{\SB}, (1,5)_{\SB}, (3,4)_{\SQ}) & \{(3,6),(1,5)\} & 421635 \\ \hline
(w \,;\, (3,6)_{\SB}, (2,5)_{\SB}) & \{(3,6),(2,5)\} & 346125 \\ \hline
(w \,;\, (3,6)_{\SB}, (2,5)_{\SB}, (3,4)_{\SQ}) & \{(3,6),(2,5)\} & 341625 \\ \hline
(w \,;\, (3,6)_{\SB}, (3,4)_{\SQ}) & \emptyset & 321645 \\ \hline
(w \,;\, (1,5)_{\SB}) & \emptyset & 42513 \\ \hline
(w \,;\, (1,5)_{\SB}, (2,5)_{\SB}) & \{(1,5),(2,5)\} & 43512 \\ \hline
(w \,;\, (1,5)_{\SB}, (2,5)_{\SB}, (3,4)_{\SQ}) & \{(1,5),(2,5)\},\{(1,5),(3,4)\} & 43152 \\ \hline
(w \,;\, (1,5)_{\SB}, (2,5)_{\SB}, (3,4)_{\SQ}, (1,4)_{\SB}) & \{(1,5),(3,4)\} & 53142 \\ \hline
(w \,;\, (1,5)_{\SB}, (2,5)_{\SB}, (3,4)_{\SQ}, (1,4)_{\SB},(2,4)_{\SB}) & \{(1,5),(3,4)\} & 54132 \\ \hline
(w \,;\, (1,5)_{\SB}, (2,5)_{\SB}, (3,4)_{\SQ}, (2,4)_{\SB}) & \{(1,5),(3,4)\} & 45132 \\ \hline
(w \,;\, (1,5)_{\SB}, (3,4)_{\SQ}) & \{(1,5),(3,4)\} & 42153 \\ \hline
(w \,;\, (1,5)_{\SB}, (3,4)_{\SQ}, (1,4)_{\SB}) & \{(1,5),(3,4)\} & 52143 \\ \hline
(w \,;\, (1,5)_{\SB}, (3,4)_{\SQ}, (1,4)_{\SB}, (2,4)_{\SB}) & \{(1,5),(3,4)\} & 54123 \\ \hline
(w \,;\, (1,5)_{\SB}, (3,4)_{\SQ}, (2,4)_{\SB}) & \{(1,5),(3,4)\} & 45123 \\ \hline
(w \,;\, (2,5)_{\SB}) & \emptyset & 34512 \\ \hline
(w \,;\, (2,5)_{\SB}, (3,4)_{\SQ}) & \{(2,5),(3,4)\} & 34152 \\ \hline
(w \,;\, (2,5)_{\SB}, (3,4)_{\SQ}, (2,4)_{\SB}) & \{(2,5),(3,4)\} & 35142 \\ \hline
(w \,;\, (3,4)_{\SQ}) & \emptyset & 32154 \\ \hline
(w \,;\, (3,4)_{\SQ}, (1,4)_{\SB}) & \{(3,4),(1,4)\} & 52134 \\ \hline
(w \,;\, (3,4)_{\SQ}, (1,4)_{\SB}, (2,4)_{\SB}) & \{(3,4),(1,4)\} & 53124 \\ \hline
(w \,;\, (3,4)_{\SQ}, (2,4)_{\SB}) & \{(3,4),(2,4)\} & 35124
\end{array}
\end{equation*}
Therefore, $\sfp{3}{2}$ consists of $20$ elements, and so it follows that
\begin{align*}
\FG_{32514}^{\Q}\FG_{1342}^{\Q} = \,& 
\FG_{426135}^{\Q} - 2 \FG_{436125}^{\Q} + 
2 \Q_{3} \FG_{431625}^{\Q}  - 
\Q_{3} \FG_{421635}^{\Q}  + 
\FG_{346125}^{\Q}  - 
\Q_{3} \FG_{341625}^{\Q} \\
& + \FG_{43512}^{\Q}  
  - 2 \Q_{3} \FG_{43152}^{\Q}
  + \Q_{3} \FG_{53142}^{\Q}
  - \Q_{3} \FG_{54132}^{\Q}
  + \Q_{3} \FG_{45132}^{\Q} \\
& + \Q_{3} \FG_{42153}^{\Q}
  - \Q_{3} \FG_{52143}^{\Q}
  + \Q_{3} \FG_{54123}^{\Q}
  - \Q_{3} \FG_{45123}^{\Q} \\
& + \Q_{3} \FG_{34152}^{\Q}
  - \Q_{3} \FG_{35142}^{\Q}
  + \Q_{3} \FG_{52134}^{\Q}
  - \Q_{3} \FG_{53124}^{\Q}
  + \Q_{3} \FG_{35124}^{\Q}. 
\end{align*}
\end{ex}

%

\end{document}